\documentclass[9pt]{article}
\usepackage{mathrsfs}
\usepackage{amsthm}
\usepackage{amssymb}
\usepackage{amsmath}
\usepackage{graphicx}
\usepackage{color}
\usepackage{amsfonts}
\usepackage{float}
\usepackage{cite}
\usepackage [latin1]{inputenc}
\usepackage[text={140mm,210mm},left=35mm,vmarginratio=1:1]{geometry}
\newtheorem{theorem}{Theorem}[section]

\newtheorem{lemma}[theorem]{Lemma}
\newtheorem{proposition}[theorem]{Proposition}

\numberwithin{equation}{section}
\normalsize

\begin{document}
\title{\textbf{CLT for quadratic additive functionals of three-dimensional and four-dimensional voter models}}

\author{Xiaofeng Xue \thanks{\textbf{E-mail}: xfxue@bjtu.edu.cn \textbf{Address}: School of Mathematics and Statistics, Beijing Jiaotong University, Beijing 100044, China.}\\ Beijing Jiaotong University}

\date{}
\maketitle

\noindent {\bf Abstract:} In this paper, we prove central limit theorems for the quadratic additive functionals of the three-dimensional and four-dimensional voter models. Our main results show that, under proper time-space scalings, the centered quadratic additive functional of the four-dimensional voter model converges weakly to a Brownian motion and that in the three-dimensional case converges weakly to a Gaussian process without independent increment. The proof of the main result in the four-dimensional case extends the martingale decomposition method. In the three-dimensional case, we introduce a quadratic fluctuation limit theorem which shows that the quadratic fluctuation density field of the voter model converges weakly to a generalized Ornstein-Uhlenbeck process. Then our main result in the three-dimensional case follows from an argument where the quadratic additive functional is approximated by an integral of a test function with respect to the quadratic fluctuation density field. The duality relationship between the voter model and the coalescing random walk is repeatedly utilized in proofs of the main results. The proof strategies introduced in this paper are applicable in deducing CLTs for general additive functionals of three-dimensional and four dimensional voter models. We only deal with those with the quadratic forms to make our exposition easy to catch.

\quad

\noindent {\bf Keywords:} voter model, quadratic additive functional, quadratic fluctuation limit, coalescing random walk.

%\noindent {\bf MSC2020:} 60F05, 60K35.

\section{Introduction}\label{section one}
In this paper, we are concerned with the voter models on three-dimensional and four-dimensional lattices. We prove central limit theorems for the so-called quadratic additive functionals of the processes. For later use, we first introduce some notations and recall the definition of the voter model. For integer $d\geq 1$, we denote by $\mathbb{Z}^d$ the $d$-dimensional lattice. For any $x\in \mathbb{Z}^d$, we denote by $\|x\|_1$ the $l_1$-norm of $x$. For any $x, y\in \mathbb{Z}^d$, we write $x\sim y$ when and only when $\|x-y\|_1=1$, i.e., $x\sim y$ means that $y$ is one of $2d$ neighbors of $x$. We denote by $\textbf{0}$ the origin of $\mathbb{Z}^d$. Let $\{0, 1\}^{\mathbb{Z}^d}$ be the set of all configurations $\eta: \mathbb{Z}^d\rightarrow \{0, 1\}$. For any different $x, y\in \mathbb{Z}^d$ and $\eta\in \{0, 1\}^{\mathbb{Z}^d}$, we define $\eta^{x, y}\in \{0, 1\}^{\mathbb{Z}^d}$ as
\[
\eta^{x, y}(z)=
\begin{cases}
\eta(z) & \text{~if~}z\neq x,\\
\eta(y) & \text{~if~}z=x.
\end{cases}
\]
The voter model $\{\eta_t\}_{t\geq 0}$ on $\mathbb{Z}^d$ is a continuous-time Markov process with state space $\{0, 1\}^{\mathbb{Z}^d}$ and generator $\mathcal{L}$ given by
\begin{equation}\label{equ 1.1 generator of the voter model}
\mathcal{L}f(\eta)=\sum_{x\in \mathbb{Z}^d}\sum_{y\sim x}\left(f(\eta^{x, y})-f(\eta)\right)
\end{equation}
for any $\eta\in \{0, 1\}^{\mathbb{Z}^d}$ and local $f: \{0, 1\}^{\mathbb{Z}^d}\rightarrow \mathbb{R}$.

The voter model describes the variations of respective sets of supporters of two opposite opinions of a topic. In detail, for any $x\in \mathbb{Z}^d$, $\eta(x)=1$ (resp. $\eta(x)=0$) means that $x$ holds the opinion `$1$' (resp. `$0$'). For any $x\in \mathbb{Z}^d$ and any neighbor $y$ of $x$, $x$ adopts $y$'s opinion at rate $1$. For a detailed survey of basic properties of the voter model, see Chapter 5 of \cite{Lig1985} and Part {\rm \uppercase\expandafter{\romannumeral2}} of \cite{Lig1999}.

We denote by $\mathbb{P}$ the probability measure of $\{\eta_t\}_{t\geq 0}$ and denote by $\mathbb{E}$ the expectation with respect to $\mathbb{P}$. When $\eta_0$ is distributed with a probability measure $\lambda$ on $\{0, 1\}^{\mathbb{Z}^d}$, we write $\mathbb{P}$ as $\mathbb{P}_\lambda$ to emphasize the dependence of $\mathbb{P}$ on $\lambda$. For given $\eta\in \{0, 1\}^{\mathbb{Z}^d}$, when $\lambda$ is the Dirac measure concentrated on $\eta$, we further write $\mathbb{P}_\lambda$ as $\mathbb{P}_\eta$.

For any local $h: \{0, 1\}^{\mathbb{Z}^d}\rightarrow \mathbb{R}$, the random variable $\int_0^t h(\eta_s)ds$ is called an additive functional of the process. Hence, for any $x, z\in \mathbb{Z}^d$ and $t\geq 0$, it is natural to call the stochastic variable
\[
\xi_t^{x, z}:=\int_0^t\eta_s(x)\eta_s(z)ds
\]
a quadratic additive functional of the voter model on $x, z$. When $x=z$, the quadratic additive functional reduces to the so-called occupation time $\int_0^t \eta_s(x)ds$. The main results of this paper give the central limit theorems of $\{\xi_t^{x, z}\}_{t\geq 0}$ under proper time-space scaling in the cases $d=3$ and $d=4$. For mathematical details, see Section \ref{section two}.

The investigations of limit theorems for additive functionals of voter models date back to 1980s. Reference \cite{Cox1983} gives central limit theorems for the occupation times of the voter models on lattices starting from a Bernoulli product measure, which perform dimension-dependent phase transitions. Reference \cite{Xue2026} extends the main results in \cite{Cox1983} to invariance principles in cases $d\geq 3$. Reference \cite{Xue2026b} gives invariance principles for general additive functionals of a family of voter models, which include the voter model on $\mathbb{Z}^d$ for $d\geq 5$ as a special case. Reference \cite{Bramson1988} gives large deviation principles for the occupation times of the voter models on $\mathbb{Z}^d$ for $d\geq 3$. Reference \cite{Maillard2009} extends the main result in \cite{Bramson1988} to the case $d=2$.

\section{Main results}\label{section two}
In this section, we give our main results. We first introduce some notations for later use. We denote by $\{X_t\}_{t\geq 0}$ the continuous-time simple random walk on $\mathbb{Z}^d$ with generator $\mathcal{G}$ given by
\begin{equation}\label{equ 2.1 generator of Random Walk}
\mathcal{G}h(x)=\sum_{y\sim x}(h(y)-h(x))
\end{equation}
for any $h: \mathbb{Z}^d\rightarrow \mathbb{R}$ and $x\in \mathbb{Z}^d$. We denote by $\hat{\mathbb{P}}$ the probability measure of $\{X_t\}_{t\geq 0}$ and denote by $\hat{\mathbb{E}}$ the expectation with respect to $\hat{\mathbb{P}}$. For any $x, z\in \mathbb{Z}^d$ and $t\geq 0$, we denote by $p_t(x, z)$ the transition probability $\hat{\mathbb{P}}\left(X_t=z\big|X_0=x\right)$. We write $p_t(x, z)$ as $p_{t, d}(x, z)$ when we need to emphasize the dependence of $p_t(x, z)$ on the dimension $d$ of the lattice. For any $x\in \mathbb{Z}^d$, we denote by $\Phi(x)$ the term
\[
\hat{\mathbb{P}}\left(X_t=x\text{~for some~}t\geq 0\big|X_0=\textbf{0}\right).
\]
For $x\sim \textbf{0}$, we denote by $\gamma$ the term $\Phi(x)$. According to the spatial homogeneity of the simple random walk, conditioned on $X_0=\textbf{0}$, $\gamma$ is probability that $X_\cdot$ will return to $\textbf{0}$ after leaving $\textbf{0}$. We write $\gamma$ as $\gamma_d$ when we need to emphasize the dependence of $\gamma$ on the dimension $d$ of the lattice.

Now we introduce the initial distribution of the voter model discussed in this paper. For given $p\in (0, 1)$, we denote by $\mu_p$ the Bernoulli product measure on $\{0, 1\}^{\mathbb{Z}^d}$ with density $p$. i.e., $\{\eta(x)\}_{x\in \mathbb{Z}^d}$ are independent under $\mu_p$ and
\[
\mu_p\left(\eta:~\eta(x)=1\right)=p
\]
for each $x\in \mathbb{Z}^d$. The following proposition, which is proved in \cite{Clifford1973} and \cite{Holley1975} independently, shows that the voter model $\{\eta_t\}_{t\geq 0}$ starting from $\mu_p$ converges weakly to an invariant distribution as $t\rightarrow+\infty$.

\begin{proposition}\label{Proposition 2.1 invariant distribution}(Clifford and Sudbury, \cite{Clifford1973}; Holley and Liggett, \cite{Holley1975})
Let $p\in (0, 1)$ and $\eta_0$ be distributed with $\mu_p$. There exists a probability measure $\nu_p$ on $\{0, 1\}^{\mathbb{Z}^d}$ such that $\eta_t$ converges weakly to $\nu_p$ as $t\rightarrow+\infty$. Furthermore,
\[
\nu_p\left(\eta:~\eta(x)=1\right)=p
\]
and
\[
\nu_p\left(\eta:~\eta(x)=\eta(z)=1\right)=p\Phi(x-z)+p^2\left(1-\Phi(x-z)\right)
\]
for any $x, z\in \mathbb{Z}^d$.
\end{proposition}
Proposition \ref{Proposition 2.1 invariant distribution} is an application of the duality relationship between the voter model and the coalescing random walk, which we recall in Appendix \ref{Appendix A.1}.

According to Proposition \ref{Proposition 2.1 invariant distribution}, $\nu_p$ is an invariant distribution of $\{\eta_t\}_{t\geq 0}$.
According to the transience of the simple random walk on $\mathbb{Z}^d$ for $d\geq 3$, $\Phi(x)<1$ when $x\neq \textbf{0}$ and hence $\nu_p$ does not reduce to a convex combination of $\varpi_1$ and $\varpi_0$ when $d\geq 3$, where $\varpi_a$ is the Dirac measure concentrated on the configuration where all the vertices hold the opinion $a$.

In this paper, we focus on the cases $d=3, 4$, which are not covered by the models discussed in \cite{Xue2026b}. From now on, we assume that $p\in (0, 1)$ is a given constant and the initial state $\eta_0$ of the voter model is distributed with the invariant measure $\nu_p$.

We further introduce some notations and definitions to give our main result. For $p\in (0, 1)$, we define
\[
\chi(p)=2p(1-p)(1-\gamma).
\]
Note that, by Proposition \ref{Proposition 2.1 invariant distribution},
\begin{equation}\label{equ 2.2 second moment of etax-etay}
\chi(p)=\mathbb{E}_{\nu_p}\left((\eta_t(x)-\eta_t(y))^2\right)
\end{equation}
for any $t\geq 0$ and $x, y\in \mathbb{Z}^d$ such that $x\sim y$. We write $\chi(p)$ as $\chi(p, d)$ when we need to emphasize the dependence of $\chi(p)$ on the dimension $d$ of the lattice.

For any different $x, z\in \mathbb{Z}^d$ and $p\in (0, 1)$, we define
\[
\alpha_{x, z}(p)=p\Phi(x-z)+p^2\left(1-\Phi(x-z)\right).
\]
By Proposition \ref{Proposition 2.1 invariant distribution},
\begin{equation}\label{equ 2.3 expectation of etaxetaz}
\alpha_{x, z}(p)=\mathbb{E}_{\nu_p}\left(\eta_t(x)\eta_t(z)\right)
\end{equation}
for any $t\geq 0$.

As in Section \ref{section one}, for any $x, z\in \mathbb{Z}^d$ and $t\geq 0$, we denote by $\xi_t^{x, z}$ the random variable $\int_0^t\eta_s(x)\eta_s(z)ds$. We denote by $\{B_t\}_{t\geq 0}$ the standard real-valued Brownian motion starting from $0$.

Now we give our main result in the case $d=4$.
\begin{theorem}\label{theorem 2.1 main theorem d=4}
Let $d=4$, $p\in (0, 1)$, $x\in \mathbb{Z}^4$ and $\eta_0$ be distributed with $\nu_p$. For any $T>0$, the random process
\[
\left\{\frac{\xi_{tN}^{\textbf{0}, x}-tN\alpha_{\textbf{0}, x}(p)}{\sqrt{N\log N}}:~0\leq t\leq T\right\}
\]
converges weakly, with respect to the uniform topology of $C[0, T]$, to $\{C_4(x)B_t\}_{0\leq t\leq T}$ as $N\rightarrow+\infty$, where
\[
C_4(x)=\sqrt{\frac{\chi(p, 4)}{2\pi^2}}\left(\Phi(x)+2(1-\Phi(x))p\right).
\]
\end{theorem}
When $x=\textbf{0}$, $\xi_t^{\textbf{0}, x}=\int_0^t \eta_s(\textbf{0})ds$, which is the occupation time of the voter model on $\textbf{0}$. Since $\Phi(\textbf{0})=1$, Theorem \ref{theorem 2.1 main theorem d=4} in the case $x=\textbf{0}$ reduces to the invariance principle for the occupation time of
the voter model on $\mathbb{Z}^4$ given in Theorem 2.1 of \cite{Xue2026}. Note that, although Reference \cite{Xue2026} only deals with that case where $\eta_0$ is distributed with $\mu_p$, the argument given in \cite{Xue2026} to prove the occupation time invariance principle is still applicable when the initial distribution is replaced by $\nu_p$.

The proof of Theorem \ref{theorem 2.1 main theorem d=4} follows from a martingale decomposition strategy, where the main difficulty is the approximate calculation of a quadratic variation process. As far as we understand, the approach in our proof is applicable in deriving the central limit theorem for the centered version of the process $\{\int_0^t\prod_{i=1}^m\eta_s(x_i)ds\}_{t\geq 0}$ for any $m\geq 2$ and $x_1, \ldots, x_m\in \mathbb{Z}^4$. However, since the precise expression of the CLT in the case $m\geq 3$ are very complicated, in this paper we only deal with the case $m=2$ to highlight the main idea of the proof and make our exposition simple. For mathematical details of the proof of Theorem \ref{theorem 2.1 main theorem d=4}, see Section \ref{section three}.

To give our main result in the case $d=3$, we further introduce some notations. For each $d\geq 1$, we denote by $\{\mathcal{W}_t^d\}_{t\geq 0}$ the $\mathbb{R}^d$-valued standard Brownian motion starting from $\textbf{0}$, i.e.,
\[
\mathcal{W}_t^d=\left(B_t^1, B_t^2, \ldots, B_t^d\right),
\]
where $\{B_t^j:~t\geq 0\}_{j=1}^d$ are independent copies of $\{B_t:~t\geq 0\}$. For any $t>0$, we denote by $f_t^d$ the probability density of $\mathcal{W}_t^d$, i.e.,
\[
f_t^d(u)=\left(\frac{1}{\sqrt{2\pi t}}\right)^de^{-\frac{\sum_{j=1}^du_j^2}{2t}}
\]
for any $u=(u_1, \ldots, u_d)\in \mathbb{R}^d$. For any $u\in \mathbb{R}^3$ and $N\geq 1$, we denote by $u_N$ the element
in
\[
\mathbb{Z}^3/\sqrt{N}:=\left\{\frac{x}{\sqrt{N}}:~x\in \mathbb{Z}^3\right\}
\]
such that
\[
u-u_N\in \left[-\frac{1}{2\sqrt{N}}, \frac{1}{2\sqrt{N}}\right)^3.
\]
For any given $u\in \mathbb{R}^3\setminus\{\textbf{0}\}$, we denote by $\{\zeta_t^u\}_{t\geq 0}$ the real-valued Gaussian process with continuous sample path, mean zero and covariance functions given by
\begin{align*}
{\rm Cov}(\zeta_t^u, \zeta_s^u)
=&2\int_s^td\theta_1\int_0^sd\theta_2\int_0^{+\infty}f^3_{2(\theta_1-\theta_2+r)}(\textbf{0})+f^3_{2(\theta_1-\theta_2+r)}(u)dr\\
&+4\int_0^sd\theta_1\int_0^{\theta_1}d\theta_2\int_0^{+\infty}f^3_{2(\theta_2+r)}(\textbf{0})+f^3_{2(\theta_2+r)}(u)dr
\end{align*}
for any $0\leq s\leq t$. Now we give our main result in the case $d=3$.
\begin{theorem}\label{theorem 2.3 main theorem d=3}
Let $d=3$, $p\in (0, 1)$, $u\in \mathbb{R}^3$ such that $u\neq \textbf{0}$ and $\eta_0$ be distributed with $\nu_p$. For any $T>0$, the random process
\[
\left\{\frac{\xi_{tN}^{\textbf{0}, \sqrt{N}u_N}-tN\alpha_{\textbf{0}, \sqrt{N}u_N}(p)}{N^{\frac{3}{4}}}:~0\leq t\leq T\right\}
\]
converges weakly, with respect to the uniform topology of $C[0, T]$, to $\{C_3\zeta_t^u\}_{0\leq t\leq T}$ as $N\rightarrow+\infty$, where
\[
C_3=\sqrt{6(1-\gamma_3)p^3(1-p)}.
\]
\end{theorem}

Note that a main difference between Theorem \ref{theorem 2.1 main theorem d=4} and Theorem \ref{theorem 2.3 main theorem d=3} is that in the case $d=3$ we require that the distance between the two vertices which the quadratic additive functional depends on is under a spatial scaling with order $\sqrt{N}$ to obtain a CLT with a Gaussian limit process. It is natural to ask whether a CLT with a nontrivial limit process holds in the case $d=3$ when $\sqrt{N}u_N$ is replaced by a fixed vertex $x\in \mathbb{Z}^3$, as in the case $d=4$. We will work on this problem as a further investigation.

The proof of Theorem \ref{theorem 2.3 main theorem d=3} follows from a quadratic fluctuation limit strategy. We first show that the quadratic fluctuation density field
\[
\mathcal{Y}_t^N(du, dv):=\frac{1}{N^{11/4}}\sum_{x\in \mathbb{Z}^3}\sum_{z\in \mathbb{Z}^3}\left(\eta_{tN}(x)\eta_{tN}(z)-\alpha_{x, z}(p)\right)\delta_{x/\sqrt{N}}(du)\delta_{z/\sqrt{N}}(dv)
\]
converges weakly to a generalized Ornstein-Uhlenbeck process. Then we introduce a series of functions $H_\epsilon: \mathbb{R}^6\rightarrow \mathbb{R}$ such that $\mathcal{Y}_t^N(H_\epsilon)\approx \frac{\xi_{tN}^{\textbf{0}, \sqrt{N}u_N}-tN\alpha_{\textbf{0}, \sqrt{N}u_N}(p)}{N^{\frac{3}{4}}}$ for large $N$ and small $\epsilon$. Then the CLT of $\left\{\frac{\xi_{tN}^{\textbf{0}, \sqrt{N}u_N}-tN\alpha_{\textbf{0}, \sqrt{N}u_N}(p)}{N^{\frac{3}{4}}}\right\}_{0\leq t\leq T}$ follows from the CLT of $\mathcal{Y}_t^N$ and a verification of the tightness under the uniform topology of $C[0, T]$. As in the case $d=4$, we think that our approach in the proof of Theorem \ref{theorem 2.3 main theorem d=3} is applicable in deriving the CLT for the centered version of the process
\[
\left\{\int_0^{tN}\prod_{i=1}^m\eta_s(\sqrt{N}(u_i)_N)ds:~0\leq t\leq T\right\}_{N\geq 1}
\]
for any $m\geq 2$ and different $u_1, u_2, \ldots, u_m\in \mathbb{R}^3$. However, to highlight the main idea in the proof and make our exposition simple, we still only deal with the case $m=2$. For mathematical details, see Section \ref{section four}.

\section{Proof of Theorem \ref{theorem 2.1 main theorem d=4}}\label{section three}
In this section, we prove Theorem \ref{theorem 2.1 main theorem d=4}. Our proof utilizes a martingale decomposition strategy, where the quadratic additive functional is decomposed as a martingale plus an error term. We show that the error term in the aforesaid decomposition converges weakly to $0$ and then verify that the martingale term converges weakly to $C_4(x)B_t$ by an approximate calculation of the quadratic variation process. Finally, we complete the proof of Theorem \ref{theorem 2.1 main theorem d=4} by verifying the tightness of the quadratic additive functional processes under the uniform topology of $C[0, T]$.

\subsection{Martingale decomposition}\label{subsection 3.1}
In this section we give our martingale decomposition of the quadratic additive functional and show that the martingale term in which converges weakly to $C_4(x)B_t$. For later use, we first introduce some notations. For given $x\in \mathbb{Z}^4$, let $\left\{\left(\mathcal{X}_{t, 1}^{\textbf{0}, 0}, \mathcal{X}_{t, 2}^{x, 0}\right)\right\}_{t\geq 0}$ be the coalescing random walk defined as in Appendix \ref{Appendix A.1}. For $z_1, z_2\in \mathbb{Z}^4$ and $t\geq 0$, we denote by $q_{t}^{\textbf{0}, x}(z_1, z_2)$ the transition probability
\[
\tilde{\mathbb{P}}\left(\left(\mathcal{X}_{t, 1}^{\textbf{0}, 0}, \mathcal{X}_{t, 2}^{x, 0}\right)=(z_1, z_2)\right).
\]
For each integer $N\geq 1$ and $\eta\in \{0, 1\}^{\mathbb{Z}^4}$, we define
\[
\tilde{G}_N(\eta)=\sum_{z_1\in \mathbb{Z}^4}\sum_{z_2\in \mathbb{Z}^4}\eta(z_1)\eta(z_2)\int_0^Nq_s^{\textbf{0}, x}(z_1, z_2)ds.
\]
For $N\geq 1$ and $t\geq 0$, let
\[
\mathcal{M}_t^N=\tilde{G}_N(\eta_t)-\tilde{G}_N(\eta_0)-\int_0^t\mathcal{L}\tilde{G}_N(\eta_s)ds,
\]
where $\mathcal{L}$ is defined as in \eqref{equ 1.1 generator of the voter model}. By Dynkin's martingale formula, $\{\mathcal{M}_t^N\}_{t\geq 0}$ is a martingale.

Let $H: \{0, 1\}^{\mathbb{Z}^4}\rightarrow \mathbb{R}$ be defined as $H(\eta)=\eta(\textbf{0})\eta(x)$ for any $\eta\in \{0, 1\}^{\mathbb{Z}^4}$. We denote by $\{S(t)\}_{t\geq 0}$ the Markov semi-group of $\{\eta_t\}_{t\geq 0}$. By Proposition \ref{Proposition A.1 dual},
\[
S(t)H(\eta)=\mathbb{E}_\eta H(\eta_t)=\sum_{z_1, z_2}\eta(z_1)\eta(z_2)q_t^{\textbf{0}, x}(z_1, z_2).
\]
Hence,
\[
\tilde{G}_N(\eta)=\int_0^NS(r)H(\eta)dr.
\]
By Kolmogorov backward equation, $\mathcal{L}S(r)=\frac{d}{dr}S(r)$ and hence
\begin{align}\label{equ 3.1}
\mathcal{M}_t^N&=\tilde{G}_N(\eta_t)-\tilde{G}_N(\eta_0)-\int_0^t\left(\int_0^N\frac{d}{dr}S(r)H(\eta_s)dr\right)ds\notag\\
&=\tilde{G}_N(\eta_t)-\tilde{G}_N(\eta_0)-\int_0^t\left(S(N)H(\eta_s)-S(0)H(\eta_s)\right)ds\notag\\
&=\tilde{G}_N(\eta_t)-\tilde{G}_N(\eta_0)-\int_0^tS(N)H(\eta_s)ds+\xi_t^{0, x}.
\end{align}
Let
\[
G_N(\eta)=\tilde{G}_N(\eta)-\mathbb{E}_{\nu_p}\tilde{G}_N(\eta_0)=\sum_{z_1}\sum_{z_2}\left(\eta(z_1)\eta(z_2)-\alpha_{z_1, z_2}(p)\right)\int_0^Nq_s^{\textbf{0}, x}(z_1, z_2)ds,
\]
then by \eqref{equ 3.1},
\[
\mathcal{M}_t^N=G_N(\eta_t)-G_N(\eta_0)-\int_0^tS(N)H(\eta_s)-\alpha_{\textbf{0}, x}(p)ds+\xi_t^{0, x}-t\alpha_{\textbf{0}, x}(p)
\]
and hence
\begin{equation}\label{equ 3.2 martingale decomposition d=4}
\frac{\xi_{tN}^{0, x}-tN\alpha_{\textbf{0}, x}(p)}{\sqrt{N\log N}}=\frac{\mathcal{M}_{tN}^N}{\sqrt{N\log N}}+\varepsilon_1^N,
\end{equation}
where
\[
\varepsilon_1^N=\frac{G_N(\eta_{tN})-G_N(\eta_0)-\int_0^{tN}S(N)H(\eta_s)-\alpha_{\textbf{0}, x}(p)ds}{\sqrt{N\log N}}.
\]
We first deal with the term $\varepsilon_1^N$ in decomposition \eqref{equ 3.2 martingale decomposition d=4}.
\begin{lemma}\label{lemma 3.1}
Let $\eta_0$ be distributed with $\nu_p$, then
\[
\lim_{N\rightarrow+\infty}\varepsilon_1^N=0
\]
in $L^2$.
\end{lemma}
\proof[Proof of Lemma \ref{lemma 3.1}]
According to the invariance of $\nu_p$,
\[
\mathbb{E}_{\nu_p}G_N(\eta_t)=\mathbb{E}_{\nu_p}\left(S(N)H(\eta_t)-\alpha_{\textbf{0}, x}(p)\right)=0
\]
for all $t\geq 0$. So, to complete the proof, we only need to verify that
\begin{equation}\label{equ 3.3}
\lim_{N\rightarrow+\infty}\frac{1}{N\log N}{\rm Var}_{\nu_p}\left(G_N(\eta_0)\right)=0
\end{equation}
and
\begin{equation}\label{equ 3.4}
\lim_{N\rightarrow+\infty}\frac{1}{N\log N}{\rm Var}_{\nu_p}\left(\int_0^{tN}S(N)H(\eta_s)-\alpha_{\textbf{0}, x}(p)ds\right)=0.
\end{equation}
We first verify \eqref{equ 3.3}. According to the bilinear property of the covariance,
\begin{align*}
&{\rm Var}_{\nu_p}\left(G_N(\eta_0)\right)\\
&=\sum_{z_1, z_2, v_1, v_2\in \mathbb{Z}^4}{\rm Cov}_{\nu_p}\left(\eta_0(z_1)\eta_0(z_2), \eta_0(v_1)\eta_0(v_2)\right)\int_0^N\int_0^N
q_{s_1}^{\textbf{0}, x}(z_1, z_2)q_{s_2}^{\textbf{0}, x}(v_1, v_2)ds_1ds_2.
\end{align*}
By Lemma 3.3 of \cite{Xue2026b}, we have
\begin{equation}\label{equ 3.5}
\left|{\rm Cov}_{\nu_p}\left(\eta_\theta(z_1)\eta_\theta(z_2), \eta_0(v_1)\eta_0(v_2)\right)\right|
\leq K_1\sum_{i=1}^2\sum_{j=1}^2\int_0^{+\infty}p_{r+\theta}(z_i, v_j)dr
\end{equation}
for any $\theta\geq 0$, where $K_1=(\int_0^{+\infty}p_s(\textbf{0}, \textbf{0})ds)^{-1}$. Especially,
\[
\left|{\rm Cov}_{\nu_p}\left(\eta_0(z_1)\eta_0(z_2), \eta_0(v_1)\eta_0(v_2)\right)\right|
\leq K_1\sum_{i=1}^2\sum_{j=1}^2\int_0^{+\infty}p_{r}(z_i, v_j)dr.
\]
For readers not familiar with \cite{Xue2026b}, we give a verification of \eqref{equ 3.5} in Appendix \ref{Appendix A.2}.

Consequently, according to the fact that
\[
\sum_{z_2}q_{s_1}^{\textbf{0}, x}(z_1, z_2)=p_{s_1}(\textbf{0}, z_1)
\]
and other three analogues and
\[
\sum_{z_1}\sum_{v_1}p_{s_1}(\textbf{0}, z_1)p_r(z_1, v_1)p_{s_2}(\textbf{0}, v_1)=p_{s_1+s_2+r}(\textbf{0}, \textbf{0})
\]
and other three analogues, we have
\begin{align*}
{\rm Var}_{\nu_p}\left(G_N(\eta_0)\right)&\leq K_1\int_0^N\int_0^N\int_0^{+\infty}p_{s_1+s_2+r}(\textbf{0}, \textbf{0})+p_{s_1+s_2+r}(x, x)+2p_{s_1+s_2+r}(\textbf{0}, x)dr\\
&\leq 4K_1\int_0^N\int_0^N\int_0^{+\infty}p_{s_1+s_2+r}(\textbf{0}, \textbf{0})ds_1ds_2dr.
\end{align*}
According to the local central limit theorem of the simple random walk (see Theorem 2.1.3 in Chapter 2 of \cite{Lawler2010}),
\[
p_{t, 4}(\textbf{0}, \textbf{0})=O(1)t^{-2}
\]
as $t\rightarrow+\infty$. Consequently, for sufficiently large $N$,
\begin{align*}
{\rm Var}_{\nu_p}\left(G_N(\eta_0)\right)
=O(N)+O(1)\int_1^{N}ds_1\int_0^Nds_2\int_0^{+\infty}(s_1+s_2+r)^{-2}dr=O(N)
\end{align*}
and hence \eqref{equ 3.3} holds.

Now we check \eqref{equ 3.4}, according to the definition of $H$ and Proposition \ref{Proposition A.1 dual},
\[
S(N)H(\eta_s)=\sum_{z_1}\sum_{z_2}\eta_s(z_1)\eta_s(z_2)q_N^{\textbf{0}, x}(z_1, z_2).
\]
Then, according to the invariance of $\nu_p$,
\begin{align*}
&{\rm Var}_{\nu_p}\left(\int_0^{tN}S(N)H(\eta_s)-\alpha_{\textbf{0}, x}(p)ds\right)\\
&=2\sum_{z_1, z_2, v_1, v_2}\int_0^{tN}ds_1\int_0^{s_1}{\rm Cov}_{\nu_p}\left(\eta_{s_2}(z_1)\eta_{s_2}(z_2), \eta_0(v_1)\eta_0(v_2)\right)q_N^{\textbf{0}, x}(z_1, z_2)q_N^{\textbf{0}, x}(v_1, v_2)ds_2.
\end{align*}
Consequently, by \eqref{equ 3.5} and an argument similar to that leading to \eqref{equ 3.3},
\begin{align*}
{\rm Var}_{\nu_p}\left(\int_0^{tN}S(N)H(\eta_s)-\alpha_{\textbf{0}, x}(p)ds\right)
\leq 8K_1\int_0^{tN}ds_1\int_0^{s_1}ds_2\int_0^{+\infty}p_{2N+r+s_2}(\textbf{0}, \textbf{0})dr.
\end{align*}
Consequently, for sufficiently large $N$,
\begin{align*}
{\rm Var}_{\nu_p}\left(\int_0^{tN}S(N)H(\eta_s)-\alpha_{\textbf{0}, x}(p)ds\right)
&=O(1)\int_0^{tN}ds_1\int_0^{s_1}ds_2\int_0^{+\infty}(r+s_2+2N)^{-2}dr\\
&=O(N)
\end{align*}
and hence \eqref{equ 3.4} holds.
\qed

Now we deal with the martingale term in decomposition \eqref{equ 3.2 martingale decomposition d=4}. We have the following crucial lemma.
\begin{lemma}\label{lemma 3.2 martingale clt d=4}
For any $T>0$,
\[
\left\{\frac{\mathcal{M}_{tN}^N}{\sqrt{N\log N}}:~0\leq t\leq T\right\}_{N\geq 1}
\]
converges weakly, with respect to the Skorohod topology of $D[0, T]$, to $\{C_4(x)B_t\}_{0\leq t\leq T}$.
\end{lemma}

We prove Lemma \ref{lemma 3.2 martingale clt d=4} step by step. Let $\{\langle \mathcal{M}^N\rangle_t\}_{t\geq 0}$ be the quadratic variation process of $\{\mathcal{M}^N_t\}_{t\geq 0}$. By Theorem 1.4 in Chapter 7 of \cite{Ethier1986}, to verify Lemma \ref{lemma 3.2 martingale clt d=4}, we only need to show that
\[
\lim_{N\rightarrow+\infty}\frac{1}{N\log N}\langle \mathcal{M}^N\rangle_{tN}=C_4(x)^2t
\]
in probability for any $t>0$. Therefore, by Chebyshev's inequality, we only need to verify that
\begin{equation}\label{equ 3.6 QV mean d=4}
\lim_{N\rightarrow+\infty}\frac{1}{N\log N}\mathbb{E}_{\nu_p}\langle \mathcal{M}^N\rangle_{tN}=C_4(x)^2t
\end{equation}
and
\begin{equation}\label{equ 3.7 QV variance d=4}
\lim_{N\rightarrow+\infty}{\rm Var}_{\nu_p}\left(\frac{1}{N\log N}\langle \mathcal{M}^N\rangle_{tN}\right)=0.
\end{equation}
Now we verify \eqref{equ 3.6 QV mean d=4} through several steps. We first introduce a correlation decay lemma.

\begin{lemma}\label{Lemma 3.3 correlation decay d=4}
For any $\epsilon>0$, there exists $M=M(\epsilon)<+\infty$ such that
\[
\left|\mathbb{E}_{\nu_p}\left((\eta_0(x)-\eta_0(y))^2\eta_0(z)\eta_0(w)\right)-\chi(p)p^2\right|\leq \epsilon
\]
and
\[
\left|\mathbb{E}_{\nu_p}\left((\eta_0(x)-\eta_0(y))^2\eta_0(z)\right)-\chi(p)p\right|\leq \epsilon
\]
for any $x, y, z, w\in \mathbb{Z}^4$ such that
\[
x\sim y, \|x-z\|_1\geq M, \|z-w\|_1\geq M\text{~and~}\|x-w\|_1\geq M.
\]
\end{lemma}
Lemma \ref{Lemma 3.3 correlation decay d=4} is easy to catch. Since the simple random walk on $\mathbb{Z}^4$ is transient, for $x, z\in \mathbb{Z}^4$ such that $\|x-z\|_1$ is large, the two coordinates starting from $x, z$ respectively of a coalescing random walk will never be coalesced into a single walk with high probability and consequently $\eta(x)$ and $\eta(z)$ are approximately independent under $\nu_p$. A detailed proof of Lemma \ref{Lemma 3.3 correlation decay d=4} is given in Appendix \ref{Appendix A.3}.

We need Lemma \ref{Lemma 3.3 correlation decay d=4} since terms $(\eta(x)-\eta(y))^2\eta(z)\eta(w), (\eta(x)-\eta(y))^2\eta(z)$ appear in the expression of $\langle \mathcal{M}^N\rangle$. In detail, according to the Dynkin's martingale formula,
\[
\langle \mathcal{M}^N\rangle_{tN}
=\int_0^{tN}\sum_{z_1}\sum_{z_2\sim z_1}\left(G_N(\eta_s^{z_1, z_2})-G_N(\eta_s)\right)^2ds.
\]
According to the definition of $G_N$,
\begin{align*}
\left(G_N(\eta_s^{z_1, z_2})-G_N(\eta_s)\right)^2=&{\rm \uppercase\expandafter{\romannumeral4}}_s^N(z_1, z_2)+{\rm \uppercase\expandafter{\romannumeral5}}_s^N(z_1, z_2)
+{\rm \uppercase\expandafter{\romannumeral6}}_s^N(z_1, z_2)\\
&+{\rm \uppercase\expandafter{\romannumeral7}}_s^N(z_1, z_2)+{\rm \uppercase\expandafter{\romannumeral8}}_s^N(z_1, z_2)+{\rm \uppercase\expandafter{\romannumeral9}}_s^N(z_1, z_2),
\end{align*}
where
\[
{\rm \uppercase\expandafter{\romannumeral4}}_s^N(z_1, z_2)=\left(\eta_s(z_2)-\eta_s(z_1)\right)^2\int_0^N\int_0^Nq_{r_1}^{\textbf{0}, x}(z_1, z_1)q_{r_2}^{\textbf{0}, x}(z_1, z_1)dr_1dr_2,
\]
\begin{align*}
&{\rm \uppercase\expandafter{\romannumeral5}}_s^N(z_1, z_2)\\
&=\sum_{v_2\neq z_1}\sum_{w_2\neq z_1}\left(\eta_s(z_2)-\eta_s(z_1)\right)^2\eta_s(v_2)\eta_s(w_2)\int_0^N\int_0^Nq_{r_1}^{\textbf{0}, x}(z_1, v_2)q_{r_2}^{\textbf{0}, x}(z_1, w_2)dr_1dr_2,
\end{align*}
\begin{align*}
&{\rm \uppercase\expandafter{\romannumeral6}}_s^N(z_1, z_2)\\
&=\sum_{v_1\neq z_1}\sum_{w_1\neq z_1}\left(\eta_s(z_2)-\eta_s(z_1)\right)^2\eta_s(v_1)\eta_s(w_1)\int_0^N\int_0^Nq_{r_1}^{\textbf{0}, x}(v_1, z_1)q_{r_2}^{\textbf{0}, x}(w_1, z_1)dr_1dr_2,
\end{align*}
\begin{align*}
{\rm \uppercase\expandafter{\romannumeral7}}_s^N(z_1, z_2)
=2\sum_{v_2\neq z_1}\left(\eta_s(z_2)-\eta_s(z_1)\right)^2\eta_s(v_2)\int_0^N\int_0^Nq_{r_1}^{\textbf{0}, x}(z_1, z_1)q_{r_2}^{\textbf{0}, x}(z_1, v_2)dr_1dr_2,
\end{align*}
\begin{align*}
{\rm \uppercase\expandafter{\romannumeral8}}_s^N(z_1, z_2)
=2\sum_{v_1\neq z_1}\left(\eta_s(z_2)-\eta_s(z_1)\right)^2\eta_s(v_1)\int_0^N\int_0^Nq_{r_1}^{\textbf{0}, x}(v_1, z_1)q_{r_2}^{\textbf{0}, x}(z_1, z_1)dr_1dr_2,
\end{align*}
and
\begin{align*}
&{\rm \uppercase\expandafter{\romannumeral9}}_s^N(z_1, z_2)\\
&=2\sum_{v_1\neq z_1}\sum_{v_2\neq z_1}\left(\eta_s(z_2)-\eta_s(z_1)\right)^2\eta_s(v_1)\eta_s(v_2)\int_0^N\int_0^Nq_{r_1}^{\textbf{0}, x}(z_1, v_2)q_{r_2}^{\textbf{0}, x}(v_1, z_1)dr_1dr_2.
\end{align*}
The following three lemmas are crucial for us to verify \eqref{equ 3.6 QV mean d=4}.
\begin{lemma}\label{lemma 3.4} For any $t\geq 0$,
\[
\lim_{N\rightarrow+\infty}\frac{1}{N\log N}\int_0^{tN}\sum_{z_1}\sum_{z_2\sim z_1}\mathbb{E}_{\nu_p}{\rm \uppercase\expandafter{\romannumeral4}}_s^N(z_1, z_2)
ds=t\frac{\chi(p)}{2\pi^2}\Phi^2(x).
\]
\end{lemma}
\begin{lemma}\label{lemma 3.5} For any $t\geq 0$,
\begin{align*}
&\lim_{N\rightarrow+\infty}\frac{1}{N\log N}\int_0^{tN}\sum_{z_1}\sum_{z_2\sim z_1}\mathbb{E}_{\nu_p}{\rm \uppercase\expandafter{\romannumeral7}}_s^N(z_1, z_2)
ds\\
&=\lim_{N\rightarrow+\infty}\frac{1}{N\log N}\int_0^{tN}\sum_{z_1}\sum_{z_2\sim z_1}\mathbb{E}_{\nu_p}{\rm \uppercase\expandafter{\romannumeral8}}_s^N(z_1, z_2)
ds=t\frac{2\chi(p)p}{2\pi^2}\Phi(x)(1-\Phi(x)).
\end{align*}
\end{lemma}
\begin{lemma}\label{lemma 3.6} For any $t\geq 0$,
\begin{align*}
&\lim_{N\rightarrow+\infty}\frac{1}{N\log N}\int_0^{tN}\sum_{z_1}\sum_{z_2\sim z_1}\mathbb{E}_{\nu_p}{\rm \uppercase\expandafter{\romannumeral5}}_s^N(z_1, z_2)
ds\\
&=\lim_{N\rightarrow+\infty}\frac{1}{N\log N}\int_0^{tN}\sum_{z_1}\sum_{z_2\sim z_1}\mathbb{E}_{\nu_p}{\rm \uppercase\expandafter{\romannumeral6}}_s^N(z_1, z_2)
ds\\
&=\frac{1}{2}\lim_{N\rightarrow+\infty}\frac{1}{N\log N}\int_0^{tN}\sum_{z_1}\sum_{z_2\sim z_1}\mathbb{E}_{\nu_p}{\rm \uppercase\expandafter{\romannumeral9}}_s^N(z_1, z_2)
ds=t\frac{\chi(p)p^2}{2\pi^2}(1-\Phi(x))^2.
\end{align*}
\end{lemma}
According to the definition of $C_4(x)$, Equation \eqref{equ 3.6 QV mean d=4} follows from Lemmas \ref{lemma 3.4}-\ref{lemma 3.6}.

Now we prove Lemma \ref{lemma 3.4}.
\proof[Proof of Lemma \ref{lemma 3.4}]
According to the invariance of $\nu_p$, for any $s\geq 0$ and $z_1\sim z_2$,
\[
\mathbb{E}_{\nu_p}\left(\left(\eta_s(z_2)-\eta_s(z_1)\right)^2\right)=\chi(p).
\]
Then, since each vertex on $\mathbb{Z}^4$ has $8$ neighbors, to complete this proof we only need to show that
\begin{equation}\label{equ 3.8}
\lim_{N\rightarrow+\infty}\frac{1}{\log N}\sum_{z_1\in \mathbb{Z}^4}\int_0^N\int_0^Nq_{r_1}^{\textbf{0}, x}(z_1, z_1)q_{r_2}^{\textbf{0}, x}(z_1, z_1)dr_1dr_2
=\frac{\Phi^2(x)}{16\pi^2}.
\end{equation}
Now we verify \eqref{equ 3.8}. According to the definition of $q_t^{\textbf{0}, x}$,
\[
\sum_{z_1}q_{r_1}^{\textbf{0}, x}(z_1, z_1)q_{r_2}^{\textbf{0}, x}(z_1, z_1)
=\tilde{\mathbb{P}}\left(\mathcal{X}_{r_1,1}^{\textbf{0}, 0}=\mathcal{X}_{r_1, 2}^{x, 0}=\tilde{\mathcal{X}}_{r_2, 1}^{\textbf{0}, 0}=\tilde{\mathcal{X}}_{r_2, 2}^{x, 0}\right),
\]
where $\left\{\left(\tilde{\mathcal{X}}_{t, 1}^{\textbf{0}, 0}, \tilde{\mathcal{X}}_{t, 2}^{x, 0}\right)\right\}_{t\geq 0}$ is an independent copy of $\left\{\left(\mathcal{X}_{t, 1}^{\textbf{0}, 0}, \mathcal{X}_{t, 2}^{x, 0}\right)\right\}_{t\geq 0}$.

Let $\tau_1=\inf\{t:~\mathcal{X}_{t, 1}^{\textbf{0}, 0}=\mathcal{X}_{t, 2}^{x, 0}\}$ and $\tau_2=\inf\{t:~\tilde{\mathcal{X}}_{t, 1}^{\textbf{0}, 0}=\tilde{\mathcal{X}}_{t, 2}^{x, 0}\}$, then
\[
\tilde{\mathbb{P}}\left(\mathcal{X}_{r_1,1}^{\textbf{0}, 0}=\mathcal{X}_{r_1, 2}^{x, 0}=\tilde{\mathcal{X}}_{r_2, 1}^{\textbf{0}, 0}=\tilde{\mathcal{X}}_{r_2, 2}^{x, 0}\right)=\tilde{\mathbb{P}}\left(\tau_1\leq r_1, \tau_2\leq r_2, \mathcal{X}_{r_1, 1}^{\textbf{0}, 0}=\tilde{\mathcal{X}}_{r_2, 1}^{\textbf{0}, 0}\right).
\]
Let $m(d\theta, dw)$ be the joint distribution of $\left(\tau_1, \mathcal{X}_{\tau_1, 1}^{\textbf{0}, 0}\right)$, then according to the Markov property,
\begin{align*}
&\tilde{\mathbb{P}}\left(\tau_1\leq r_1, \tau_2\leq r_2, \mathcal{X}_{r_1, 1}^{\textbf{0}, 0}=\tilde{\mathcal{X}}_{r_2, 1}^{\textbf{0}, 0}\right)\\
&=\sum_{w_1\in \mathbb{Z}^4}\sum_{w_2\in \mathbb{Z}^4}\int_0^{r_1}\int_0^{r_2}p_{r_1+r_2-\theta_1-\theta_2}(w_1, w_2)m(d\theta_1, dw_1)m(d\theta_2, dw_2).
\end{align*}
Therefore,
\begin{align*}
&\int_0^N\int_0^N\tilde{\mathbb{P}}\left(\mathcal{X}_{r_1,1}^{\textbf{0}, 0}=\mathcal{X}_{r_1, 2}^{x, 0}=\tilde{\mathcal{X}}_{r_2, 1}^{\textbf{0}, 0}=\tilde{\mathcal{X}}_{r_2, 2}^{x, 0}\right)dr_1dr_2\\
&=\sum_{w_1\in \mathbb{Z}^4}\sum_{w_2\in \mathbb{Z}^4}\int_0^N\int_0^N\left(\int_{\theta_1}^N\int_{\theta_2}^Np_{r_1+r_2-\theta_1-\theta_2}(w_1, w_2)dr_1dr_2\right)m(d\theta_1, dw_1)m(d\theta_2, dw_2)\\
&=\sum_{w_1, w_2\in \mathbb{Z}^4}\int_0^N\int_0^N\left(\int_{0}^{N-\theta_1}\int_{0}^{N-\theta_2}p_{r_1+r_2}(\textbf{0}, w_2-w_1)dr_1dr_2\right)m(d\theta_1, dw_1)m(d\theta_2, dw_2).
\end{align*}
According to the local central limit theorem of the simple random walk on $\mathbb{Z}^4$ (see Theorem 2.1.3 in Chapter 2 of \cite{Lawler2010}), for any given $M>0$,
\[
\lim_{t\rightarrow+\infty}\sup_{z\in \mathbb{Z}^4:\|z\|_1\leq M}\left|\frac{p_t(\textbf{0}, z)}{t^{-2}}-f^4_2(\textbf{0})\right|=0.
\]
where $f_t^4$ is defined as in Section \ref{section two}. Then, according to the fact that $f_2^4(\textbf{0})=\frac{1}{16\pi^2}$ and
\[
\int_1^N\int_1^N(r_1+r_2)^{-2}dr_1dr_2=\log N(1+o(1)),
\]
we have
\begin{align}\label{equ 3.9}
&\sum_{\|w_1\|_1\leq M, \|w_2\|\leq M}\int_0^{T}\int_0^{T}\left(\int_{0}^{N-\theta_1}\int_{0}^{N-\theta_2}p_{r_1+r_2}(\textbf{0}, w_2-w_1)dr_1dr_2\right)m(d\theta_1, dw_1)m(d\theta_2, dw_2)\notag\\
&=\sum_{\|w_1\|_1\leq M, \|w_2\|\leq M}\int_0^{T}\int_0^{T}\frac{1}{16\pi^2}m(d\theta_1, dw_1)m(d\theta_2, dw_2)(1+o(1))\log N\notag\\
&=\frac{(1+o(1))\log N}{16\pi^2}\tilde{\mathbb{P}}\left(\tau_1\leq T, \tau_2\leq T, \|\mathcal{X}_{\tau_1, 1}^{\textbf{0}, 0}\|_1\leq M,
\|\tilde{\mathcal{X}}_{\tau_2, 1}^{\textbf{0}, 0}\|_1\leq M\right)
\end{align}
for any $M>0, T>0$ and sufficiently large $N$.

According to the definition of $m(d\theta, dw)$,
\begin{align}\label{equ 3.10}
&\limsup_{M\rightarrow+\infty}\sum_{\|w_1\|_1\geq M \text{~or~} \|w_2\|_1\geq M}\int_0^{T}\int_0^{T}1m(d\theta_1, dw_1)m(d\theta_2, dw_2)\notag\\
&\leq 2\limsup_{M\rightarrow+\infty}\tilde{\mathbb{P}}\left(\sup_{0\leq t\leq T}\|\mathcal{X}_{t, 1}^{\textbf{0}, 0}\|_1\geq M\right)=0.
\end{align}
and
\begin{align}\label{equ 3.11}
&\limsup_{T\rightarrow+\infty}\sum_{w_1, w_2\in \mathbb{Z}^4}\int\int_{T\leq \theta_1<+\infty\text{~or~}T\leq \theta_2<+\infty}1m(d\theta_1, dw_1)m(d\theta_2, dw_2)\notag\\
&\leq 2\limsup_{T\rightarrow+\infty}\tilde{\mathbb{P}}\left(T\leq \tau_1<+\infty\right)=0.
\end{align}
According to the fact that $p_t(\textbf{0}, \textbf{0})=O(1)t^{-2}$ for sufficiently large $t$, for any $w_1, w_2\in \mathbb{Z}^4, \theta_1, \theta_2\leq N$,
\begin{align}\label{equ 3.12}
&\int_{0}^{N-\theta_1}\int_{0}^{N-\theta_2}p_{r_1+r_2}(\textbf{0}, w_2-w_1)dr_1dr_2\notag\\
&\leq \int_{0}^N\int_0^Np_{r_1+r_2}(\textbf{0}, \textbf{0})dr_1dr_2=O(1)\log N
\end{align}
for sufficiently large $N$. By \eqref{equ 3.9}-\eqref{equ 3.12} and a classic truncation argument, we have
\begin{align*}
&\int_0^N\int_0^N\tilde{\mathbb{P}}\left(\mathcal{X}_{r_1,1}^{\textbf{0}, 0}=\mathcal{X}_{r_1, 2}^{x, 0}=\tilde{\mathcal{X}}_{r_2, 1}^{\textbf{0}, 0}=\tilde{\mathcal{X}}_{r_2, 2}^{x, 0}\right)dr_1dr_2\\
&=\frac{(1+o(1))\log N}{16\pi^2}\tilde{\mathbb{P}}(\tau_1<+\infty, \tau_2<+\infty)=\frac{(1+o(1))\log N}{16\pi^2}\Phi^2(x)
\end{align*}
for sufficiently large $N$ and hence \eqref{equ 3.8} holds.
\qed

Now we prove Lemma \ref{lemma 3.5}.
\proof[Proof of Lemma \ref{lemma 3.5}]
We only deal with the ${\rm \uppercase\expandafter{\romannumeral7}}_s^N$ part since the other part follows from a similar argument. For any $M>0$, we claim that
\begin{equation}\label{equ 3.13}
\lim_{N\rightarrow+\infty}\frac{1}{\log N}\sum_{z_1\in \mathbb{Z}^4}\sum_{v_2:0<\|v_2-z_1\|_1\leq M}\int_0^N\int_0^Nq_{r_1}^{\textbf{0}, x}(z_1, z_1)q_{r_2}^{\textbf{0}, x}(z_1, v_2)dr_1dr_2=0.
\end{equation}
We verify \eqref{equ 3.13} in Appendix \ref{Appendix A.4}. According to \eqref{equ 3.13}, Lemma \ref{Lemma 3.3 correlation decay d=4} and a classic truncation argument, to complete this proof we only need to verify that
\begin{equation}\label{equ 3.14}
\lim_{N\rightarrow+\infty}\frac{1}{\log N}\sum_{z_1\in \mathbb{Z}^4}\sum_{v_2\neq z_1}\int_0^N\int_0^Nq_{r_1}^{\textbf{0}, x}(z_1, z_1)q_{r_2}^{\textbf{0}, x}(z_1, v_2)dr_1dr_2=\frac{\Phi(x)(1-\Phi(x))}{16\pi^2},
\end{equation}
which is an analogue of \eqref{equ 3.8}. By \eqref{equ 3.8} and the fact that
\[
\sum_{v_2\in \mathbb{Z}^4}q_{r_2}^{\textbf{0}, x}(z_1, v_2)=p_{r_2}(\textbf{0}, z_1),
\]
to verify \eqref{equ 3.14}, we only need to show that
\begin{equation}\label{equ 3.15}
\lim_{N\rightarrow+\infty}\frac{1}{\log N}\sum_{z_1\in \mathbb{Z}^4}\int_0^N\int_0^Nq_{r_1}^{\textbf{0}, x}(z_1, z_1)p_{r_2}(\textbf{0}, z_1)dr_1dr_2=\frac{\Phi(x)}{16\pi^2}.
\end{equation}
Note that
\begin{align*}
\sum_{z_1\in \mathbb{Z}^4}q_{r_1}^{\textbf{0}, x}(z_1, z_1)p_{r_2}(\textbf{0}, z_1)&=\tilde{\mathbb{P}}\left(\mathcal{X}_{r_1,1}^{\textbf{0}, 0}=\mathcal{X}_{r_1, 2}^{x, 0}=\tilde{\mathcal{X}}_{r_2, 1}^{\textbf{0}, 0}\right)\\
&=\tilde{\mathbb{P}}\left(\tau_1\leq r_1, \mathcal{X}_{r_1, 1}^{\textbf{0}, 0}=\tilde{\mathcal{X}}_{r_2, 1}^{\textbf{0}, 0}\right),
\end{align*}
where all notations are defined as in the proof of Lemma \ref{lemma 3.4}. According to an argument similar to that leading to \eqref{equ 3.8},
\begin{align*}
&\int_0^N\int_0^N\tilde{\mathbb{P}}\left(\tau_1\leq r_1, \mathcal{X}_{r_1, 1}^{\textbf{0}, 0}=\tilde{\mathcal{X}}_{r_2, 1}^{\textbf{0}, 0}\right)dr_1dr_2\\
&=\frac{(1+o(1))\log N}{16\pi^2}\tilde{\mathbb{P}}(\tau_1<+\infty)=\frac{(1+o(1))\log N}{16\pi^2}\Phi(x)
\end{align*}
for sufficiently large $N$ and hence \eqref{equ 3.15} holds.
\qed

Now we prove Lemma \ref{lemma 3.6}.
\proof[Proof of Lemma \ref{lemma 3.6}]
We only deal with the ${\rm \uppercase\expandafter{\romannumeral5}}_s^N$ part since the other two parts follow from similar arguments. For any $z_1, v, w\in \mathbb{Z}^4$, we denote by $\psi(z_1, v, w)$ the term
\[
\min\left\{\|z_1-v\|_1, \|z_1-w\|_1, \|v-w\|_1\right\}.
\]
For any $M>0$ and $z_1\in \mathbb{Z}^4$, we define
\[
\mathcal{R}(M, z_1)
=\left\{(v, w):~v, w\in \mathbb{Z}^4, v\neq z_1, w\neq z_1, \psi(z_1, v, w)\leq M\right\}.
\]
We claim that, for any $M>0$,
\begin{equation}\label{equ 3.16}
\lim_{N\rightarrow+\infty}\frac{1}{\log N}\sum_{z_1\in \mathbb{Z}^4}
\sum_{(v_2, w_2)\in \mathcal{R}(M, z_1)}\int_0^N\int_0^Nq_{r_1}^{\textbf{0}, x}(z_1, v_2)q_{r_2}^{\textbf{0}, x}(z_1, w_2)dr_1dr_2=0,
\end{equation}
which is an analogue of \eqref{equ 3.13}. The verification of \eqref{equ 3.16} is given in Appendix \ref{Appendix A.5}. By \eqref{equ 3.16}, Lemma \ref{Lemma 3.3 correlation decay d=4} and a classic truncation argument, to complete the proof, we only need to verify that
\begin{equation}\label{equ 3.17}
\lim_{N\rightarrow+\infty}\frac{1}{\log N}\sum_{z_1\in \mathbb{Z}^4}\sum_{v_2\neq z_1}\sum_{w_2\neq z_1}\int_0^N\int_0^Nq_{r_1}^{\textbf{0}, x}(z_1, v_2)q_{r_2}^{\textbf{0}, x}(z_1, w_2)
dr_1dr_2=\frac{\left(1-\Phi(x)\right)^2}{16\pi^2}.
\end{equation}
Note that
\[
\sum_{v_2\in \mathbb{Z}^4}q_{r_1}^{\textbf{0}, x}(z_1, v_2)=p_{r_1}(\textbf{0}, z_1).
\]
Hence, by \eqref{equ 3.14}, to verify \eqref{equ 3.17} we only need to show that
\begin{equation}\label{equ 3.18}
\lim_{N\rightarrow+\infty}\frac{1}{\log N}\sum_{z_1\in \mathbb{Z}^4}\sum_{w_2\neq z_1}\int_0^N\int_0^Np_{r_1}(\textbf{0}, z_1)q_{r_2}^{\textbf{0}, x}(z_1, w_2)
dr_1dr_2=\frac{1-\Phi(x)}{16\pi^2}.
\end{equation}
Furthermore, according to \eqref{equ 3.15} and the fact that $\sum_{w_2\in \mathbb{Z}^4}q_{r_2}^{\textbf{0}, x}(z_1, w_2)=p_{r_2}(\textbf{0}, z_1)$, to verify \eqref{equ 3.18} we only need to show that
\begin{equation}\label{equ 3.19}
\lim_{N\rightarrow+\infty}\frac{1}{\log N}\sum_{z_1\in \mathbb{Z}^4}\int_0^N\int_0^Np_{r_1}(\textbf{0}, z_1)p_{r_2}(\textbf{0}, z_1)
dr_1dr_2=\frac{1}{16\pi^2}.
\end{equation}
Since
\[
\sum_{z_1\in \mathbb{Z}^4}p_{r_1}(\textbf{0}, z_1)p_{r_2}(\textbf{0}, z_1)=p_{r_1+r_2}(\textbf{0}, \textbf{0}),
\]
Equation \eqref{equ 3.19} follows from the fact that
\[
p_{t, 4}(\textbf{0}, \textbf{0})=(1+o(1))t^{-2}f_2^4(\textbf{0})
=(1+o(1))t^{-2}\frac{1}{16\pi^2}
\]
as $t\rightarrow+\infty$. Since \eqref{equ 3.19} holds, the proof is complete.
\qed

As we have explained, Equation \eqref{equ 3.6 QV mean d=4} follows from Lemmas \ref{lemma 3.4}-\ref{lemma 3.6}. Hence, we only need to verify \eqref{equ 3.7 QV variance d=4} to complete the proof of Lemma \ref{lemma 3.2 martingale clt d=4}. The verification of \eqref{equ 3.7 QV variance d=4} utilizes a lemma about the correlation decay between $\eta_t$ and $\eta_0$ as $t\rightarrow+\infty$. So we first recall this correlation decay lemma. For any integer $m\geq 1$, $x_1, \ldots, x_m\in \mathbb{Z}^d$, $\eta\in \{0, 1\}^{\mathbb{Z}^d}$ and $H: \{0, 1\}^{m}\rightarrow \mathbb{R}$, we denote by $H(\eta, x_1, \ldots, x_m)$ the term
\[
H(\eta(x_1), \eta(x_2), \ldots, \eta(x_m)).
\]
We denote by $\|H\|_\infty$ the maximum of $|H|$.
We have the following lemma, which is a special case of Lemma 3.3 in \cite{Xue2026b}.
\begin{lemma}\label{lemma 3.7 time correlation decay}(\cite{Xue2026b})
For any $d\geq 3$, $m\geq 1$, $H_1, H_2:\{0, 1\}^{\mathbb{Z}^m}\rightarrow \mathbb{R}$, $t\geq 0$ and
\[
x_1, \ldots, x_m, y_1, \ldots, y_m\in \mathbb{Z}^d,
\]
we have
\[
\left|{\rm Cov}_{\nu_p}\left(H_1(\eta_t, x_1, \ldots, x_m), H_2(\eta_0, y_1, \ldots, y_m)\right)\right|
\leq m^2K_1\|H_1\|_\infty\|H_2\|_\infty\int_t^{+\infty}p_{r, d}(\textbf{0}, \textbf{0})dr,
\]
where $K_1=\left(\int_0^{+\infty}p_{r, d}(\textbf{0}, \textbf{0})dr\right)^{-1}$.
\end{lemma}
Note that \eqref{equ 3.5} is a special case of Lemma \ref{lemma 3.7 time correlation decay}, where $m=2$ and $H_1=H_2=H$ such that $H(u)=u_1u_2$ for any
$u=(u_1, u_2)\in \{0, 1\}^2$. The verification of \eqref{equ 3.5} given in Appendix \ref{Appendix A.2} can be easily extended to general $m$ and $H_1, H_2: \{0, 1\}^{\mathbb{Z}^m}\rightarrow \mathbb{R}$, so we omit a recall of the detailed proof of Lemma \ref{lemma 3.7 time correlation decay} in this paper.

Now we proof \eqref{equ 3.7 QV variance d=4}.
\proof[Proof of \eqref{equ 3.7 QV variance d=4}]
According to the invariance of $\nu_p$, the bilinear property of the covariance, Equations \eqref{equ 3.8}, \eqref{equ 3.14} and \eqref{equ 3.16}, to complete this proof we only need to verify that
\begin{equation*}
\lim_{t\rightarrow+\infty}\sup
{\rm Cov}_{\nu_p}\left((\eta_t(x_1)-\eta_t(x_2))^2, (\eta_0(x_3)-\eta_0(x_4))^2\right)=0,
\end{equation*}
\begin{equation*}
\lim_{t\rightarrow+\infty}\sup
{\rm Cov}_{\nu_p}\left((\eta_t(x_1)-\eta_t(x_2))^2, (\eta_0(x_3)-\eta_0(x_4))^2\eta_0(x_5)\right)=0,
\end{equation*}
\begin{equation*}
\lim_{t\rightarrow+\infty}\sup
{\rm Cov}_{\nu_p}\left((\eta_t(x_1)-\eta_t(x_2))^2, (\eta_0(x_3)-\eta_0(x_4))^2\eta_0(x_5)\eta_0(x_6)\right)=0,
\end{equation*}
\begin{equation*}
\lim_{t\rightarrow+\infty}\sup
{\rm Cov}_{\nu_p}\left((\eta_t(x_1)-\eta_t(x_2))^2\eta_t(x_5), (\eta_0(x_3)-\eta_0(x_4))^2\right)=0,
\end{equation*}
\begin{equation*}
\lim_{t\rightarrow+\infty}\sup
{\rm Cov}_{\nu_p}\left((\eta_t(x_1)-\eta_t(x_2))^2\eta_t(x_5), (\eta_0(x_3)-\eta_0(x_4))^2\eta_0(x_6)\right)=0,
\end{equation*}
\begin{equation*}
\lim_{t\rightarrow+\infty}\sup
{\rm Cov}_{\nu_p}\left((\eta_t(x_1)-\eta_t(x_2))^2\eta_t(x_5), (\eta_0(x_3)-\eta_0(x_4))^2\eta_0(x_6)\eta_0(x_7)\right)=0,
\end{equation*}
\begin{equation*}
\lim_{t\rightarrow+\infty}\sup
{\rm Cov}_{\nu_p}\left((\eta_t(x_1)-\eta_t(x_2))^2\eta_t(x_5)\eta_t(x_6), (\eta_0(x_3)-\eta_0(x_4))^2\right)=0,
\end{equation*}
\begin{equation*}
\lim_{t\rightarrow+\infty}\sup
{\rm Cov}_{\nu_p}\left((\eta_t(x_1)-\eta_t(x_2))^2\eta_t(x_5)\eta_t(x_6), (\eta_0(x_3)-\eta_0(x_4))^2\eta_0(x_7)\right)=0,
\end{equation*}
and
\begin{equation}\label{equ 3.20}
\lim_{t\rightarrow+\infty}\sup
{\rm Cov}_{\nu_p}\left((\eta_t(x_1)-\eta_t(x_2))^2\eta_t(x_5)\eta_t(x_6), (\eta_0(x_3)-\eta_0(x_4))^2\eta_0(x_7)\eta_0(x_8)\right)=0,
\end{equation}
where the supremums in the above nine equations are all over $(x_1, x_2, x_3, x_4, x_5, x_6, x_7, x_8)\in (\mathbb{Z}^4)^8$ such that
$x_1\sim x_2$ and $x_3\sim x_4$.

The above nine equations are all special cases of Lemma \ref{lemma 3.7 time correlation decay} according to the fact that
\[
\lim_{t\rightarrow+\infty}\int_t^{+\infty}p_{r, 4}(\textbf{0}, \textbf{0})dr=0.
\]
For instance, applying Lemma \ref{lemma 3.7 time correlation decay} for $m=4$ and $H_1=H_2=H$ such that
\[
H(u)=(u_1-u_2)^2u_3u_4
\]
for any $u=(u_1, u_2, u_3, u_4)\in \{0, 1\}^4$, we obtain \eqref{equ 3.20}. Since \eqref{equ 3.20} and other eight analogues hold, the proof is complete.
\qed

Note that, since \eqref{equ 3.6 QV mean d=4} and \eqref{equ 3.7 QV variance d=4} hold, the proof of Lemma \ref{lemma 3.2 martingale clt d=4} is complete.

\subsection{Tightness}\label{subsection 3.2}
In this subsection we complete the proof of Theorem \ref{theorem 2.1 main theorem d=4}. By decomposition \eqref{equ 3.2 martingale decomposition d=4} and Lemmas \ref{lemma 3.1}, \ref{lemma 3.2 martingale clt d=4}, we only need to show that
\[
\left\{\frac{\xi_{tN}^{\textbf{0}, x}-tN\alpha_{\textbf{0}, x}(p)}{\sqrt{N\log N}}:~0\leq t\leq T\right\}_{N\geq 1}
\]
are tight under the uniform topology of $C[0, T]$. We will show that
\[
\mathbb{E}_{\nu_p}\left(\left(\frac{\xi_{tN}^{\textbf{0}, x}-tN\alpha_{\textbf{0}, x}(p)}{\sqrt{N\log N}}-\frac{\xi_{sN}^{\textbf{0}, x}-sN\alpha_{\textbf{0}, x}(p)}{\sqrt{N\log N}}\right)^4\right)=O(1)(t-s)^2
\]
for any $0\leq s<t$ and then the aforesaid tightness follows from Corollary 14.9 of \cite{Kallenberg1997}.

According to the definition of $\xi_t$ and the invariance of $\nu_p$,
\begin{align*}
&\mathbb{E}_{\nu_p}\left(\left(\frac{\xi_{tN}^{\textbf{0}, x}-tN\alpha_{\textbf{0}, x}(p)}{\sqrt{N\log N}}-\frac{\xi_{sN}^{\textbf{0}, x}-sN\alpha_{\textbf{0}, x}(p)}{\sqrt{N\log N}}\right)^4\right)\\
&=\frac{24}{\left(N\log N\right)^2}\int_{0}^{(t-s)N}ds_1\int_{0}^{s_1}ds_2\int_{0}^{s_2}ds_3\int_0^{s_3}\mathbb{E}_{\nu_p}\left(\prod_{i=1}^4F(\eta_{s_i})\right)ds_4,
\end{align*}
where $F(\eta)=\eta(\textbf{0})\eta(x)-\alpha_{\textbf{0}, x}(p)$. According to the definition of $\nu_p$ and Proposition \ref{Proposition A.1 dual}, for $s_4< s_3< s_2< s_1$,
\[
\mathbb{E}_{\nu_p}\left(\prod_{i=1}^4F(\eta_{s_i})\right)=
\lim_{r\rightarrow+\infty}\mathbb{E}_{\mu_p}\left(\prod_{i=1}^4\left(F(\eta_{s_i+r})
-\mathbb{E}_{\mu_p}F(\eta_{s_i+r})\right)\right)
\]
and
\begin{align*}
&\mathbb{E}_{\mu_p}\left(\prod_{i=1}^4\left(F(\eta_{s_i+r})-\mathbb{E}_{\mu_p}F(\eta_{s_i+r})\right)\right)\\
&=\tilde{\mathbb{E}}
\left(\mathbb{E}_{\mu_p}\left(\prod_{i=1}^4\left(\eta_0(\mathcal{X}_{s_1+r, 2i-1}^{\textbf{0}, s_1-s_i})\eta_{0}(\mathcal{X}^{x, s_1-s_i}_{s_1+r, 2i})-\alpha_{\textbf{0}, x}(p)-\mathbb{E}_{\mu_p}F(\eta_{s_i+r})\right)\right)\right).
\end{align*}
To bound the right-hand side of the above equation, we introduce eight independent random walks. In detail, for $s_4<s_3<s_2<s_1$ and $i\in \{1, 2, 3, 4\}$, we denote by $\{\hat{X}_{t, 2i-1}\}_{t\geq 0}$ the random walk such that $\hat{X}_{t, 2i-1}=\textbf{0}$ for $0\leq t\leq s_1-s_i$ and $\{\hat{X}_{s_1-s_i+t, 2i-1}\}_{t\geq 0}$ is a copy of $\{X_t\}_{t\geq 0}$ with generator given by \eqref{equ 2.1 generator of Random Walk} starting from $\textbf{0}$. We further denote by $\{\hat{X}_{t, 2i}\}_{t\geq 0}$ the random walk such that $\hat{X}_{t, 2i}=x$ for $0\leq t\leq s_1-s_i$ and $\{\hat{X}_{s_1-s_i+t, 2i}\}_{t\geq 0}$ is a copy of $\{X_t\}_{t\geq 0}$ starting from $x$. We assume that the above eight random walks are independent.

Later we will couple the coalescing random walk
\[
\left\{\left(\mathcal{X}_{t, 1}^{\textbf{0}, 0}, \mathcal{X}_{t, 2}^{x, 0}, \mathcal{X}_{t, 3}^{\textbf{0}, s_1-s_2},
\mathcal{X}_{t, 4}^{x, s_1-s_2}, \mathcal{X}_{t, 5}^{\textbf{0}, s_1-s_3}, \mathcal{X}_{t, 6}^{x, s_1-s_3}, \mathcal{X}_{t, 7}^{\textbf{0}, s_1-s_4}, \mathcal{X}_{t, 8}^{x, s_1-s_4}\right)\right\}_{t\geq 0}
\]
and $\left\{\hat{X}_{t, i}:~t\geq 0\right\}_{1\leq i\leq 8}$ in the same probability space. Now we give some upper bounds of collision probabilities between $\left\{\hat{X}_{t, i}:~t\geq 0\right\}_{1\leq i\leq 8}$ as a preliminary step for the tightness verification.

For any different $i, j\in \{1, 2, 3, 4\}$, $k\in \{2i-1, 2i\}$ and $l\in \{2j-1, 2j\}$, we denote by $\hat{A}_{k,l}$ the event that
$\hat{X}_{t, k}=\hat{X}_{t, l}$ for some
\[
t\geq \max\{s_1-s_i, s_1-s_j\}.
\]
For any different $i, j, n\in \{1, 2, 3, 4\}$, $k\in \{2i-1, 2i\}$, $l\in \{2j-1, 2j\}$ and $m\in \{2n-1, 2n\}$, we denote by $\hat{H}_{k, l, m}$ the event
$\hat{A}_{k, l}\bigcap \hat{A}_{k, n}$. For any different $i, j\in \{1, 2, 3, 4\}$, $\{\hat{i}, \hat{j}\}=\{1, 2, 3, 4\}\setminus\{i, j\}$ and $k\in \{2i-1, 2i\}$, $l\in \{2i-1, 2i, 2j-1, 2j\}\setminus\{k\}$, $m\in \{2\hat{i}-1, 2\hat{i}\}$ and $\tilde{m}\in \{2\hat{j}-1, 2\hat{j}\}$, we denote by $\hat{J}_{k, l, m, \tilde{m}}$ the event
$\hat{A}_{k, m}\bigcap\hat{A}_{l, \tilde{m}}$.
We write $\hat{A}_{k, l}$ as $\hat{A}_{k, l}^{s_1, s_2, s_3, s_4}$ and so does other notations when we need to emphasize the dependence of $\hat{A}_{k, l}$ on $s_1, s_2, s_3, s_4$. The following two lemmas are crucial for the verification of the tightness of the quadratic additive functionals.
\begin{lemma}\label{lemma 3.8}
There exists $K_4<+\infty$ independent of $t, k, l, m, N$ such that
\[
\int_0^{tN}ds_1\int_0^{s_1}ds_2\int_0^{s_2}ds_3\int_0^{s_3}\tilde{\mathbb{P}}\left(\hat{H}_{k, l, m}^{s_1, s_2, s_3, s_4}\right)ds_4
\leq K_4(N\log N)^2t^2
\]
for any $0\leq t\leq 1$, $N\geq 5$, different $i, j, n\in \{1, 2, 3, 4\}$ and $k\in \{2i-1, 2i\}, l\in \{2j-1, 2j\}, m\in \{2n-1, 2n\}$.
\end{lemma}

\begin{lemma}\label{lemma 3.9}
There exists $K_5<+\infty$ independent of $t, k, l, m, \tilde{m}, N$ such that
\[
\int_0^{tN}ds_1\int_0^{s_1}ds_2\int_0^{s_2}ds_3\int_0^{s_3}\tilde{\mathbb{P}}\left(\hat{J}_{k, l, m, \tilde{m}}^{s_1, s_2, s_3, s_4}\right)ds_4
\leq K_5(N\log N)^2t^2
\]
for any $0\leq t\leq 1$, $N\geq 5$, different $i, j\in \{1, 2, 3, 4\}$ and $k\in \{2i-1, 2i\}$, $l\in \{2i-1, 2i, 2j-1, 2j\}\setminus\{k\}$, $m\in \{2\hat{i}-1, 2\hat{i}\}$, $\tilde{m}\in
\{2\hat{j}-1, 2\hat{j}\}$, where $\{\hat{i}, \hat{j}\}=\{1, 2, 3, 4\}\setminus\{i, j\}$.
\end{lemma}

\proof[Proof of Lemma \ref{lemma 3.8}]
We only deal with the case where $(k, l, m)=(1, 3, 5)$ since other cases follow from similar arguments. Let
\[
\hat{\varpi}_1=\inf\{t:~t\geq s_1-s_2, \hat{X}_{t, 1}=\hat{X}_{t, 3}\}
\]
and
\[
\hat{\varpi}_2=\inf\{t:~t\geq s_1-s_3, \hat{X}_{t, 1}=\hat{X}_{t, 5}\},
\]
then
\begin{align*}
&\tilde{\mathbb{P}}\left(\hat{H}_{1, 3, 5}^{s_1, s_2, s_3, s_4}\right)\\
&=\tilde{\mathbb{P}}\left(\hat{\varpi}_1<s_1-s_3<\hat{\varpi}_2<+\infty, s_1-s_3<\hat{\varpi}_1<\hat{\varpi}_2<+\infty, \hat{\varpi}_2<\hat{\varpi}_1<+\infty\right).
\end{align*}
Let $\sigma(dr, dz)$ be the joint distribution of $\hat{\varpi}_1-(s_1-s_2)$ and $\hat{X}_{\hat{\varpi}_1, 1}$. According to the Markov property,
\begin{align*}
&\tilde{\mathbb{P}}\left(\hat{\varpi}_1<s_1-s_3<\hat{\varpi}_2<+\infty\right)\\
&=\sum_{z}\int_{0}^{s_2-s_3}\sum_{w}p_{s_2-s_3-r}(\textbf{0}, w)\Phi(w)\sigma(dr, dz)\\
&=K_1\sum_{z}\int_{0}^{s_2-s_3}\left(\int_0^{+\infty}p_{s_2-s_3-r+\theta}(\textbf{0}, \textbf{0})d\theta\right)\sigma(dr, dz),
\end{align*}
where $K_1=\int_0^{+\infty}p_s(\textbf{0}, \textbf{0})ds$. As a result,
\begin{align*}
&\int_0^{tN}ds_1\int_0^{s_1}ds_2\int_0^{s_2}ds_3\int_0^{s_3}\tilde{\mathbb{P}}\left(\hat{\varpi}_1<s_1-s_3<\hat{\varpi}_2<+\infty\right)ds_4\\
&\leq tN \int_0^{tN}ds_1\int_0^{s_1}ds_2\int_0^{s_2}\tilde{\mathbb{P}}\left(\hat{\varpi}_1<s_1-s_3<\hat{\varpi}_2<+\infty\right)ds_3\\
&\leq tNK_1\sum_{z}\int_0^{tN}ds_1\int_0^{s_1}ds_2\int_0^{s_2}\left(\int_{0}^{s_2-r}ds_3\int_0^{+\infty}p_{s_2-s_3-r+\theta}(\textbf{0}, \textbf{0})d\theta\right)\sigma(dr, dz)\\
&=tNK_1\sum_{z}\int_0^{tN}ds_1\int_0^{s_1}ds_2\int_0^{s_2}\left(\int_{0}^{s_2-r}ds_3\int_0^{+\infty}p_{s_3+\theta}(\textbf{0}, \textbf{0})d\theta\right)\sigma(dr, dz)\\
&\leq tNK_1\sum_{z}\int_0^{tN}ds_1\int_0^{s_1}\hat{G}(s_2)\left(\int_0^{s_2}1\sigma(dr, dz)\right)ds_2,
\end{align*}
where
\[
\hat{G}(s_2)=\left(\int_{0}^{s_2}ds_3\int_0^{+\infty}p_{s_3+\theta}(\textbf{0}, \textbf{0})d\theta\right).
\]
As a result,
\begin{align*}
&\int_0^{tN}ds_1\int_0^{s_1}ds_2\int_0^{s_2}ds_3\int_0^{s_3}\tilde{\mathbb{P}}\left(\hat{\varpi}_1<s_1-s_3<\hat{\varpi}_2<+\infty\right)ds_4\\
&\leq tNK_1\hat{G}(N)\sum_{z}\int_0^{tN}ds_1\int_0^{s_1}\left(\int_0^{s_2}1\sigma(dr, dz)\right)ds_2\\
&=tNK_1\hat{G}(N)\int_0^{tN}ds_1\int_0^{s_1}\tilde{\mathbb{P}}\left(\hat{\varpi}_1<s_1\right)ds_2\\
&\leq tNK_1\hat{G}(N)\int_0^{tN}ds_1\int_0^{s_1}\tilde{\mathbb{P}}\left(\hat{\varpi}_1<+\infty\right)ds_2\\
&=tNK_1\hat{G}(N)\int_0^{tN}ds_1\int_0^{s_1}\sum_{w}p_{s_1-s_2}(\textbf{0}, w)\Phi(w)ds_2\\
&=tNK_1^2\hat{G}(N)\int_0^{tN}ds_1\int_0^{s_1}\left(\int_0^{+\infty}p_{s_1-s_2+r}\left(\textbf{0}, \textbf{0}\right)dr\right)ds_2\\
&=tNK_1^2\hat{G}(N)\int_0^{tN}\hat{G}(s_1)ds_1.
\end{align*}
According to similar arguments, the above inequality still holds when $\{\hat{\varpi}_1<s_1-s_3<\hat{\varpi}_2<+\infty\}$ is replaced by $\{s_1-s_3<\hat{\varpi}_1<\hat{\varpi}_2<+\infty\}$ or $\{\hat{\varpi}_2<\hat{\varpi}_1<+\infty\}$.
According to the local central limit theorem of the simple random walk on $\mathbb{Z}^4$, $p_{t, 4}(\textbf{0}, \textbf{0})=O(1)t^{-2}$ as $t\rightarrow +\infty$. Consequently, $\hat{G}(s)=O(1)\log s$ as $s\rightarrow+\infty$ and then the existence of $K_4$ follows from the fact that $\int_1^s (\log r) dr=O(1)s\log s$ as $s\rightarrow+\infty$.
\qed

\proof[Proof of Lemma \ref{lemma 3.9}]
We only deal with the cases where $(k, l, m, \tilde{m})=(1, 3, 5, 7)$ or $(1, 2, 5, 7)$ since other cases follow from similar arguments. Let
\[
\tilde{\varpi}_1=\inf\{t:~t\geq s_1-s_3, \hat{X}_{t, 1}=\hat{X}_{t, 5}\}
\]
and
\[
\tilde{\varpi}_2=\inf\{t:~t\geq s_1-s_4, \hat{X}_{t, 3}=\hat{X}_{t, 7}\}.
\]
According to the definition of $\hat{J}_{1, 3, 5, 7}$,
\[
\tilde{\mathbb{P}}\left(\hat{J}_{1, 3, 5, 7}^{s_1, s_2, s_3, s_4}\right)
=\tilde{\mathbb{P}}\left(\tilde{\varpi}_1<+\infty\right)\tilde{\mathbb{P}}\left(\tilde{\varpi}_2<+\infty\right).
\]
According to the Markov property,
\begin{align*}
\tilde{\mathbb{P}}\left(\tilde{\varpi}_1<+\infty\right)
=\sum_{w}p_{s_1-s_3}(\textbf{0}, w)\Phi(w)
=K_1\int_0^{+\infty}p_{s_1-s_3+r}(\textbf{0}, \textbf{0})dr.
\end{align*}
Similarly,
\[
\tilde{\mathbb{P}}\left(\tilde{\varpi}_2<+\infty\right)=K_1\int_0^{+\infty}p_{s_2-s_4+r}(\textbf{0}, \textbf{0})dr.
\]
As a result,
\begin{align*}
&\int_0^{tN}ds_1\int_0^{s_1}ds_2\int_0^{s_2}ds_3\int_0^{s_3}\tilde{\mathbb{P}}\left(\hat{J}_{1, 3, 5, 7}^{s_1, s_2, s_3, s_4}\right)ds_4\\
&\leq K_1^2\int_0^{tN}\left(\int_0^{s_1}ds_3\int_0^{+\infty}p_{s_1-s_3+r}(\textbf{0}, \textbf{0})dr\right)ds_1\\
&\text{\quad\quad\quad\quad\quad}\times\int_0^{tN}\left(\int_0^{s_2}ds_4\int_0^{+\infty}p_{s_2-s_4+r}(\textbf{0}, \textbf{0})dr\right)ds_2\\
&=K_1^2\left(\int_0^{tN}\hat{G}(s_1)ds_1\right)^2,
\end{align*}
where $\hat{G}$ is defined as in the proof of Lemma \ref{lemma 3.8}.

According to a similar argument,
\begin{align*}
\int_0^{tN}ds_1\int_0^{s_1}ds_2\int_0^{s_2}ds_3\int_0^{s_3}\tilde{\mathbb{P}}\left(\hat{J}_{1, 2, 5, 7}^{s_1, s_2, s_3, s_4}\right)ds_4
\leq K_1^2tN\int_0^{tN}\left(\hat{G}(s_1)\right)^2ds_1.
\end{align*}
As we have shown in the proof of Lemma \ref{lemma 3.8}, $\hat{G}(s)=O(1)\log s$ as $s\rightarrow+\infty$. Consequently, the existence of $K_5$ follows from the facts that $\int_1^s (\log r) dr=O(1)s\log s$ and $\int_1^s(\log r)^2dr=O(1)s(\log s)^2$ as $s\rightarrow+\infty$.
\qed

Now we couple
\[
\left\{\left(\mathcal{X}_{t, 1}^{\textbf{0}, 0}, \mathcal{X}_{t, 2}^{x, 0}, \mathcal{X}_{t, 3}^{\textbf{0}, s_1-s_2},
\mathcal{X}_{t, 4}^{x, s_1-s_2}, \mathcal{X}_{t, 5}^{\textbf{0}, s_1-s_3}, \mathcal{X}_{t, 6}^{x, s_1-s_3}, \mathcal{X}_{t, 7}^{\textbf{0}, s_1-s_4}, \mathcal{X}_{t, 8}^{x, s_1-s_4}\right)\right\}_{t\geq 0}
\]
and $\left\{\hat{X}_{t, i}:~t\geq 0\right\}_{1\leq i\leq 8}$ in the same probability space. For any different $k, l\in \{1, 2, \ldots, 8\}$, we denote by $a_{k,l}$ the stopping time
\[
\inf\left\{t:~t\geq s_1-s_{\lceil\frac{k}{2}\rceil}, t\geq s_1-s_{\lceil\frac{l}{2}\rceil}, \hat{X}_{t, k}=\hat{X}_{t, l}\right\}.
\]
We denote by $\Omega_0$ the set $\{1, 2, \ldots, 8\}$. Let
\[
\tilde{a}_1=\inf\{a_{k, l}:~1\leq k<l\leq 8\}.
\]
If $\tilde{a}_1=+\infty$, we define $\mathcal{X}_{t, j}^{y_j, s_1-s_{\lceil\frac{j}{2}\rceil}}=\hat{X}_{t, j}$ for any $t\geq 0$ and $1\leq j\leq 8$, where
$y_j=\textbf{0}$ when $j\in \{1, 3, 5, 7\}$ and $y_j=x$ when $j\in \{2, 4, 6, 8\}$. If $\tilde{a}_1<+\infty$, there exist $1\leq k_1<l_1\leq 8$ such that $\tilde{a}_1=a_{k_1, l_1}$. We let $\Omega_1=\Omega_0\setminus\{l_1\}$ and
\[
\tilde{a}_2=\inf\{a_{k, l}:~k\in \Omega_1, l\in \Omega_1, k\neq l\}.
\]
If $\tilde{a}_2=+\infty$, we define $\mathcal{X}_{t, j}^{y_j, s_1-s_{\lceil\frac{j}{2}\rceil}}=\hat{X}_{t, j}$ for $0\leq t<\tilde{a}_1, j\in \Omega_0$ and $\mathcal{X}_{t, j}^{y_j, s_1-s_{\lceil\frac{j}{2}\rceil}}=\hat{X}_{t, j}$ for $t\geq \tilde{a}_1, j\in \Omega_1$ and
\[
\mathcal{X}_{t, l_1}^{y_{l_1}, s_1-s_{\lceil\frac{l_1}{2}\rceil}}=\mathcal{X}_{t, k_1}^{y_{k_1}, s_1-s_{\lceil\frac{k_1}{2}\rceil}}
\]
for any $t\geq \tilde{a}_1$. If $\tilde{a}_2<+\infty$, there exist $k_2, l_2\in \Omega_1$ such that $k_2<l_2$ and $\tilde{a}_2=a_{k_2, l_2}$. We denote by $\Omega_2$ the set $\Omega_1\setminus\{l_2\}$ and let
\[
\tilde{a}_3=\inf\{a_{k, l}:~k\in \Omega_2, l\in \Omega_2, k\neq l\}.
\]
If $\tilde{a}_3=+\infty$, we define
\[
\mathcal{X}_{t, j}^{y_j, s_1-s_{\lceil\frac{j}{2}\rceil}}=\hat{X}_{t, j} \text{~for~} 0\leq t<\tilde{a}_1, j\in \Omega_0
\]
and
\[
\mathcal{X}_{t, j}^{y_j, s_1-s_{\lceil\frac{j}{2}\rceil}}=\hat{X}_{t, j} \text{~for~} \tilde{a}_1\leq t<\tilde{a}_2, j\in \Omega_1
\]
and
\[
\mathcal{X}_{t, j}^{y_j, s_1-s_{\lceil\frac{j}{2}\rceil}}=\hat{X}_{t, j} \text{~for~} t>\tilde{a}_2, j\in \Omega_2
\]
and
\[
\mathcal{X}_{t, l_i}^{y_{l_i}, s_1-s_{\lceil\frac{l_i}{2}\rceil}}=\mathcal{X}_{t, k_i}^{y_{k_i}, s_1-s_{\lceil\frac{k_i}{2}\rceil}}
\]
for $t\geq \tilde{a}_i$, where $i=2, 1$.

By induction,  if $\tilde{a}_1<\tilde{a}_2<\ldots<\tilde{a}_n<+\infty$, there exist $k_n, l_n\in \Omega_{n-1}$ such that $k_n<l_n$ and $\tilde{a}_n=a_{k_n, l_n}$. We let $\Omega_n=\Omega_{n-1}\setminus\{l_n\}$ and let
\[
\tilde{a}_{n+1}=\inf\{a_{k, l}:~k\in \Omega_n, l\in \Omega_n, k\neq l\}.
\]
If $\tilde{a}_{n+1}=+\infty$, for $0\leq i\leq n$, we define $\mathcal{X}_{t, j}^{y_j, s_1-s_{\lceil\frac{j}{2}\rceil}}=\hat{X}_{t, j}$ for $\tilde{a}_{i}\leq t<\tilde{a}_{i+1}, j\in\Omega_i$, where $\tilde{a}_0=0$. We further define
\[
\mathcal{X}_{t, l_i}^{y_{l_i}, s_1-s_{\lceil\frac{l_i}{2}\rceil}}=\mathcal{X}_{t, k_i}^{y_{k_i}, s_1-s_{\lceil\frac{k_i}{2}\rceil}}
\]
for $t\geq \tilde{a}_i$ where $i=n, \ldots, 1$. Note that, since $\Omega_7$ has only one element when $\tilde{a}_7<+\infty$, we have $\tilde{a}_8=+\infty$. So the above induction has finite many steps.

The meaning of the above equivalent definition of the coalescing random walk is clear. We can equivalently consider the coalescing random walk as the process where one active coordinate is remained as the coalescence result and the other active coordinate is deleted when these two coordinates collide with each other.
For $i<j$, if $\hat{X}_{\cdot, i}$ and $\hat{X}_{\cdot, j}$ are remained until $t\geq \max\{s_1-s_{\lceil\frac{i}{2}\rceil}, s_1-s_{\lceil\frac{j}{2}\rceil}\}$ and $\hat{X}_{t, i}=\hat{X}_{t, j}$, we require that $\hat{X}_{\cdot, i}$ is remained and $\hat{X}_{\cdot, j}$ is deleted after moment $t$.

We further introduce four independent coalescing random walks. For $0\leq s_4<s_3<s_2<s_1$ and $i\in \{1, 2, 3, 4\}$, we define $\hat{\mathcal{X}}_{t, 1}^{\textbf{0}, s_1-s_i}=\hat{X}_{t, 2i-1}$ and
\[
\hat{\mathcal{X}}_{t, 2}^{x, s_1-s_i}
=
\begin{cases}
\hat{X}_{t, 2i} & \text{~if~}t<a_{2i-1, 2i},\\
\hat{X}_{t, 2i-1} & \text{~if~}t\geq a_{2i-1, 2i}.
\end{cases}
\]
Then, $\left\{\left(\hat{\mathcal{X}}_{t, 1}^{\textbf{0}, s_1-s_i}, \hat{\mathcal{X}}_{t, 2}^{x, s_1-s_i}\right)\right\}_{t\geq 0}$ is a copy of $\left\{\left(\mathcal{X}_{t, 1}^{\textbf{0}, s_1-s_i}, \mathcal{X}_{t, 2}^{x, s_1-s_i}\right)\right\}_{t\geq 0}$ and
\[
\left\{\left(\hat{\mathcal{X}}_{t, 1}^{\textbf{0}, s_1-s_i}, \hat{\mathcal{X}}_{t, 2}^{x, s_1-s_i}\right):~t\geq 0\right\}_{1\leq i\leq 4}
\]
are independent.

For $0\leq s_4<s_3<s_2<s_1$ and $r>0$, we denote by $\tilde{V}^{s_1, s_2, s_3, s_4, r}$ the random term
\[
\mathbb{E}_{\mu_p}\left(\prod_{i=1}^4\left(\eta_0(\mathcal{X}_{s_1+r, 2i-1}^{\textbf{0}, s_1-s_i})\eta_{0}(\mathcal{X}^{x, s_1-s_i}_{s_1+r, 2i})-\alpha_{\textbf{0}, x}(p)-\mathbb{E}_{\mu_p}F(\eta_{s_i+r})\right)\right).
\]
As we have explained,
\[
\mathbb{E}_{\mu_p}\left(\prod_{i=1}^4\left(F(\eta_{s_i+r})-\mathbb{E}_{\mu_p}F(\eta_{s_i+r})\right)\right)
=\tilde{\mathbb{E}}\tilde{V}^{s_1, s_2, s_3, s_4, r}.
\]
To bound $\tilde{\mathbb{E}}\tilde{V}^{s_1, s_2, s_3, s_4, r}$ from above, we further introduce some notations.
For any different $i, j\in \{1, 2, 3, 4\}$, we denote by $\hat{D}_{i, j}$ the event that
\[
\left\{\hat{\mathcal{X}}_{t, 1}^{\textbf{0}, s_1-s_i}, \hat{\mathcal{X}}_{t, 2}^{x, s_1-s_i}\right\}\bigcap \left\{\hat{\mathcal{X}}_{t, 1}^{\textbf{0}, s_1-s_j}, \hat{\mathcal{X}}_{t, 2}^{x, s_1-s_j}\right\}\neq \emptyset
\]
for some $t\geq \max\{s_1-s_i, s_1-s_j\}$. For each $i\in \{1, 2, 3, 4\}$, we denote by $\tilde{D}_i$ the event $\bigcup_{j\neq i}\hat{D}_{i, j}$.
For different $i, j, k\in \{1, 2, 3, 4\}$, we denote by $\mathcal{E}_{i, j, k}$ the event $\hat{D}_{i, j}\bigcap \hat{D}_{i, k}$. For different $i, j\in \{1, 2, 3, 4\}$ and $\{\hat{i}, \hat{j}\}=\{1, 2, 3, 4\}\setminus \{i, j\}$, we denote by $\mathcal{E}_{ij, \hat{i}\hat{j}}$ the event $\hat{D}_{i, j}\bigcap \hat{D}_{\hat{i}, \hat{j}}$ and denote by $\mathcal{E}_{ij, \hat{i}, \hat{j}}$ the event
\[
\hat{D}_{i, j}\bigcap \left(\tilde{D}_{\hat{i}}\right)^c \bigcap \left(\tilde{D}_{\hat{j}}\right)^c,
\]
where $A^c$ is the complement of the set $A$. We denote by $\tilde{\mathcal{E}}$ the event
\[
\left(\tilde{D}_{1}\right)^c \bigcap \left(\tilde{D}_{2}\right)^c \bigcap \left(\tilde{D}_{3}\right)^c\bigcap \left(\tilde{D}_{4}\right)^c.
\]
We write $\hat{D}_{i, j}$ as $\hat{D}_{i, j}^{s_1, s_2, s_3, s_4}$ and so does other notations when we need to emphasize the dependence of $\hat{D}_{i, j}$ on $s_1, s_2, s_3, s_4$. By Lemmas \ref{lemma 3.8}, \ref{lemma 3.9} and aforesaid equivalent definition of the coalescing random walk, we have the following lemma.
\begin{lemma}\label{lemma 3.10}
There exists $K_6, K_7, K_8, K_9<+\infty$ independent of $t, N, i, j, k, r$ such that
\begin{equation}\label{equ 3.20 b}
\left|\int_0^{tN}ds_1\int_0^{s_1}ds_2\int_0^{s_2}ds_3\int_0^{s_3}\tilde{\mathbb{E}}\left(\tilde{V}^{s_1, s_2, s_3, s_4, r}1_{\mathcal{E}_{i, j, k}^{s_1, s_2, s_3, s_4}}\right)ds_4
\right|\leq K_6(N\log N)^2t^2,
\end{equation}
\begin{equation}\label{equ 3.20 c}
\left|\int_0^{tN}ds_1\int_0^{s_1}ds_2\int_0^{s_2}ds_3\int_0^{s_3}\tilde{\mathbb{E}}\left(\tilde{V}^{s_1, s_2, s_3, s_4, r}1_{\mathcal{E}_{ij, \hat{i}\hat{j}}^{s_1, s_2, s_3, s_4}}\right)ds_4
\right|\leq K_7(N\log N)^2t^2,
\end{equation}
\begin{equation}\label{equ 3.20 d}
\left|\int_0^{tN}ds_1\int_0^{s_1}ds_2\int_0^{s_2}ds_3\int_0^{s_3}\tilde{\mathbb{E}}\left(\tilde{V}^{s_1, s_2, s_3, s_4, r}1_{\mathcal{E}_{ij, \hat{i}, \hat{j}}^{s_1, s_2, s_3, s_4}}\right)ds_4
\right|\leq K_8(N\log N)^2t^2
\end{equation}
and
\begin{equation}\label{equ 3.20 f}
\left|\int_0^{tN}ds_1\int_0^{s_1}ds_2\int_0^{s_2}ds_3\int_0^{s_3}\tilde{\mathbb{E}}\left(\tilde{V}^{s_1, s_2, s_3, s_4, r}1_{\tilde{\mathcal{E}}^{s_1, s_2, s_3, s_4}}\right)ds_4
\right|\leq K_9(N\log N)^2t^2
\end{equation}
for any $0\leq t\leq 1$, $N\geq 5$ and different $i, j, k\in \{1, 2, 3, 4\}$, where $1_A$ is the indicator function of the event $A$.
\end{lemma}

\proof[Proof of Lemma \ref{lemma 3.10}]
We only deal with the case where $(i, j, k)=(1, 2, 3)$, since other cases follow from similar arguments. According to our equivalent definition of the coalescing random walks,
\[
\mathcal{E}_{1, 2, 3}\subseteq \bigcup_{k, l, n} \hat{H}_{k, l, n}\bigcup\bigcup_{k, l, m, \tilde{m}}\hat{J}_{k, l, m, \tilde{m}},
\]
where the first union is over all $(k, l, n)$ such that $k\in \{1, 2\}, l\in \{3, 4\}, n\in \{5, 6\}$ and the third union is over all $(k, l, m, \tilde{m})$ such that $\{k, l\}=\{1, 2\}, m\in \{3, 4\}, \tilde{m}\in \{5, 6\}$. Hence, the existence of $K_6$ follows from Lemmas \ref{lemma 3.8} and \ref{lemma 3.9}.

Similarly,
\[
\mathcal{E}_{12, 34}\subseteq \bigcup_{k, l, m, \tilde{m}}\hat{J}_{k, l, m, \tilde{m}},
\]
where the union is over all $(k, l, m, \tilde{m})$ such that $k\in \{1, 2\}, l\in \{5, 6\}, m\in \{3, 4\}, \tilde{m}\in \{7, 8\}$. Hence, the existence of $K_7$ follows from Lemma \ref{lemma 3.9}.

On the event $\mathcal{E}_{12, 3, 4}$,
\[
\left(\mathcal{X}_{s_1+r, 5}^{\textbf{0}, s_1-s_3}, \mathcal{X}_{s_1+r, 6}^{x, s_1-s_3}\right)
=\left(\hat{\mathcal{X}}_{s_1+r, 1}^{\textbf{0}, s_1-s_3}, \hat{\mathcal{X}}_{s_1+r, 2}^{x, s_1-s_3}\right),
\]
\[
\left(\mathcal{X}_{s_1+r, 7}^{\textbf{0}, s_1-s_4}, \mathcal{X}_{s_1+r, 8}^{x, s_1-s_4}\right)
=\left(\hat{\mathcal{X}}_{s_1+r, 1}^{\textbf{0}, s_1-s_4}, \hat{\mathcal{X}}_{s_1+r, 2}^{x, s_1-s_4}\right)
\]
and
\[
\mathbb{E}_{\mu_p}\left(\prod_{i=1}^2\left(\eta_0(\mathcal{X}_{s_1+r, 2i-1}^{\textbf{0}, s_1-s_i})\eta_{0}(\mathcal{X}^{x, s_1-s_i}_{s_1+r, 2i})-\alpha_{\textbf{0}, x}(p)-\mathbb{E}_{\mu_p}F(\eta_{s_i+r})\right)\right)
\]
equals a random variable $\mathcal{U}^{s_1, s_2, s_3, s_4, r}$ which is measurable with
\[
\left\{\left(\hat{\mathcal{X}}_{t, 1}^{\textbf{0}, 0}, \hat{\mathcal{X}}_{t, 2}^{x, 0}, \hat{\mathcal{X}}_{t, 1}^{\textbf{0}, s_1-s_2}, \hat{\mathcal{X}}_{t, 2}^{x, s_1-s_2}\right)\right\}_{t\geq 0}.
\]
Then, according to the fact that $\mu_p$ is a product measure,
\[
\tilde{V}^{s_1, s_2, s_3, s_4, r}=\tilde{R}^{s_1, s_2, s_3, s_4, r}
\]
on the event $\mathcal{E}_{12, 3, 4}$, where
\begin{align*}
\tilde{R}^{s_1, s_2, s_3, s_4, r}=&\mathcal{U}^{s_1, s_2, s_3, s_4, r}
\mathbb{E}_{\mu_p}\left(\eta_0(\hat{\mathcal{X}}_{s_1+r, 1}^{\textbf{0}, s_1-s_3})\eta_{0}(\hat{\mathcal{X}}^{x, s_1-s_3}_{s_1+r, 2})-\alpha_{\textbf{0}, x}(p)-\mathbb{E}_{\mu_p}F(\eta_{s_3+r})\right)\\
&\times\mathbb{E}_{\mu_p}\left(\eta_0(\hat{\mathcal{X}}_{s_1+r, 1}^{\textbf{0}, s_1-s_4})\eta_{0}(\hat{\mathcal{X}}^{x, s_1-s_4}_{s_1+r, 2})-\alpha_{\textbf{0}, x}(p)-\mathbb{E}_{\mu_p}F(\eta_{s_4+r})\right).
\end{align*}
Since
\begin{align*}
\hat{D}_{1, 2}\setminus \mathcal{E}_{12, 3, 4}\subseteq\bigcup_{\{k, l, n\}=\{1, 2, 3\}}\mathcal{E}_{k, l, n}\bigcup\bigcup_{\{\tilde{k}, \tilde{l}, \tilde{n}\}=\{1, 2, 4\}}\mathcal{E}_{\tilde{k}, \tilde{l}, \tilde{n}}\bigcup\mathcal{E}_{12, 34},
\end{align*}
there exists $K_{11}<+\infty$ independent of $t, N, r$ such that
\begin{align}\label{equ 3.24 b}
&\left|\int_0^{tN}ds_1\int_0^{s_1}ds_2\int_0^{s_2}ds_3\int_0^{s_3}\tilde{\mathbb{E}}\left(\tilde{R}^{s_1, s_2, s_3, s_4, r}1_{\hat{D}_{1, 2}\setminus\mathcal{E}_{12, 3, 4}^{s_1, s_2, s_3, s_4}}\right)ds_4
\right|\notag\\
&\leq K_{11}(N\log N)^2t^2
\end{align}
for any $0\leq t\leq 1$ and $N\geq 5$ according to an argument similar to those leading to \eqref{equ 3.20 b} and \eqref{equ 3.20 c}.

Since $\hat{D}_{1, 2}$ is measurable with
\[
\left\{\left(\hat{\mathcal{X}}_{t, 1}^{\textbf{0}, 0}, \hat{\mathcal{X}}_{t, 2}^{x, 0}, \hat{\mathcal{X}}_{t, 1}^{\textbf{0}, s_1-s_2}, \hat{\mathcal{X}}_{t, 2}^{x, s_1-s_2}\right)\right\}_{t\geq 0},
\]
we have
\begin{align*}
&\tilde{\mathbb{E}}\left(\tilde{R}^{s_1, s_2, s_3, s_4, r}1_{\hat{D}_{1, 2}}\right)\\
&=\tilde{\mathbb{E}}\left(\mathcal{U}^{s_1, s_2, s_3, s_4, r}1_{\hat{D}_{1, 2}}\right)\tilde{\mathbb{E}}\left(\mathbb{E}_{\mu_p}\left(\eta_0(\hat{\mathcal{X}}_{s_1+r, 1}^{\textbf{0}, s_1-s_3})\eta_{0}(\hat{\mathcal{X}}^{x, s_1-s_3}_{s_1+r, 2})-\alpha_{\textbf{0}, x}(p)-\mathbb{E}_{\mu_p}F(\eta_{s_3+r})\right)\right)\\
&\text{\quad}\times\tilde{\mathbb{E}}\left(\mathbb{E}_{\mu_p}\left(\eta_0(\hat{\mathcal{X}}_{s_1+r, 1}^{\textbf{0}, s_1-s_4})\eta_{0}(\hat{\mathcal{X}}^{x, s_1-s_4}_{s_1+r, 2})-\alpha_{\textbf{0}, x}(p)-\mathbb{E}_{\mu_p}F(\eta_{s_4+r})\right)\right).
\end{align*}
According to Proposition \ref{Proposition A.1 dual},
\begin{align*}
&\tilde{\mathbb{E}}\left(\mathbb{E}_{\mu_p}\left(\eta_0(\hat{\mathcal{X}}_{s_1+r, 1}^{\textbf{0}, s_1-s_3})\eta_{0}(\hat{\mathcal{X}}^{x, s_1-s_3}_{s_1+r, 2})-\alpha_{\textbf{0}, x}(p)-\mathbb{E}_{\mu_p}F(\eta_{s_3+r})\right)\right)\\
&=\tilde{\mathbb{E}}\left(\mathbb{E}_{\mu_p}\left(\eta_0(\hat{\mathcal{X}}_{s_1+r, 1}^{\textbf{0}, s_1-s_3})\eta_{0}(\hat{\mathcal{X}}^{x, s_1-s_3}_{s_1+r, 2})-\mathbb{E}_{\mu_p}\left(\eta_{s_3+r}(\textbf{0})\eta_{s_3+r}(x)\right)\right)\right)\\
&=\mathbb{E}_{\mu_p}\left(\eta_{s_3+r}(\textbf{0})\eta_{s_3+r}(x)\right)-\mathbb{E}_{\mu_p}\left(\eta_{s_3+r}(\textbf{0})\eta_{s_3+r}(x)\right)=0.
\end{align*}
Therefore,
\begin{equation}\label{equ 3.24 c}
\tilde{\mathbb{E}}\left(\tilde{R}^{s_1, s_2, s_3, s_4, r}1_{\hat{D}_{1, 2}}\right)=0.
\end{equation}
The existence of $K_8$ follows from \eqref{equ 3.24 b} and \eqref{equ 3.24 c}.

On the event $\tilde{\mathcal{E}}$,
\[
\left(\mathcal{X}_{s_1+r, 2i-1}^{\textbf{0}, s_1-s_i}, \mathcal{X}_{s_1+r, 2i}^{x, s_1-s_i}\right)
=\left(\hat{\mathcal{X}}_{s_1+r, 1}^{\textbf{0}, s_1-s_i}, \hat{\mathcal{X}}_{s_1+r, 2}^{x, s_1-s_i}\right)
\]
for $i=1, 2, 3, 4$. So, on the event $\tilde{\mathcal{E}}$,
\[
\tilde{V}^{s_1, s_2, s_3, s_4, r}=\tilde{W}^{s_1, s_2, s_3, s_4, r},
\]
where
\[
\tilde{W}^{s_1, s_2, s_3, s_4, r}=\prod_{i=1}^4\mathbb{E}_{\mu_p}\left(\eta_0(\hat{\mathcal{X}}_{s_1+r, 1}^{\textbf{0}, s_1-s_i})\eta_{0}(\hat{\mathcal{X}}^{x, s_1-s_i}_{s_1+r, 2})-\alpha_{\textbf{0}, x}(p)-\mathbb{E}_{\mu_p}F(\eta_{s_i+r})\right).
\]
According to an argument similar to that leading to \eqref{equ 3.24 c}, we have
\begin{equation}\label{equ 3.24 d}
\tilde{\mathbb{E}}\left(\tilde{W}^{s_1, s_2, s_3, s_4, r}\right)=0.
\end{equation}
Since
\[
(\tilde{\mathcal{E}})^c\subseteq \bigcup_{i, j, k}\mathcal{E}_{i, j, k}\bigcup \bigcup_{i, j}\mathcal{E}_{ij, \hat{i}\hat{j}}\bigcup_{i, j}\mathcal{E}_{ij, \hat{i}, \hat{j}},
\]
there exists $K_{12}<+\infty$ independent of $t, N, r$ such that
\begin{align}\label{equ 3.24 f}
&\left|\int_0^{tN}ds_1\int_0^{s_1}ds_2\int_0^{s_2}ds_3\int_0^{s_3}\tilde{\mathbb{E}}\left(\tilde{W}^{s_1, s_2, s_3, s_4, r}1_{(\tilde{\mathcal{E}}^{s_1, s_2, s_3, s_4})^c}\right)ds_4
\right|\notag\\
&\leq K_{12}(N\log N)^2t^2
\end{align}
for any $0\leq t\leq 1$ and $N\geq 5$ according to an argument similar to those leading to \eqref{equ 3.20 b}, \eqref{equ 3.20 c} and \eqref{equ 3.20 d}. The existence of $K_9$ follows from \eqref{equ 3.24 d} and \eqref{equ 3.24 f}.
\qed

At last, we prove Theorem \ref{theorem 2.1 main theorem d=4}.

\proof[Proof of Theorem \ref{theorem 2.1 main theorem d=4}]
Since
\[
\tilde{\mathcal{E}}\bigcup\bigcup_{i, j}\mathcal{E}_{ij, \hat{i}, \hat{j}}\bigcup\bigcup_{i, j}\mathcal{E}_{ij, \hat{i}\hat{j}}
\bigcup\bigcup_{i, j, k}\mathcal{E}_{i, j, k}
\]
is the universal set, by Lemma \ref{lemma 3.10}, there exists $K_{10}<+\infty$ independent of $t, N, r$ such that
\[
\int_{0}^{tN}ds_1\int_{0}^{s_1}ds_2\int_{0}^{s_2}ds_3\int_0^{s_3}\mathbb{E}_{\mu_p}\left(\prod_{i=1}^4\left(F(\eta_{s_i+r})
-\mathbb{E}_{\mu_p}F(\eta_{s_i+r})\right)\right)ds_4
\leq K_{10}t^2(N\log N)^2
\]
for any $0\leq t\leq 1, N\geq 5$ and $r\geq 0$. Then, according to the definition of $\nu_p$ and the invariance of $\nu_p$,
\begin{equation}\label{equ 3.25}
\mathbb{E}_{\nu_p}\left(\left(\frac{\xi_{tN}^{\textbf{0}, x}-tN\alpha_{\textbf{0}, x}(p)}{\sqrt{N\log N}}-\frac{\xi_{sN}^{\textbf{0}, x}-sN\alpha_{\textbf{0}, x}(p)}{\sqrt{N\log N}}\right)^4\right)\leq 24K_{10}(t-s)^2
\end{equation}
for any $N\geq 5$ and $0\leq s\leq t\leq 1$. According to Equation \eqref{equ 3.25} and Corollary 14.9 of \cite{Kallenberg1997}, the processes
\[
\left\{\frac{\xi_{tN}^{\textbf{0}, x}-tN\alpha_{\textbf{0}, x}(p)}{\sqrt{N\log N}}:~0\leq t\leq T\right\}_{N\geq 1}
\]
are tight under the uniform topology of $C[0, T]$. Theorem \ref{theorem 2.1 main theorem d=4} follows from decomposition \eqref{equ 3.2 martingale decomposition d=4}, Lemmas \ref{lemma 3.1}, \ref{lemma 3.2 martingale clt d=4} and the aforesaid tightness.
\qed

\section{Proof of Theorem \ref{theorem 2.3 main theorem d=3}}\label{section four}
In this Section, we prove Theorem \ref{theorem 2.3 main theorem d=3}. As we have explained in Section \ref{section two}, our proof of Theorem \ref{theorem 2.3 main theorem d=3} utilizes a quadratic fluctuation limit strategy. So we first introduce the quadratic fluctuation density field. For any $N\geq 1$ and $t\geq 0$, the quadratic fluctuation density field $\mathcal{Y}_t^N$ is defined as
\[
\mathcal{Y}_t^N(du, dv)=\frac{1}{N^{\frac{11}{4}}}\sum_{z_1\in \mathbb{Z}^3}\sum_{z_2\in \mathbb{Z}^3}\left(\eta_{tN}(z_1)\eta_{tN}(z_2)-\alpha_{z_1, z_2}(p)\right)
\delta_{z_1/\sqrt{N}}(du)\delta_{z_2/\sqrt{N}}(dv),
\]
where $\delta_a(du)$ is the Dirac measure concentrated on $a$ for any $a\in \mathbb{R}^3$. The quadratic fluctuation density field $\mathcal{Y}_t^N$ can be considered as a random tempered distribution on $\mathbb{R}^6$, i.e.,
\[
\mathcal{Y}_t^N(f)=\frac{1}{N^{\frac{11}{4}}}\sum_{z_1\in \mathbb{Z}^3}\sum_{z_2\in \mathbb{Z}^3}\left(\eta_{tN}(z_1)\eta_{tN}(z_2)-\alpha_{z_1, z_2}(p)\right)
f\left(\frac{z_1}{\sqrt{N}}, \frac{z_2}{\sqrt{N}}\right)
\]
for any Schwartz function $f: \mathbb{R}^6\rightarrow \mathbb{R}$.

To give our quadratic fluctuation theorem, i.e., the CLT of $\{\mathcal{Y}^N\}_{N\geq 1}$, we introduce the following generalized Ornstein-Uhlenbeck (O-U) process
\begin{equation}\label{equ 4.1 generalized O-U}
d\mathcal{Y}_t=\Delta\mathcal{Y}_tdt+Cd\mathcal{B}_t
\end{equation}
for some nonzero constant $C$, where $\Delta$ is the Laplace operator on $\mathbb{R}^6$ such that $\Delta f(u)=\sum_{j=1}^6\partial^2_{u_ju_j}f(u)$ and $(\Delta \mathcal{A})(f)=\mathcal{A}(\Delta f)$
for any Schwartz function $f: \mathbb{R}^6\rightarrow \mathbb{R}$ and tempered distribution $\mathcal{A}$ on $\mathbb{R}^6$ and $\{\mathcal{B}_t\}_{t\geq 0}$ is a Gaussian time-space noise such that $\{\mathcal{B}_t(f)\}_{t\geq 0}$ is a copy of
\[
\left\{\sqrt{\int_{\mathbb{R}^3}\left(\int_{\mathbb{R}^3}f(u, v)+f(v, u)dv\right)^2du}B_t\right\}_{t\geq 0}
\]
for any Schwartz function $f: \mathbb{R}^6\rightarrow \mathbb{R}$, where $\{B_t\}_{t\geq 0}$ is the standard real-valued Brownian motion defined as in Section \ref{section two}.

According to the theory of the generalized O-U process established in \cite{Holley1978}, the solution to Equation \eqref{equ 4.1 generalized O-U} has two equivalent rigorous definitions. The first definition relates Equation \eqref{equ 4.1 generalized O-U} to a martingale problem. In detail, when the distribution of $\mathcal{Y}_0$ is given, according to Theorem 1.4 of \cite{Holley1978} there exists a unique tempered-distribution-valued random process $\{\mathcal{Y}_t\}_{t\geq 0}$ such that
\begin{align*}
\left\{\mathcal{Y}_t(f)-\mathcal{Y}_0(f)-\int_0^t\mathcal{Y}_s(\Delta f)ds\right\}_{t\geq 0}
%-\frac{C^2}{2}[f, f]\int_0^tG^{\prime\prime}(\mathcal{Y}_s(f))ds\Bigg\}_{t\geq 0}
\end{align*}
is a martingale with quadratic variation process $\{C^2[f, f]t\}_{t\geq 0}$ for any schwartz function $f:\mathbb{R}^6\rightarrow \mathbb{R}$, where
\[
\left[f, g\right]=\int_{\mathbb{R}^3}\left(\int_{\mathbb{R}^3}\left(f(u, v)+f(v, u)\right)dv\int_{\mathbb{R}^3}\left(g(u, v)+g(v, u)\right)dv\right)du
\]
for any Schwartz functions $f, g: \mathbb{R}^6\rightarrow \mathbb{R}$. According to It\^{o}'s formula, it is natural to define the above unique $\{\mathcal{Y}_t\}_{t\geq 0}$ as the solution to \eqref{equ 4.1 generalized O-U}.

The second definition is via solving \eqref{equ 4.1 generalized O-U} directly by the integrating factor method. In detail, let $\{\mathcal{P}_t\}_{t\geq 0}$ be the Markov semigroup with respect to $\Delta$, i.e,
\[
\mathcal{P}_tf(u)=\mathbb{E}f(\mathcal{W}_{2t}^6+u)\text{~and~}(\mathcal{P}_t\mathcal{A})f=\mathcal{A}(\mathcal{P}_tf)
\]
for any $u\in \mathbb{R}^6$, Schwartz function $f: \mathbb{R}^6\rightarrow \mathbb{R}$ and tempered distribution $\mathcal{A}$ on $\mathbb{R}^6$, where $\{\mathcal{W}_t^d\}_{t\geq 0}$ is the $\mathbb{R}^6$-valued standard Brownian motion starting from the origin defined as in Section \ref{section two}. Then, when the distribution of $\mathcal{Y}_0$ is given, the solution $\{\mathcal{Y}_t\}_{t\geq 0}$ to Equation \eqref{equ 4.1 generalized O-U} can be directly defined as
\begin{equation}\label{equ 4.2 solving O-U by integrating factor}
\mathcal{Y}_t=\mathcal{P}_t\mathcal{Y}_0+C\int_0^t\mathcal{P}_{t-s}d\mathcal{B}_s,
\end{equation}
where $\mathcal{Y}_0$ is independent of $\{\mathcal{B}_t\}_{t\geq 0}$. According to the argument given in the proof of Theorem 1.4 of \cite{Holley1978}, the above two definitions of the solution to \eqref{equ 4.1 generalized O-U} are equivalent.

By \eqref{equ 4.2 solving O-U by integrating factor}, for any Schwartz functions $f, g: \mathbb{R}^6\rightarrow \mathbb{R}$ and $0\leq s\leq t$,
\[
{\rm Cov}(\mathcal{Y}_t(f), \mathcal{Y}_s(g))
={\rm Cov}(\mathcal{Y}_0(\mathcal{P}_tf), \mathcal{Y}_0(\mathcal{P}_sf))+C^2\int_0^s\left[\mathcal{P}_{t-r}f, \mathcal{P}_{s-r}g\right]dr.
\]
Then, since the generator of a Markov process is a negative operator, the process $\{\mathcal{Y}_t\}_{t\geq 0}$ has a unique invariant distribution $\Lambda$ such that $\mathcal{A}(f)$ follows the normal distribution with mean zero and variance
\[
C^2\int_0^{+\infty}\left[\mathcal{P}_rf, \mathcal{P}_rf\right]dr
\]
for any Schwartz functions $f:\mathbb{R}^6\rightarrow \mathbb{R}$ and random tempered distribution $\mathcal{A}$ which is distributed with $\Lambda$. We write $\Lambda$ as $\Lambda_C$ when we need to emphasize the dependence of $\Lambda$ on $C$.

Now we give our quadratic fluctuation limit theorem.
\begin{lemma}\label{Lemma 4.1 quadratic fluctuation limit}
Let $p\in (0, 1)$ and $\eta_0$ be distributed with $\nu_p$. For any $T>0$, the random process
\[
\left\{\mathcal{Y}_t^N:~0\leq t\leq T\right\}
\]
converges weakly, with respect to the Skorohod topology, to $\{\mathcal{Y}_t\}_{0\leq t\leq T}$ as $N\rightarrow+\infty$, where $\{\mathcal{Y}_t\}_{t\geq 0}$ is the solution to Equation \eqref{equ 4.1 generalized O-U} with $C=\sqrt{6\chi(p)p^2}$ and $\mathcal{Y}_0$ being distributed with $\Lambda_C$.
\end{lemma}

According to Theorem 2 of \cite{Presutti1983}, when $\eta_0$ is distributed with $\nu_p$, the fluctuation density field process
\[
\left\{\frac{1}{N^{5/4}}\sum_{z\in \mathbb{Z}^3}\left(\eta_{tN}(z)-p\right)\delta_{z/\sqrt{N}}(du)\right\}_{t\geq 0}
\]
converges weakly to a generalized O-U process as $N\rightarrow+\infty$. Lemma \ref{Lemma 4.1 quadratic fluctuation limit} extends this fluctuation limit theorem to a quadratic version.

The proof of Lemma \ref{Lemma 4.1 quadratic fluctuation limit} utilizes a Dynkin's martingale strategy, which we give later. Now we utilize Lemma \ref{Lemma 4.1 quadratic fluctuation limit} to prove Theorem \ref{theorem 2.3 main theorem d=3}. From now on, we let $\{\mathcal{Y}_t\}_{t\geq 0}$ be defined as in Lemma \ref{Lemma 4.1 quadratic fluctuation limit}. Then, for any $0<t_1<t_2<\ldots<t_m$ and $f_1, \ldots, f_m\in L^2(\mathbb{R}^6)$,
\[
\left(\mathcal{Y}_{t_1}(f_1), \mathcal{Y}_{t_2}(f_2), \ldots, \mathcal{Y}_{t_m}(f_m)\right)
\]
is a $\mathbb{R}^m$-valued Gaussian random variable such that $\mathbb{E}\mathcal{Y}_{t_i}(f_i)=0$ and
\begin{equation}\label{equ 4.3}
{\rm Cov}\left(\mathcal{Y}_{t_i}(f_i), \mathcal{Y}_{t_j}(f_j)\right)=6\chi(p)p^2\int_0^{+\infty}\left[\mathcal{P}_{r+t_j-t_i}f_j, \mathcal{P}_rf_i\right]dr
\end{equation}
for any $1\leq i\leq j\leq m$. Let $u\in \mathbb{R}^3$ such that $u\neq \textbf{0}$. For any $\epsilon>0$, we define $\hat{H}_\epsilon: \mathbb{R}^6\rightarrow \mathbb{R}$ as
\[
\hat{H}_\epsilon(v_1, v_2)=\frac{1}{(2\epsilon)^6}1_{\{\|v_1\|_1\leq \epsilon\}}1_{\{\|v_2-u\|_1\leq \epsilon\}}
\]
for any $v_1, v_2\in \mathbb{R}^3$, where $1_A$ is the indicator function of the event $A$. We have the following lemma.
\begin{lemma}\label{lemma 4.2 covariance limit}
For any $0<s\leq t$,
\[
\lim_{\epsilon\rightarrow 0}{\rm Cov}\left(\int_0^s\mathcal{Y}_r(\hat{H}_\epsilon)dr, \int_0^t\mathcal{Y}_r(\hat{H}_\epsilon)dr\right)=C_3^2{\rm Cov}(\zeta_s^u, \zeta_t^u),
\]
where $C_3, \zeta_t^u$ are defined as in Theorem \ref{theorem 2.3 main theorem d=3}.
\end{lemma}
\proof[Proof of Lemma \ref{lemma 4.2 covariance limit}]
According to the definition of $\mathcal{W}_t^d$, $\mathcal{W}_t^6=(\mathcal{W}_t^3, \tilde{\mathcal{W}}_t^3)$, where $\{\tilde{\mathcal{W}}_t^3\}_{t\geq 0}$ is an independent copy of $\{\mathcal{W}_t\}_{t\geq 0}$. According to the definition of $\hat{H}_\epsilon$, for any $v_1, v_2\in \mathbb{R}^3$ and $r\geq 0$,
\begin{align*}
\mathcal{P}_r\hat{H}_\epsilon(v_1, v_2)
=\frac{1}{(2\epsilon)^3}\mathbb{P}\left(\|\mathcal{W}_{2r}^3+v_1\|_1\leq \epsilon\right)\frac{1}{(2\epsilon)^3}\mathbb{P}\left(\|\tilde{\mathcal{W}}_{2r}^3+v_2-u\|_1\leq \epsilon\right).
\end{align*}
As a result,
\begin{equation}\label{equ 4.4}
\lim_{\epsilon\rightarrow 0}\mathcal{P}_r\hat{H}_\epsilon(v_1, v_2)=f_{2r}^3(v_1)f_{2r}^3(v_2-u).
\end{equation}
By Kolmogorov-Chapman equation, for any $r_1, r_2>0$ and $u\in \mathbb{R}^3$,
\[
\int_{\mathbb{R}^3}f_{r_1}^3(v_1)f_{r_2}^3(v_1-u)dv_1=f^3_{r_1+r_2}(u).
\]
Then, according to \eqref{equ 4.4} and the fact that $\int_{\mathbb{R}^3}f_r^3(v)dv=1$,
\begin{equation}\label{equ 4.5}
\lim_{\epsilon\rightarrow0}\left[\mathcal{P}_{r+r_2-r_1}\hat{H}_\epsilon, \mathcal{P}_r\hat{H}_\epsilon\right]
=2f^3_{2(2r+r_2-r_1)}(\textbf{0})+2f^3_{2(2r+r_2-r_1)}(u)
\end{equation}
for any $0\leq r_1\leq r_2$. According to the definition of $C_3$ and $\zeta_t^u$, Lemma \ref{lemma 4.2 covariance limit} follows from \eqref{equ 4.3} and \eqref{equ 4.5}.
\qed

Now we utilized Lemma \ref{lemma 4.2 covariance limit} to prove the following finite dimensional distribution CLT.
\begin{lemma}\label{lemma 4.3 finite dimensional distribution CLT d=3}
For any $m\geq 1$ and $0<t_1<t_2<\ldots<t_m$, the $\mathbb{R}^m$-valued random variable
\[
\frac{1}{N^{\frac{3}{4}}}\left(\xi_{t_1N}^{\textbf{0}, \sqrt{N}u_N}-t_1N\alpha_{\textbf{0}, \sqrt{N}u_N}(p)
,\ldots, \xi_{t_mN}^{\textbf{0}, \sqrt{N}u_N}-t_mN\alpha_{\textbf{0}, \sqrt{N}u_N}(p)\right)
\]
converges weakly to $C_3\left(\zeta_{t_1}^u, \ldots, \zeta_{t_m}^u\right)$ as $N\rightarrow+\infty$.
\end{lemma}

\proof[Proof of Lemma \ref{lemma 4.3 finite dimensional distribution CLT d=3}]
By Lemma \ref{Lemma 4.1 quadratic fluctuation limit}, the $\mathbb{R}^m$-valued random variable
\[
\left(\int_0^{t_1}\mathcal{Y}_r^N(\hat{H}_\epsilon)dr, \ldots, \int_0^{t_m}\mathcal{Y}_r^N(\hat{H}_\epsilon)dr\right)
\]
converges weakly to the $\mathbb{R}^m$-valued Gaussian random variable
\[
\left(\int_0^{t_1}\mathcal{Y}_r(\hat{H}_\epsilon)dr, \ldots, \int_0^{t_m}\mathcal{Y}_r(\hat{H}_\epsilon)dr\right)
\]
as $N\rightarrow+\infty$. Then, by Lemma \ref{lemma 4.2 covariance limit}, to complete this proof we only need to show that
\begin{equation}\label{equ 4.6 replace formula}
\limsup_{\epsilon\rightarrow 0}\limsup_{N\rightarrow+\infty}\mathbb{P}_{\nu_p}\left(\left|\int_0^{t}\mathcal{Y}_r^N(\hat{H}_\epsilon)dr
-\frac{1}{N^{\frac{3}{4}}}\left(\xi_{tN}^{\textbf{0}, \sqrt{N}u_N}-tN\alpha_{\textbf{0}, \sqrt{N}u_N}(p)\right)\right|>\delta\right)=0
\end{equation}
for any $\delta>0$ and $t\geq 0$. According to Chebyshev's inequality, to verify \eqref{equ 4.6 replace formula}, we only need to show that
\begin{equation}\label{equ 4.7}
\limsup_{\epsilon\rightarrow 0}\limsup_{N\rightarrow+\infty}{\rm Var}_{\nu_p}\left(\int_0^{t}\mathcal{Y}_r^N(\hat{H}_\epsilon)dr
-\frac{1}{N^{\frac{3}{4}}}\xi_{tN}^{\textbf{0}, \sqrt{N}u_N}\right)=0
\end{equation}
for any $t\geq 0$. We claim that
\begin{align}\label{equ 4.8}
&\limsup_{N\rightarrow+\infty}{\rm Var}_{\nu_p}\left(\int_0^{t}\mathcal{Y}_r^N(\hat{H}_\epsilon)dr
-\frac{1}{N^{\frac{3}{4}}}\xi_{tN}^{\textbf{0}, \sqrt{N}u_N}\right)\notag\\
&=\frac{2p^3(1-p)K_1}{(2\epsilon)^{12}}\int_0^td\theta_1\int_0^{\theta_1}\left(\int_{D(\epsilon)}\mathcal{J}_{\theta_2}(v_1, v_2, v_3, v_4)dv_1dv_2dv_3dv_4\right)d\theta_2
\end{align}
for $\epsilon$ sufficiently small such that $3\epsilon<\|u\|_1$, where $K_1=\left(\int_0^{+\infty}p_{s, 3}(\textbf{0}, \textbf{0})ds\right)^{-1}$ and $D(\epsilon)\subseteq (\mathbb{R}^3)^4$ is defined as
\[
D(\epsilon)=
\left\{(v_1, v_2, v_3, v_4)\in (\mathbb{R}^3)^4:~\|v_1\|_1\leq \epsilon, \|v_2-u\|_1\leq \epsilon, \|v_3\|_1\leq \epsilon,
\|v_4-u\|_1\leq \epsilon\right\}
\]
and $\mathcal{J}_{\theta_2}: (\mathbb{R}^3)^4\rightarrow \mathbb{R}$ is defined as
\begin{align*}
\mathcal{J}_{\theta_2}(v_1, v_2, v_3, v_4)
=&\sum_{i=1}^2\sum_{j=3}^4\int_0^{+\infty}f^3_{2(\theta_2+r)}(v_i-v_j)dr-2\sum_{j=3}^4\int_0^{+\infty}f_{2(\theta_2+r)}^3(-v_j)dr\\
&-2\sum_{j=3}^4\int_0^{+\infty}f_{2(\theta_2+r)}^3(u-v_j)dr+2\int_0^{+\infty}f_{2(\theta_2+r)}^3(0)dr\\
&+2\int_0^{+\infty}f_{2(\theta_2+r)}^3(-u)dr.
\end{align*}
Since $f_t^3$ is continuous on $\mathbb{R}^3$, Equation \eqref{equ 4.7} follows from \eqref{equ 4.8}.

So we need to verify \eqref{equ 4.8} to complete the proof of Lemma \ref{lemma 4.3 finite dimensional distribution CLT d=3}. The proof of \eqref{equ 4.8} utilizes an argument similar to that leading to \eqref{equ 3.8}, where the duality relationship between the voter model and the coalescing random walk plays the key role. To avoid too many complex details in the main body of this paper, we give the proof of \eqref{equ 4.8} in Appendix \ref{Appendix A.6}.
\qed

Now we prove Theorem \ref{theorem 2.3 main theorem d=3}.
\proof[Proof of Theorem \ref{theorem 2.3 main theorem d=3}]
For any $0\leq s\leq t$, according to the invariance of $\nu_p$ and the bilinear property of the covariance,
\begin{align*}
&\mathbb{E}_{\nu_p}\left(\left(\frac{1}{N^{\frac{3}{4}}}\left(\xi_{tN}^{\textbf{0}, \sqrt{N}u_N}-tN\alpha_{\textbf{0}, \sqrt{N}u_N}(p)\right)-\frac{1}{N^{\frac{3}{4}}}\left(\xi_{sN}^{\textbf{0}, \sqrt{N}u_N}-sN\alpha_{\textbf{0}, \sqrt{N}u_N}(p)\right)\right)^2\right)\\
&=\frac{2}{N^{\frac{3}{2}}}\int_0^{(t-s)N}d\theta_1\int_0^{\theta_1}{\rm Cov}_{\nu_p}\left(\eta_{\theta_2}(\textbf{0})\eta_{\theta_2}(\sqrt{N}u_N), \eta_{0}(\textbf{0})\eta_{0}(\sqrt{N}u_N)\right)d\theta_2.
\end{align*}
Then, by Lemma \ref{lemma 3.7 time correlation decay},
\begin{align*}
&\mathbb{E}_{\nu_p}\left(\left(\frac{1}{N^{\frac{3}{4}}}\left(\xi_{tN}^{\textbf{0}, \sqrt{N}u_N}-tN\alpha_{\textbf{0}, \sqrt{N}u_N}(p)\right)-\frac{1}{N^{\frac{3}{4}}}\left(\xi_{sN}^{\textbf{0}, \sqrt{N}u_N}-sN\alpha_{\textbf{0}, \sqrt{N}u_N}(p)\right)\right)^2\right)\\
&\leq \frac{8}{N^{\frac{3}{2}}}K_1\int_0^{(t-s)N}d\theta_1\int_0^{\theta_1}d\theta_2\int_0^{+\infty}p_{r+\theta_2, 3}(\textbf{0}, \textbf{0})dr.
\end{align*}
According to the local central limit theorem of the simple random walk on $\mathbb{Z}^3$, there exists $K_2<+\infty$ such that $p_{t, 3}(\textbf{0}, \textbf{0})\leq K_2t^{-\frac{3}{2}}$ for all $t\geq 0$. Consequently, for any $0\leq s\leq t$,
\begin{align}\label{equ proof of Thm2.3 1}
&\mathbb{E}_{\nu_p}\left(\left(\frac{1}{N^{\frac{3}{4}}}\left(\xi_{tN}^{\textbf{0}, \sqrt{N}u_N}-tN\alpha_{\textbf{0}, \sqrt{N}u_N}(p)\right)-\frac{1}{N^{\frac{3}{4}}}\left(\xi_{sN}^{\textbf{0}, \sqrt{N}u_N}-sN\alpha_{\textbf{0}, \sqrt{N}u_N}(p)\right)\right)^2\right)\notag\\
&\leq K_3(t-s)^{\frac{3}{2}},
\end{align}
where $K_3=\frac{64}{3}K_1K_2$. By \eqref{equ proof of Thm2.3 1} and Corollary 14.9 of \cite{Kallenberg1997}, the random processes
\[
\left\{\frac{1}{N^{\frac{3}{4}}}\left(\xi_{tN}^{\textbf{0}, \sqrt{N}u_N}-tN\alpha_{\textbf{0}, \sqrt{N}u_N}(p)\right):~0\leq t\leq T\right\}_{N\geq 1}
\]
are tight under the uniform topology of $C[0, T]$. Theorem \ref{theorem 2.3 main theorem d=3} follows from Lemma \ref{lemma 4.3 finite dimensional distribution CLT d=3} and the aforesaid tightness.
\qed

At last, we prove Lemma \ref{Lemma 4.1 quadratic fluctuation limit} by several steps. For any $N\geq 1, t\geq 0$ and Schwartz function $f: \mathbb{R}^6\rightarrow \mathbb{R}$, we denote by $\tilde{\mathcal{M}}_t^{N, f}$ the term
\[
\mathcal{Y}_t^N(f)-\mathcal{Y}_0^N(f)-\int_0^tN\mathcal{L}\mathcal{Y}_s^N(f)ds.
\]
By Dynkin's martingale formula, $\{\tilde{\mathcal{M}}_t^{N, f}\}_{t\geq 0}$ is a martingale. According to the definition of $\mathcal{L}$,
\begin{align*}
N\mathcal{L}\mathcal{Y}_s^N(f)
=N\sum_{z\in \mathbb{Z}^3}\sum_{w\sim z}\Upsilon_s^{N, f}(z, w),
\end{align*}
where
\[
\Upsilon_s^{N, f}(z, w)=\frac{1}{N^{\frac{11}{4}}}\sum_{v_1}\sum_{v_2}
\left(\eta_{sN}^{z, w}(v_1)\eta_{sN}^{z, w}(v_2)-\eta_{sN}(v_1)\eta_{sN}(v_2)\right)f(\frac{v_1}{\sqrt{N}}, \frac{v_2}{\sqrt{N}}).
\]
%According to the fact that $\eta_{sN}^{z, w}(v_1)\eta_{sN}^{z, w}(v_2)-\eta_{sN}(v_1)\eta_{sN}(v_2)=0$ when $\{z\}\cap \{v_1, v_2\}=\emptyset$, it is easy to verify that
%\[
%\Upsilon_s^{N, f}(z, w)=O(N^{-5/4})\int_{\mathbb{R}^3}f(z/\sqrt{N}, u)+f(u, z/\sqrt{N})du+O(N^{-11/4})f(z/\sqrt{N}, z/\sqrt{N}).
%\]
%Then, according to the Taylor's expansion formula up to the second order,
%\begin{align}\label{equ 4.10}
%&N\mathcal{L}G(\mathcal{Y}_s^N(f))\\
%&=N\sum_{z}\sum_{w\sim z}G^\prime(\mathcal{Y}_s^N(f))\Upsilon_s^{N, f}(z, w)
%+\frac{N}{2}\sum_{z}\sum_{w\sim z}G^{\prime\prime}(\mathcal{Y}_s^N(f))\left(\Upsilon_s^{N, f}(z, w)\right)^2+O(N^{-5/4}).\notag
%\end{align}
According to Dynkin's martingale formula, the quadratic variation process $\{\langle \tilde{\mathcal{M}}^{N, f}\rangle_t\}_{t\geq 0}$ of $\{\tilde{\mathcal{M}}_t^{N, f}\}_{t\geq 0}$ is given by
\begin{equation}\label{equ 4.10 minus 1}
\langle \tilde{\mathcal{M}}^{N, f}\rangle_t=\int_0^tN\sum_{z\in \mathbb{Z}^3}\sum_{w\sim z}\left(\Upsilon_s^{N, f}(z, w)\right)^2ds.
\end{equation}
For $\{\langle \tilde{\mathcal{M}}^{N, f}\rangle_t\}_{t\geq 0}$, we have the following lemma.
\begin{lemma}\label{lemma 4.4}
For any $t\geq 0$,
\[
\lim_{N\rightarrow+\infty}\langle \tilde{\mathcal{M}}^{N, f}\rangle_t
=6\chi(p)p^2[f, f]t
\]
in $L^2$.
\end{lemma}

\proof[Proof of Lemma \ref{lemma 4.4}]
According to the definition of $\Upsilon_s^{N, f}(z, w)$,
\begin{align}\label{equ 4.10}
\Upsilon_s^{N, f}(z, w)
=&\frac{\eta_{sN}(w)-\eta_{sN}(z)}{N^{\frac{11}{4}}}
\Bigg(\sum_{v_2\neq z}\eta_{sN}(v_2)f(\frac{z}{\sqrt{N}}, \frac{v_2}{\sqrt{N}})+\sum_{v_1\neq z}\eta_{sN}(v_1)f(\frac{v_1}{\sqrt{N}}, \frac{z}{\sqrt{N}})\notag\\
&+f(\frac{z}{\sqrt{N}},\frac{z}{\sqrt{N}})\Bigg).
\end{align}
According to the invariance of $\nu_p$, a $d=3$ version of Lemma \ref{Lemma 3.3 correlation decay d=4}
and the facts that each $z\in \mathbb{Z}^3$ has $6$ neighbors and
\[
\lim_{N\rightarrow+\infty}\frac{1}{N^{\frac{9}{2}}}\sum_{z_1, v_1, v_2\in \mathbb{Z}^3}h(z_1/\sqrt{N}, v_1/\sqrt{N}, v_2/\sqrt{N})
=\int_{\mathbb{R}^3}\int_{\mathbb{R}^3}\int_{\mathbb{R}^3}h(u_1, u_2, u_3)du_1du_2du_3
\]
for any integrable $h:\mathbb{R}^9\rightarrow \mathbb{R}$, we have
\begin{equation}\label{equ 4.11}
\lim_{N\rightarrow+\infty}\mathbb{E}_{\nu_p}\left(N\sum_{z}\sum_{w\sim z}\left(\Upsilon_s^{N, f}(z, w)\right)^2
\right)=6\chi(p)p^2[f, f]
\end{equation}
for any $s\geq 0$. Note that the proof of Lemma \ref{Lemma 3.3 correlation decay d=4} is still applicable in the case $d=3$
since the simple random walk on $\mathbb{Z}^3$ is transience.

According to an argument similar to that leading to \eqref{equ 3.7 QV variance d=4}, where Lemma \ref{lemma 3.7 time correlation decay} plays the key role, we have
\begin{equation}\label{equ 4.12}
\lim_{N\rightarrow+\infty}{\rm Var}_{\nu_p}\left(\int_0^tN\sum_{z}\sum_{w\sim z}\left(\Upsilon_s^{N, f}(z, w)\right)^2ds\right)
=0.
\end{equation}
Lemma \ref{lemma 4.4} follows from \eqref{equ 4.11}, \eqref{equ 4.12} and Chebyshev's inequality.
\qed

The following lemma gives the tightness of the quadratic fluctuation density fields.
\begin{lemma}\label{lemma 4.5 fluctuation field tightness}
For any $T>0$, the stochastic processes
\[
\left\{\mathcal{Y}_t^N:~0\leq t\leq T\right\}_{N\geq 1}
\]
are tight under the Skorohod topology.
\end{lemma}

\proof[Proof of Lemma \ref{lemma 4.5 fluctuation field tightness}]
By Aldous' criteria, we only need to verify that
\begin{equation}\label{equ 4.12 two}
\lim_{M\rightarrow+\infty}\limsup_{N\rightarrow+\infty}\mathbb{P}_{\nu_p}\left(\left|\mathcal{Y}_t^N(f)\right|\geq M\right)=0
\end{equation}
and
\begin{equation}\label{equ 4.12 three}
\lim_{\delta\rightarrow 0}\limsup_{N\rightarrow+\infty}\sup_{\upsilon\in \tilde{\mathcal{T}}^N, s\leq \delta}
\mathbb{P}_{\nu_p}\left(\left|\mathcal{Y}^N_{\upsilon+s}(f)-\mathcal{Y}^N_{\upsilon}(f)\right|>\epsilon\right)=0
\end{equation}
for any $t\geq 0, \epsilon>0$ and Schwartz function $f: \mathbb{R}^6\rightarrow \mathbb{R}$, where $\tilde{\mathcal{T}}^N$ is the set of stopping times of $\{\eta_{tN}\}_{t\geq 0}$ bounded by $T$.

We first verify \eqref{equ 4.12 two}. According to the definition of $\mathcal{Y}_t^N$,
\begin{align*}
\mathbb{E}_{\nu_p}\left(\left(\mathcal{Y}_t^N(f)\right)^2\right)
&=\frac{1}{N^{6}}
\sum_{v_1}\sum_{v_2}\sum_{w_1}\sum_{w_2}\\
&\text{\quad}\sqrt{N}{\rm Cov}_{\nu_p}\left(\eta_{tN}(v_1)\eta_{tN}(v_2), \eta_{tN}(w_1)\eta_{tN}(w_2)\right)f(\frac{v_1}{\sqrt{N}}, \frac{v_2}{\sqrt{N}})f(\frac{w_1}{\sqrt{N}}, \frac{w_2}{\sqrt{N}}).
\end{align*}
By \eqref{equ 3.5} and the invariance of $\nu_p$,
\[
|{\rm Cov}_{\nu_p}\left(\eta_{tN}(v_1)\eta_{tN}(v_2), \eta_{tN}(w_1)\eta_{tN}(w_2)\right)|
\leq \sum_{i=1}^2\sum_{j=1}^2\Phi(v_i-w_j).
\]
Then, according to the well known conclusion that $\Phi(x)=O(1)\|x\|_1^{2-3}=O(1)\|x\|_1^{-1}$ for the simple random walk on $\mathbb{Z}^3$ (see Proposition 1 in Section 26 of \cite{Spitzer1976}), we have
\begin{align}\label{equ 4.12 four}
\mathbb{E}_{\nu_p}\left(\left(\mathcal{Y}_t^N(f)\right)^2\right)
&\leq O(1)\sum_{i=1}^2\sum_{j=1}^2\int_{(\mathbb{R}^3)^4}\|u_i-\tilde{u}_j\|_1^{-1}|f(u_1, u_2)||f(\tilde{u}_1, \tilde{u}_2)|du_1du_2du_3du_4\notag\\
&=O(1)
\end{align}
for sufficiently large $N$. Equation \eqref{equ 4.12 two} follows from \eqref{equ 4.12 four} and Markov's inequality.

Now we verify \eqref{equ 4.12 three}. We only need to show that
\begin{equation}\label{equ 4.12 five}
\lim_{\delta\rightarrow 0}\limsup_{N\rightarrow+\infty}\sup_{\upsilon\in \tilde{\mathcal{T}}^N, s\leq \delta}
\mathbb{P}_{\nu_p}\left(\left|\tilde{\mathcal{M}}^{N,f}_{\upsilon+s}-\tilde{\mathcal{M}}^{N,f}_{\upsilon}\right|>\epsilon\right)=0
\end{equation}
and
\begin{equation}\label{equ 4.12 six}
\lim_{\delta\rightarrow 0}\limsup_{N\rightarrow+\infty}\sup_{\upsilon\in \tilde{\mathcal{T}}^N, s\leq \delta}
\mathbb{P}_{\nu_p}\left(\left|\int_{\upsilon}^{\upsilon+s}N\mathcal{L}\mathcal{Y}_r^N(f)dr\right|>\epsilon\right)=0.
\end{equation}
By \eqref{equ 4.10 minus 1},
\[
\mathbb{E}_{\nu_p}\left(\left(\tilde{\mathcal{M}}^{N,f}_{\upsilon+s}-\tilde{\mathcal{M}}^{N,f}_{\upsilon}\right)^2\right)
=\mathbb{E}_{\nu_p}\left(\int_{\upsilon}^{\upsilon+s}N\sum_{z\in \mathbb{Z}^3}\sum_{w\sim z}\left(\Upsilon_r^{N, f}(z, w)\right)^2dr\right).
\]
By \eqref{equ 4.10},
\[
|\Upsilon_r^{N, f}(z, w)|\leq \frac{2}{N^{\frac{11}{4}}}
\left(\sum_{v}\left|f(\frac{v}{\sqrt{N}}, \frac{z}{\sqrt{N}})\right|+\sum_{v}\left|f(\frac{z}{\sqrt{N}}, \frac{v}{\sqrt{N}})\right|\right)
\]
and hence
\begin{align*}
N\sum_{z\in \mathbb{Z}^3}\sum_{w\sim z}\left(\Upsilon_r^{N, f}(z, w)\right)^2
&=O(1)\int_{\mathbb{R}^3}\left(\int_{\mathbb{R}^3}|f(u_2, u_1)|+|f(u_1, u_2)|du_2\right)^2du_1\\
&=O(1)
\end{align*}
for sufficiently large $N$. As a result,
\[
\mathbb{E}_{\nu_p}\left(\left(\tilde{\mathcal{M}}^{N,f}_{\upsilon+s}-\tilde{\mathcal{M}}^{N,f}_{\upsilon}\right)^2\right)=O(1)\delta
\]
for $s\leq \delta$ and then \eqref{equ 4.12 five} follows from Markov's inequality.

By \eqref{equ 4.10}, for any Schwartz function $f: \mathbb{R}^6\rightarrow \mathbb{R}$,
\begin{align}\label{equ 4.13}
N\mathcal{L}\mathcal{Y}_s^N(f)
&=\frac{1}{N^{\frac{11}{4}}}\sum_{z}\sum_{v}\eta_{sN}(z)\eta_{sN}(v)\left(N\sum_{w\sim z}\left(f(\frac{w}{\sqrt{N}}, \frac{v}{\sqrt{N}})-f(\frac{z}{\sqrt{N}}, \frac{v}{\sqrt{N}})\right)\right)\notag\\
&\text{\quad}+\frac{1}{N^{\frac{11}{4}}}\sum_{z}\sum_{v}\eta_{sN}(z)\eta_{sN}(v)\left(N\sum_{w\sim v}\left(f(\frac{z}{\sqrt{N}}, \frac{w}{\sqrt{N}})-f(\frac{z}{\sqrt{N}}, \frac{v}{\sqrt{N}})\right)\right)\notag\\
&\text{\quad}+\hat{\varepsilon}_1^N,
\end{align}
where
\[
|\hat{\varepsilon}_1^N|\leq \frac{18}{N^{\frac{7}{4}}}\sum_{z}\left|f(\frac{z}{\sqrt{N}}, \frac{z}{\sqrt{N}})\right|=O(N^{-1/4}).
\]
We denote by $\Delta_N$ the discrete Laplace operator such that
\[
\Delta_Nf(u)=N\sum_{i=1}^{6}\left(f(u+\frac{e_i}{\sqrt{N}})+f(u-\frac{e_i}{\sqrt{N}})-2f(u)\right)
\]
for any $u\in \mathbb{R}^6$ and $f\in C^2(\mathbb{R}^6)$, where
\[
e_j=(0,~\ldots,~0,\mathop 1\limits_{j \text{th}},0,~\ldots,~0)
\]
for $j=1,2,\ldots,6$.
By \eqref{equ 4.13},
\[
N\mathcal{L}\mathcal{Y}_s^N(f)=\mathcal{Y}_s^N(\Delta_Nf)+\hat{\varepsilon}_1^N+\hat{\varepsilon}_2^N,
\]
where
\[
\hat{\varepsilon}_2^N=\frac{1}{N^{\frac{11}{4}}}\sum_{z}\sum_{v}\alpha_{z, v}(p)\Delta_Nf(\frac{z}{\sqrt{N}}, \frac{v}{\sqrt{N}}).
\]
According to the invariance of $\nu_p$ and the definition of $\alpha_{z, v}(p)$,
\[
\mathbb{E}_{\nu_p}N\mathcal{L}\mathcal{Y}_s^N(f)=\mathbb{E}_{\nu_p}\mathcal{Y}_s^N(\Delta_Nf)=0
\]
and consequently $\hat{\varepsilon}_2^N=O(N^{-1/4})$. Therefore,
\begin{equation}\label{equ 4.14}
N\mathcal{L}\mathcal{Y}_s^N(f)=\mathcal{Y}_s^N(\Delta_Nf)+\hat{\varepsilon}_3^N,
\end{equation}
where $|\hat{\varepsilon}_3^N|\leq O(1)N^{-1/4}$. By \eqref{equ 4.12 four} and \eqref{equ 4.14}, to verify \eqref{equ 4.12 six}, we only need to show that
\begin{equation}\label{equ 4.12 seven}
\lim_{\delta\rightarrow 0}\limsup_{N\rightarrow+\infty}\sup_{\upsilon\in \tilde{\mathcal{T}}^N, s\leq \delta}
\mathbb{P}_{\nu_p}\left(\left|\int_{\upsilon}^{\upsilon+s}\mathcal{Y}_r^N(\Delta f)dr\right|>\epsilon\right)=0.
\end{equation}
By Cauchy-Schwarz inequality, for $s\leq \delta$ and $\upsilon\in \tilde{\mathcal{T}}^N$,
\begin{align*}
&\mathbb{E}_{\nu_p}\left|\int_{\upsilon}^{\upsilon+s}\mathcal{Y}_r^N(\Delta f)dr\right|
=\mathbb{E}_{\nu_p}\left|\int_0^{T+\delta}1_{\{\upsilon\leq r\leq \upsilon+s\}}\mathcal{Y}_r^N(\Delta f)dr\right|\\
&\leq \mathbb{E}_{\nu_p}\left(\sqrt{\int_0^{T+\delta}1_{\{\upsilon\leq r\leq \upsilon+\delta\}}dr}\sqrt{\int_0^{T+\delta}\left(\mathcal{Y}_r^N(\Delta f)\right)^2dr}\right)\\
&\leq \sqrt{\mathbb{E}_{\nu_p}\int_0^{T+\delta}1_{\{\upsilon\leq r\leq \upsilon+\delta\}}dr}\sqrt{\mathbb{E}_{\nu_p}\int_0^{T+\delta}\left(\mathcal{Y}_r^N(\Delta f)\right)^2dr}.
\end{align*}
Then, by \eqref{equ 4.12 four},
\begin{align*}
&\mathbb{E}_{\nu_p}\left|\int_{\upsilon}^{\upsilon+s}\mathcal{Y}_r^N(\Delta f)dr\right|
\leq O(1)\sqrt{\mathbb{E}_{\nu_p}\int_0^{T+\delta}1_{\{\upsilon\leq r\leq \upsilon+\delta\}}dr}\\
&=O(1)\sqrt{\mathbb{E}_{\nu_p}\int_\upsilon^{\upsilon+s}1dr}\leq O(1)\sqrt{\delta}
\end{align*}
and hence \eqref{equ 4.12 seven} follows from Markov's inequality.
\qed

At last, we prove Lemma \ref{Lemma 4.1 quadratic fluctuation limit}
\proof[Proof of Lemma \ref{Lemma 4.1 quadratic fluctuation limit}]
Let $\{\hat{\mathcal{Y}}_t\}_{0\leq t\leq T}$ be a weak limit of a subsequence of $\{\mathcal{Y}_t^N:~0\leq t\leq T\}_{N\geq 1}$ under the Skorohod topology. Note that the existence of $\hat{\mathcal{Y}}$ follows from Lemma \ref{lemma 4.5 fluctuation field tightness}. For simplicity, we still denote the aforesaid subsequence by $\{\mathcal{Y}_t^N:~0\leq t\leq T\}_{N\geq 1}$.
By \eqref{equ 4.12 four} and \eqref{equ 4.14},  for any $t\geq 0$,  $\tilde{\mathcal{M}}_t^{N, f}$ converges weakly to
\[
\hat{\mathcal{Y}}_t(f)-\hat{\mathcal{Y}}_0(f)-\int_0^t\hat{\mathcal{Y}}_s(\Delta f)ds
\]
as $N\rightarrow+\infty$. Then, by Lemma \ref{lemma 4.4} and Theorem 1.4 in Chapter 7 of \cite{Ethier1986},
\[
\left\{\hat{\mathcal{Y}}_t(f)-\hat{\mathcal{Y}}_0(f)-\int_0^t\hat{\mathcal{Y}}_s(\Delta f)ds\right\}_{t\geq 0}
\]
is a martingale with quadratic variation process $\left\{6\chi(p)p^2[f, f]t\right\}_{t\geq 0}$. As a result, $\{\hat{\mathcal{Y}}_t\}_{t\geq 0}$ is a solution to Equation \eqref{equ 4.1 generalized O-U} with $C=\sqrt{6\chi(p)p^2}$. According to the invariance of $\nu_p$, $\{\hat{\mathcal{Y}}_t\}_{t\geq 0}$ starts from an invariant distribution and hence $\hat{\mathcal{Y}}_0$ is distributed with $\Lambda_C$. In conclusion, the proof is complete.
\qed

\appendix{}
\section{Appendix}
\subsection{Duality}\label{Appendix A.1}
In this appendix, we recall the duality relationship between the voter model and the coalescing random walk. For integer $m\geq 1$, $x_1, \ldots, x_m\in \mathbb{Z}^d$ and $t_1, t_2, \ldots, t_m\geq 0$, the coalescing random walk
\[
\left\{\left(\mathcal{X}_{t, 1}^{x_1, t_1}, \mathcal{X}_{t, 2}^{x_2, t_2}, \ldots, \mathcal{X}_{t, m}^{x_m, t_m}\right)\right\}_{t\geq 0}
\]
is defined as follows. For each $1\leq i\leq m$ and $0\leq t\leq t_i$, $\mathcal{X}_{t, i}^{x_i, t_i}=x_i$.  For each $1\leq i\leq m$ and $t\geq t_i$, $\{\mathcal{X}_{t+t_i}^{x_i, t_i}\}_{t\geq 0}$ is a copy of the simple random walk $\{X_t\}_{t\geq 0}$ starting from $x_i$ with generator given by \eqref{equ 2.1 generator of Random Walk}. So, we call $\mathcal{X}_{t, i}^{x_i, t_i}$ dormant when $t\leq t_i$ and call $\mathcal{X}_{t, i}^{x_i, t_i}$ active when $t>t_i$. All the active walks evolve independently until collisions occur. When two active walks collide with each other, they are coalesced into a single active walk. As a result, for any $B\subseteq \{1, 2, \ldots, m\}$, let
\[
A_t^B=\left\{\mathcal{X}_{t, i}^{x_i, t_i}:~i\in B\right\}\subseteq \mathbb{Z}^d
\]
and $\hat{t}_B=\max\{t_i:~i\in B\}$, then $\{A_{t+\hat{t}_B}^B\}_{t\geq 0}$ is a continuous-time Markov process with state space
\[
2^{\mathbb{Z}^d}:=\{A:~A\subseteq \mathbb{Z}^d\}
\]
and generator $\mathcal{Q}$ given by
\[
\mathcal{Q}f(A)=\sum_{x\in A}\sum_{y\sim x}\left(f\left(\left(A\setminus\{x\}\right)\cup \{y\}\right)-f(A)\right)
\]
for any $A\subseteq \mathbb{Z}^d$ and $f: 2^{\mathbb{Z}^d}\rightarrow \mathbb{R}$.

We denote by $\tilde{\mathbb{P}}$ the probability measure of the coalescing random walk and be $\tilde{\mathbb{E}}$ the expectation with respect to $\tilde{\mathbb{P}}$.

The following proposition, which is equivalent to Theorem 2.12 of \cite{Holley1975}, gives the duality relationship between the voter model and the coalescing random walk.

\begin{proposition}\label{Proposition A.1 dual}(Holley and Liggett, \cite{Holley1975})
For any integer $m\geq 1$, $x_1, \ldots, x_m\in \mathbb{Z}^d$, $\eta\in \{0, 1\}^{\mathbb{Z}^d}$, $t_1, \ldots, t_m\geq 0$, $s\geq \max\{t_1,\ldots, t_m \}$ and $H: \{0, 1\}^m\rightarrow \mathbb{R}$,
\begin{align*}
&\mathbb{E}_\eta H\left(\eta_{t_1}(x_1), \ldots, \eta_{t_m}(x_m)\right)\\
&=\tilde{\mathbb{E}}H\left(\eta\left(\mathcal{X}_{s, 1}^{x_1, s-t_1}\right), \eta\left(\mathcal{X}_{s, 2}^{x_2, s-t_2}\right), \ldots, \eta\left(\mathcal{X}_{s, m}^{x_m, s-t_m}\right)\right).
\end{align*}
\end{proposition}
Proposition \ref{Proposition A.1 dual} will be repeatedly utilized in this paper. Here we utilize Proposition \ref{Proposition A.1 dual} to verify Proposition \ref{Proposition 2.1 invariant distribution}. For any $x_1, \ldots, x_m\in \mathbb{Z}^d$, by Proposition \ref{Proposition A.1 dual},
\[
\mathbb{P}_\eta\left(\eta_t(x_i)=1\text{~for all~}1\leq i\leq m\right)
=\tilde{\mathbb{E}}\left(\prod_{i=1}^m1_{\{\eta\left(\mathcal{X}_{t, i}^{x_i, 0}\right)=1\}}\right),
\]
where $1_A$ is the indicator function of the event $A$. Hence,
\begin{align*}
\mathbb{P}_{\mu_p}\left(\eta_t(x_i)=1\text{~for all~}1\leq i\leq m\right)
&=\tilde{\mathbb{E}}\left(\mathbb{E}_{\mu_p}\left(\prod_{i=1}^m1_{\{\eta_0\left(\mathcal{X}_{t, i}^{x_i, 0}\right)=1\}}\right)\right)\\
&=\tilde{\mathbb{E}}\left(p^{{\rm card}\{\mathcal{X}_{t, 1}^{x_1, 0}, \ldots, \mathcal{X}_{t, m}^{x_m, 0}\}}\right),
\end{align*}
where ${\rm card}(A)$ the cardinality of the set $A$. Consequently, the existence of $\nu_p$ follows from the fact that ${\rm card}\{\mathcal{X}_{t, 1}^{x_1, 0}, \ldots, \mathcal{X}_{t, m}^{x_m, 0}\}$ is decreasing in $t$. Furthermore,
\begin{equation}\label{equ A.1}
\nu_p(\eta:~\eta(x_i)=1\text{~for all~}1\leq i\leq m)=\tilde{\mathbb{E}}p^{\lim_{t\rightarrow+\infty}{\rm card}\{\mathcal{X}_{t, 1}^{x_1, 0}, \ldots, \mathcal{X}_{t, m}^{x_m, 0}\}}.
\end{equation}
Applying \eqref{equ A.1} for $m=1$, we have
\[
\nu_p\left(\eta:~\eta(x)=1\right)=p
\]
since ${\rm card}\{\mathcal{X}_{t, 1}^{x, 0}\}\equiv 1$. Applying \eqref{equ A.1} for $m=2$, we have
\[
\nu_p\left(\eta:~\eta(x)=\eta(z)=1\right)=p\Phi(x-z)+p^2\left(1-\Phi(x-z)\right)
\]
since
\[
\tilde{\mathbb{P}}\left(\lim_{t\rightarrow+\infty}{\rm card}\{\mathcal{X}_{t, 1}^{x, 0}, \mathcal{X}_{t, 2}^{z, 0}\}=1\right)=\Phi(x-z)
\]
and
\[
\tilde{\mathbb{P}}\left(\lim_{t\rightarrow+\infty}{\rm card}\{\mathcal{X}_{t, 1}^{x, 0}, \mathcal{X}_{t, 2}^{z, 0}\}=2\right)=1-\Phi(x-z).
\]

\subsection{Proof of \eqref{equ 3.5}}\label{Appendix A.2}

\proof[Proof of \eqref{equ 3.5}]
According to the definition of $\nu_p$,
\[
{\rm Cov}_{\nu_p}\left(\eta_\theta(z_1)\eta_\theta(z_2), \eta_0(v_1)\eta_0(v_2)\right)=\lim_{t\rightarrow+\infty}{\rm Cov}_{\mu_p}\left(\eta_{\theta+t}(z_1)\eta_{\theta+t}(z_2), \eta_t(v_1)\eta_t(v_2)\right)
\]
By Proposition \ref{Proposition A.1 dual},
\begin{align*}
&{\rm Cov}_{\mu_p}\left(\eta_{\theta+t}(z_1)\eta_{\theta+t}(z_2), \eta_t(v_1)\eta_t(v_2)\right)\\
&=\mathbb{E}_{\mu_p}\left(\eta_{\theta+t}(z_1)\eta_{\theta+t}(z_2)\eta_t(v_1)\eta_t(v_2)\right)
-\mathbb{E}_{\mu_p}\left(\eta_{\theta+t}(z_1)\eta_{\theta+t}(z_2)\right)\mathbb{E}_{\mu_p}\left(\eta_t(v_1)\eta_t(v_2)\right)\\
&={\rm \uppercase\expandafter{\romannumeral1}}-{\rm \uppercase\expandafter{\romannumeral2}},
\end{align*}
where
\[
{\rm \uppercase\expandafter{\romannumeral1}}=\tilde{\mathbb{E}}
\left(\mathbb{E}_{\mu_p}\left(\prod_{i=1}^2\eta_0(\mathcal{X}_{t+\theta, i}^{z_i, 0})\prod_{j=3}^4\eta_0(\mathcal{X}_{t+\theta, j}^{v_{j-2}, \theta})\right)\right)
\]
and
\[
{\rm \uppercase\expandafter{\romannumeral2}}=\tilde{\mathbb{E}}
\left(\mathbb{E}_{\mu_p}\left(\prod_{i=1}^2\eta_0(\mathcal{X}_{t+\theta, i}^{z_i, 0})\right)\right)\tilde{\mathbb{E}}
\left(\mathbb{E}_{\mu_p}\left(\prod_{j=1}^2\eta_0(\mathcal{X}_{t+\theta, j}^{v_j, \theta})\right)\right).
\]
According to the definition of the coalescing random walk, we can further assume that $\left\{\left(\mathcal{X}_{t, 1}^{v_1, \theta}, \mathcal{X}_{t, 2}^{v_2, \theta}\right)\right\}_{t\geq 0}$ and
\[
\left\{\left(\mathcal{X}_{t, 1}^{z_1, 0}, \mathcal{X}_{t, 2}^{z_2, 0}, \mathcal{X}_{t, 3}^{v_1, \theta}, \mathcal{X}_{t, 4}^{v_2, \theta}\right)\right\}_{t\geq 0}
\]
are defined in the same probability space such that $\left\{\left(\mathcal{X}_{t, 1}^{v_1, \theta}, \mathcal{X}_{t, 2}^{v_2, \theta}\right)\right\}_{t\geq 0}$ are independent of $\left\{\left(\mathcal{X}_{t, 1}^{z_1, 0}, \mathcal{X}_{t, 2}^{z_2, 0}\right)\right\}_{t\geq 0}$ and
\[
\left(\mathcal{X}_{t, 1}^{v_1, \theta}, \mathcal{X}_{t, 2}^{v_2, \theta}\right)=\left(\mathcal{X}_{t, 3}^{v_1, \theta}, \mathcal{X}_{t, 4}^{v_2, \theta}\right)
\]
for $t\leq \hat{\beta}$, where
\[
\hat{\beta}=\inf\left\{t:~t\geq \theta, \left\{\mathcal{X}_{t, 1}^{z_1, 0}, \mathcal{X}_{t, 2}^{z_2, 0}\right\}\bigcap \left\{\mathcal{X}_{t, 1}^{v_1, \theta}, \mathcal{X}_{t, 2}^{v_2, \theta}\right\}\neq \emptyset\right\}.
\]
Then, according to our assumption,
\[
{\rm \uppercase\expandafter{\romannumeral2}}
=\tilde{\mathbb{E}}\left(\mathbb{E}_{\mu_p}\left(\prod_{i=1}^2\eta_0(\mathcal{X}_{t+\theta, i}^{z_i, 0})\right)\mathbb{E}_{\mu_p}\left(\prod_{j=1}^2\eta_0(\mathcal{X}_{t+\theta, j}^{v_j, \theta})\right)\right).
\]
On the event $\hat{\beta}>t+\theta$,
\begin{align*}
\mathbb{E}_{\mu_p}\left(\prod_{i=1}^2\eta_0(\mathcal{X}_{t+\theta, i}^{z_i, 0})\prod_{j=3}^4\eta_0(\mathcal{X}_{t+\theta, j}^{v_{j-2}, \theta})\right)
&=\mathbb{E}_{\mu_p}\left(\prod_{i=1}^2\eta_0(\mathcal{X}_{t+\theta, i}^{z_i, 0})\prod_{j=1}^2\eta_0(\mathcal{X}_{t+\theta, j}^{v_j, \theta})\right)\\
&=\mathbb{E}_{\mu_p}\left(\prod_{i=1}^2\eta_0(\mathcal{X}_{t+\theta, i}^{z_i, 0})\right)\mathbb{E}_{\mu_p}\left(\prod_{j=1}^2\eta_0(\mathcal{X}_{t+\theta, j}^{v_j, \theta})\right)
\end{align*}
since $\left\{\mathcal{X}_{t+\theta, 1}^{z_1, 0}, \mathcal{X}_{t+\theta, 2}^{z_2, 0}\right\}\bigcap \left\{\mathcal{X}_{t+\theta, 1}^{v_1, \theta}, \mathcal{X}_{t+\theta, 2}^{v_2, \theta}\right\}=\emptyset$ and $\mu_p$ is a product measure.

Consequently,
\[
\left|{\rm \uppercase\expandafter{\romannumeral1}}-{\rm \uppercase\expandafter{\romannumeral2}}\right|\leq \tilde{\mathbb{P}}(\hat{\beta}\leq t+\theta)\leq
\tilde{\mathbb{P}}(\hat{\beta}<+\infty).
\]
According to the definition of $\hat{\beta}$, $\hat{\beta}=\inf\{\beta_{11}, \beta_{12}, \beta_{21}, \beta_{22}\}$, where
\[
\beta_{ij}=\inf\{t:~t\geq \theta, \mathcal{X}_{t, i}^{z_i, 0}=\mathcal{X}_{t, j}^{v_j, \theta}\}
\]
for $i, j\in \{1, 2\}$. According to the total probability formula,
\[
\tilde{\mathbb{P}}(\beta_{ij}<+\infty)=\sum_{w\in \mathbb{Z}^4}p_\theta(z_i, w)\Phi(w-v_j).
\]
According to the strong Markov property of the simple random walk,
\[
\Phi(w-v_j)=K_1\int_0^{+\infty}p_r(w, v_j)dr
\]
and hence
\[
\sum_{w\in \mathbb{Z}^4}p_\theta(z_i, w)\Phi(w-v_j)=K_1\int_0^{+\infty}p_{r+\theta}(z_i, v_j)dr.
\]
As a result,
\begin{align*}
\left|{\rm \uppercase\expandafter{\romannumeral1}}-{\rm \uppercase\expandafter{\romannumeral2}}\right|\leq \tilde{\mathbb{P}}(\hat{\beta}\leq t+\theta)&\leq
\tilde{\mathbb{P}}(\hat{\beta}<+\infty)\\
&\leq \sum_{i=1}^2\sum_{j=1}^2\tilde{\mathbb{P}}(\beta_{ij}<+\infty)\leq K_1\sum_{i=1}^2\sum_{j=1}^2\int_0^{+\infty}p_{r+\theta}(z_i, v_j)dr
\end{align*}
and \eqref{equ 3.5} holds.
\qed

\subsection{Proof of Lemma \ref{Lemma 3.3 correlation decay d=4}}\label{Appendix A.3}
\proof[Proof of Lemma \ref{Lemma 3.3 correlation decay d=4}]
We only deal with the $\mathbb{E}_{\nu_p}\left((\eta_0(x)-\eta_0(y))^2\eta_0(z)\eta_0(w)\right)$ part since the other part follows from a similar argument. According to the definition of $\nu_p$,
\[
\mathbb{E}_{\nu_p}\left((\eta_0(x)-\eta_0(y))^2\eta_0(z)\eta_0(w)\right)
=\lim_{t\rightarrow+\infty}\mathbb{E}_{\mu_p}\left((\eta_t(x)-\eta_t(y))^2\eta_t(z)\eta_t(w)\right)
\]
By Proposition \ref{Proposition A.1 dual},
\[
\mathbb{E}_{\mu_p}\left((\eta_t(x)-\eta_t(y))^2\eta_t(z)\eta_t(w)\right)
=\tilde{\mathbb{E}}\mathbb{E}_{\mu_p}\left((\eta_0(\mathcal{X}_{t, 1}^{x, 0})-\eta_0(\mathcal{X}_{t, 2}^{y, 0}))^2\eta_0(\mathcal{X}_{t, 3}^{z, 0})\eta_0(\mathcal{X}_{t, 4}^{w, 0})\right).
\]
As in the proof of Equation \eqref{equ 3.5}, we can assume that $\left\{\left(\mathcal{X}_{t, 1}^{x, 0}, \mathcal{X}_{t, 2}^{y, 0}, \mathcal{X}_{t, 3}^{z, 0}, \mathcal{X}_{t, 4}^{w, 0}\right)\right\}_{t\geq 0}$ and $\{\mathcal{X}_{t, 1}^{z, 0}\}_{t\geq 0}$, $\{\mathcal{X}_{t, 1}^{w, 0}\}_{t\geq 0}$ are defined in the same probability space such that
\[
\left\{\left(\mathcal{X}_{t, 1}^{x, 0}, \mathcal{X}_{t, 2}^{y, 0}\right)\right\}_{t\geq 0}, \{\mathcal{X}_{t, 1}^{z, 0}\}_{t\geq 0}
\text{~and~}\{\mathcal{X}_{t, 1}^{w, 0}\}_{t\geq 0}
\]
are independent and
\[
\left(\mathcal{X}_{t, 3}^{z, 0}, \mathcal{X}_{t, 4}^{w, 0}\right)=\left(\mathcal{X}_{t, 1}^{z, 0}, \mathcal{X}_{t, 1}^{w, 0}\right)
\]
for $t\leq \hat{\vartheta}$, where
\[
\hat{\vartheta}=\inf\{\vartheta_1, \vartheta_2, \vartheta_3\}
\]
and
\[
\vartheta_1=\inf\left\{t:~\left\{\mathcal{X}_{t, 1}^{x, 0}, \mathcal{X}_{t, 2}^{y, 0}\right\}\bigcap\{\mathcal{X}_{t, 1}^{z, 0}\}\neq \emptyset\right\},
\]
\[
\vartheta_2=\inf\left\{t:~\left\{\mathcal{X}_{t, 1}^{x, 0}, \mathcal{X}_{t, 2}^{y, 0}\right\}\bigcap\{\mathcal{X}_{t, 1}^{w, 0}\}\neq \emptyset\right\}
\]
and
\[
\vartheta_3=\inf\left\{t:~\mathcal{X}_{t, 1}^{w, 0}=\mathcal{X}_{t, 1}^{z, 0}\right\}.
\]
Then, according to an argument similar to that in the verification of \eqref{equ 3.5},
\begin{align*}
&\left|\tilde{\mathbb{E}}\mathbb{E}_{\mu_p}\left((\eta_0(\mathcal{X}_{t, 1}^{x, 0})-\eta_0(\mathcal{X}_{t, 2}^{y, 0}))^2\eta_0(\mathcal{X}_{t, 3}^{z, 0})\eta_0(\mathcal{X}_{t, 4}^{w, 0})\right)
-{\rm \uppercase\expandafter{\romannumeral3}}\right|\\
& \leq \tilde{\mathbb{P}}\left(\hat{\vartheta}<+\infty\right)\leq \Phi(x-z)+\Phi(x-w)+\Phi(y-z)+\Phi(y-w)+\Phi(z-w),
\end{align*}
where
\[
{\rm \uppercase\expandafter{\romannumeral3}}=
\tilde{\mathbb{E}}\left(\mathbb{E}_{\mu_p}\left((\eta_0(\mathcal{X}_{t, 1}^{x, 0})-\eta_0(\mathcal{X}_{t, 2}^{y, 0}))^2\right)\right)
\tilde{\mathbb{E}}\left(\mathbb{E}_{\mu_p}\eta_0(\mathcal{X}_{t, 1}^{z, 0})\right)\tilde{\mathbb{E}}\left(\mathbb{E}_{\mu_p}\eta_0(\mathcal{X}_{t, 1}^{w, 0})\right).
\]
For $x\sim y$, by Proposition \ref{Proposition A.1 dual},
\begin{align*}
{\rm \uppercase\expandafter{\romannumeral3}}&
=\mathbb{E}_{\mu_p}\left(\left(\eta_t(x)-\eta_t(y)\right)^2\right)\mathbb{E}_{\mu_p}(\eta_t(z))\mathbb{E}_{\mu_p}(\eta_t(w))
\end{align*}
and hence
\[
\lim_{t\rightarrow+\infty}{\rm \uppercase\expandafter{\romannumeral3}}=\mathbb{E}_{\nu_p}\left(\left(\eta_0(x)-\eta_0(y)\right)^2\right)\mathbb{E}_{\nu_p}(\eta_0(z))\mathbb{E}_{\nu_p}(\eta_0(w))
=p^2\chi(p).
\]
As a result, for $x\sim y$,
\begin{align*}
&\left|\mathbb{E}_{\nu_p}\left((\eta_0(x)-\eta_0(y))^2\eta_0(z)\eta_0(w)\right)-\chi(p)p^2\right|\\
&\leq \Phi(x-z)+\Phi(x-w)+\Phi(y-z)+\Phi(y-w)+\Phi(z-w).
\end{align*}
Since the simple random walk on $\mathbb{Z}^4$ is transient, for any $\epsilon>0$, there exists $M_1=M_1(\epsilon)<+\infty$ such that
$\Phi(z)\leq \frac{\epsilon}{10}$ for any $z$ such that $\|z\|_1\geq M_1$. Then, Lemma \ref{Lemma 3.3 correlation decay d=4} holds with $M=M_1+1$.
\qed
\subsection{Proof of \eqref{equ 3.13}}\label{Appendix A.4}
\proof[Proof of \eqref{equ 3.13}]
According to the definition of $q_t^{\textbf{0}, x}$,
\begin{align*}
&\sum_{z_1\in \mathbb{Z}^4}\sum_{v_2:0<\|v_2-z_1\|_1\leq M}\int_0^N\int_0^Nq_{r_1}^{\textbf{0}, x}(z_1, z_1)q_{r_2}^{\textbf{0}, x}(z_1, v_2)dr_1dr_2\\
&=\int_0^N\int_0^N\tilde{\mathbb{P}}\left(\mathcal{X}_{r_1, 1}^{\textbf{0}, 0}=\mathcal{X}_{r_1, 2}^{x, 0}=\tilde{\mathcal{X}}_{r_2, 1}^{\textbf{0}, 0}, 0<\|\tilde{\mathcal{X}}_{r_2, 2}^{x, 0}-\tilde{\mathcal{X}}_{r_2, 1}^{\textbf{0}, 0}\|_1\leq M\right)dr_1dr_2,
\end{align*}
where all the notations are defined as in the proof of Lemma \ref{lemma 3.4}. We assume that $\{Y_t\}_{t\geq 0}$ is a copy of $\{X_t\}_{t\geq 0}$ starting from $x$ and $\{Y_t\}_{t\geq 0}$ is independent of $\left\{\left(\mathcal{X}_{t, 1}^{\textbf{0}, 0}, \mathcal{X}_{t, 2}^{x, 0}, \tilde{\mathcal{X}}_{t, 1}^{\textbf{0}, 0}\right)\right\}_{t\geq 0}$. Then, according to the definition of the coalescing random walk, we can further assume that $Y_t=\tilde{\mathcal{X}}_{t, 2}^{x, 0}$ for $0\leq t\leq \tau_3$,
where
\[
\tau_3=\inf\left\{t: \tilde{\mathcal{X}}_{t, 1}^{\textbf{0}, 0}=Y_t\right\}.
\]
Hence,
\begin{align*}
&\tilde{\mathbb{P}}\left(\mathcal{X}_{r_1, 1}^{\textbf{0}, 0}=\mathcal{X}_{r_1, 2}^{x, 0}=\tilde{\mathcal{X}}_{r_2, 1}^{\textbf{0}, 0}, 0<\|\tilde{\mathcal{X}}_{r_2, 2}^{x, 0}-\tilde{\mathcal{X}}_{r_2, 1}^{\textbf{0}, 0}\|_1\leq M\right)\\
&=\tilde{\mathbb{P}}\left(\mathcal{X}_{r_1, 1}^{\textbf{0}, 0}=\mathcal{X}_{r_1, 2}^{x, 0}=\tilde{\mathcal{X}}_{r_2, 1}^{\textbf{0}, 0}, \tau_3>r_2, 0<\|Y_{r_2}-\tilde{\mathcal{X}}_{r_2, 1}^{\textbf{0}, 0}\|_1\leq M\right)\\
&\leq \tilde{\mathbb{P}}\left(\mathcal{X}_{r_1, 1}^{\textbf{0}, 0}=\mathcal{X}_{r_1, 2}^{x, 0}=\tilde{\mathcal{X}}_{r_2, 1}^{\textbf{0}, 0}, \|Y_{r_2}-\tilde{\mathcal{X}}_{r_2, 1}^{\textbf{0}, 0}\|_1\leq M\right).
\end{align*}
For each $z\in \mathbb{Z}^4$, there are $(2M+1)^4$ vertices $v\in \mathbb{Z}^4$ satisfies that $\|v-z\|_1\leq M$. Then, according to the fact that
$p_t(w_1, w_2)\leq p_t(\textbf{0}, \textbf{0})$ for any $w_1, w_2\in \mathbb{Z}^4$ and the total probability formula,
\begin{align*}
&\tilde{\mathbb{P}}\left(\mathcal{X}_{r_1, 1}^{\textbf{0}, 0}=\mathcal{X}_{r_1, 2}^{x, 0}=\tilde{\mathcal{X}}_{r_2, 1}^{\textbf{0}, 0}, \|Y_{r_2}-\tilde{\mathcal{X}}_{r_2, 1}^{\textbf{0}, 0}\|_1\leq M\right)\\
&\leq (2M+1)^4\tilde{\mathbb{P}}\left(\mathcal{X}_{r_1, 1}^{\textbf{0}, 0}=\mathcal{X}_{r_1, 2}^{x, 0}=\tilde{\mathcal{X}}_{r_2, 1}^{\textbf{0}, 0}\right)p_{r_2}(\textbf{0}, \textbf{0})
\end{align*}
and hence
\begin{align*}
&\int_0^N\int_0^N\tilde{\mathbb{P}}\left(\mathcal{X}_{r_1, 1}^{\textbf{0}, 0}=\mathcal{X}_{r_1, 2}^{x, 0}=\tilde{\mathcal{X}}_{r_2, 1}^{\textbf{0}, 0}, 0<\|\tilde{\mathcal{X}}_{r_2, 2}^{x, 0}-\tilde{\mathcal{X}}_{r_2, 1}^{\textbf{0}, 0}\|_1\leq M\right)dr_1dr_2\\
&\leq (2M+1)^4\int_0^N\int_0^N\tilde{\mathbb{P}}\left(\mathcal{X}_{r_1, 1}^{\textbf{0}, 0}=\mathcal{X}_{r_1, 2}^{x, 0}=\tilde{\mathcal{X}}_{r_2, 1}^{\textbf{0}, 0}\right)p_{r_2}(\textbf{0}, \textbf{0})dr_1dr_2\\
&\leq (2M+1)^4\int_0^N\int_0^N\tilde{\mathbb{P}}\left(\mathcal{X}_{r_1, 1}^{\textbf{0}, 0}=\tilde{\mathcal{X}}_{r_2, 1}^{\textbf{0}, 0}\right)p_{r_2}(\textbf{0}, \textbf{0})dr_1dr_2\\
&=(2M+1)^4\int_0^N\int_0^Np_{r_1+r_2}(\textbf{0}, \textbf{0})p_{r_2}(\textbf{0}, \textbf{0})dr_1dr_2\leq (2M+1)^4\left(\int_0^{+\infty}p_s(\textbf{0}, \textbf{0})ds\right)^2.
\end{align*}
Note that $\int_0^{+\infty}p_s(\textbf{0}, \textbf{0})ds<+\infty$ according to the transience of the simple random walk on $\mathbb{Z}^4$. As a result,
\[
\int_0^N\int_0^N\tilde{\mathbb{P}}\left(\mathcal{X}_{r_1, 1}^{\textbf{0}, 0}=\mathcal{X}_{r_1, 2}^{x, 0}=\tilde{\mathcal{X}}_{r_2, 1}^{\textbf{0}, 0}, 0<\|\tilde{\mathcal{X}}_{r_2, 2}^{x, 0}-\tilde{\mathcal{X}}_{r_2, 1}^{\textbf{0}, 0}\|_1\leq M\right)dr_1dr_2=O(1)
\]
and hence \eqref{equ 3.13} holds.
\qed

\subsection{Proof of \eqref{equ 3.16}}\label{Appendix A.5}
\proof[Proof of \eqref{equ 3.16}]
According to the definition of $q_t^{\textbf{0}, x}$,
\begin{align*}
&\sum_{z_1\in \mathbb{Z}^4}
\sum_{(v_2, w_2)\in \mathcal{R}(M, z_1)}\int_0^N\int_0^Nq_{r_1}^{\textbf{0}, x}(z_1, v_2)q_{r_2}^{\textbf{0}, x}(z_1, w_2)dr_1dr_2\\
&=\int_0^N\int_0^N\tilde{\mathbb{P}}\left(\mathcal{X}_{r_1, 1}^{\textbf{0}, 0}=\tilde{\mathcal{X}}_{r_2, 1}^{\textbf{0}, 0},
\left(\mathcal{X}_{r_1, 2}^{x, 0}, \tilde{\mathcal{X}}_{r_2, 2}^{x, 0}\right)\in\mathcal{R}(M, \mathcal{X}_{r_1, 1}^{\textbf{0}, 0})\right)dr_1dr_2,
\end{align*}
where all the notations are defined as in the proof of Lemma \ref{lemma 3.4}. Let $Y_t, \tau_3$ be defined as in the proof of \eqref{equ 3.16}. We further define $\{\tilde{Y}_t\}_{t\geq 0}$ as another copy of $\{X_t\}_{t\geq 0}$ starting from $x$ such that $\{\tilde{Y}_t\}_{t\geq 0}$ is independent of
$\left\{\left(\mathcal{X}_{t, 1}^{\textbf{0}, 0}, \tilde{\mathcal{X}}_{t, 1}^{\textbf{0}, 0}, Y_t\right)\right\}_{t\geq 0}$ and
$\mathcal{X}_{t, 2}^{x, 0}=\tilde{Y}_t$ for $t\leq \tau_4$, where
\[
\tau_4=\inf\left\{t:~\tilde{Y}_t=\mathcal{X}_{t, 1}^{\textbf{0}, 0}\right\}.
\]
Then,
\begin{align*}
&\tilde{\mathbb{P}}\left(\mathcal{X}_{r_1, 1}^{\textbf{0}, 0}=\tilde{\mathcal{X}}_{r_2, 1}^{\textbf{0}, 0},
\left(\mathcal{X}_{r_1, 2}^{x, 0}, \tilde{\mathcal{X}}_{r_2, 2}^{x, 0}\right)\in\mathcal{R}(M, \mathcal{X}_{r_1, 1}^{\textbf{0}, 0})\right)\\
&=\tilde{\mathbb{P}}\left(\mathcal{X}_{r_1, 1}^{\textbf{0}, 0}=\tilde{\mathcal{X}}_{r_2, 1}^{\textbf{0}, 0}, \tau_3>r_2, \tau_4>r_1,
\left(\tilde{Y}_{r_1}, Y_{r_2}\right)\in\mathcal{R}(M, \mathcal{X}_{r_1, 1}^{\textbf{0}, 0})\right)\\
&\leq \tilde{\mathbb{P}}\left(\mathcal{X}_{r_1, 1}^{\textbf{0}, 0}=\tilde{\mathcal{X}}_{r_2, 1}^{\textbf{0}, 0}, \|\tilde{Y}_{r_1}-\mathcal{X}_{r_1, 1}^{\textbf{0}, 0}\|_1\leq M\right)+\tilde{\mathbb{P}}\left(\mathcal{X}_{r_1, 1}^{\textbf{0}, 0}=\tilde{\mathcal{X}}_{r_2, 1}^{\textbf{0}, 0}, \|Y_{r_2}-\mathcal{X}_{r_1, 1}^{\textbf{0}, 0}\|_1\leq M\right)\\
&\text{\quad}+\tilde{\mathbb{P}}\left(\mathcal{X}_{r_1, 1}^{\textbf{0}, 0}=\tilde{\mathcal{X}}_{r_2, 1}^{\textbf{0}, 0}, \|Y_{r_2}-\tilde{Y}_{r_1}\|_1\leq M\right).
\end{align*}
According to the total probability formula,
\begin{align*}
&\tilde{\mathbb{P}}\left(\mathcal{X}_{r_1, 1}^{\textbf{0}, 0}=\tilde{\mathcal{X}}_{r_2, 1}^{\textbf{0}, 0}, \|\tilde{Y}_{r_1}-\mathcal{X}_{r_1, 1}^{\textbf{0}, 0}\|_1\leq M\right)\leq (2M+1)^4p_{r_1}(\textbf{0}, \textbf{0})\tilde{\mathbb{P}}\left(\mathcal{X}_{r_1, 1}^{\textbf{0}, 0}=\tilde{\mathcal{X}}_{r_2, 1}^{\textbf{0}, 0}\right)\\
&=(2M+1)^4p_{r_1}(\textbf{0}, \textbf{0})p_{r_1+r_2}(\textbf{0}, \textbf{0}).
\end{align*}
Similarly,
\[
\tilde{\mathbb{P}}\left(\mathcal{X}_{r_1, 1}^{\textbf{0}, 0}=\tilde{\mathcal{X}}_{r_2, 1}^{\textbf{0}, 0}, \|Y_{r_2}-\mathcal{X}_{r_1, 1}^{\textbf{0}, 0}\|_1\leq M\right)\leq (2M+1)^4p_{r_2}(\textbf{0}, \textbf{0})p_{r_1+r_2}(\textbf{0}, \textbf{0})
\]
and
\[
\tilde{\mathbb{P}}\left(\mathcal{X}_{r_1, 1}^{\textbf{0}, 0}=\tilde{\mathcal{X}}_{r_2, 1}^{\textbf{0}, 0}, \|Y_{r_2}-\tilde{Y}_{r_1}\|_1\leq M\right)
\leq (2M+1)^4p_{r_2}(\textbf{0}, \textbf{0})p_{r_1+r_2}(\textbf{0}, \textbf{0}).
\]
Therefore,
\begin{align*}
&\sum_{z_1\in \mathbb{Z}^4}
\sum_{(v_2, w_2)\in \mathcal{R}(M, z_1)}\int_0^N\int_0^Nq_{r_1}^{\textbf{0}, x}(z_1, v_2)q_{r_2}^{\textbf{0}, x}(z_1, w_2)dr_1dr_2\\
&\leq 3(2M+1)^4\int_0^{+\infty}\int_0^{+\infty}p_{r_2}(\textbf{0}, \textbf{0})p_{r_1+r_2}(\textbf{0}, \textbf{0})dr_1dr_2\\
&\leq 3(2M+1)^4\left(\int_0^{+\infty}p_s(\textbf{0}, \textbf{0})ds\right)^2<+\infty
\end{align*}
and hence \eqref{equ 3.16} holds.
\qed

\subsection{Proof of Equation \eqref{equ 4.8}}\label{Appendix A.6}
In this subsection, we prove \eqref{equ 4.8} through several steps. According to the invariance of $\nu_p$ the definition of $\hat{H}_\epsilon$,
\begin{align}\label{equ A.6 1}
&{\rm Var}_{\nu_p}\left(\int_0^{t}\mathcal{Y}_r^N(\hat{H}_\epsilon)dr
-\frac{1}{N^{\frac{3}{4}}}\xi_{tN}^{\textbf{0}, \sqrt{N}u_N}\right)\notag\\
&=\frac{2}{(2\epsilon)^{12}N^{\frac{3}{2}}}\int_0^{tN}d\theta_1\int_0^{\theta_2}\frac{1}{N^6}\sum_{(z_1, z_2, z_3, z_4)\in D^N(\epsilon)}\notag\\
&
{\rm Cov}_{\nu_p}\left(\eta_{\theta_2}(z_1)\eta_{\theta_2}(z_2)-\eta_{\theta_2}(\textbf{0})\eta_{\theta_2}(\sqrt{N}u_N), \eta_{0}(z_3)\eta_{0}(z_4)-\eta_{0}(\textbf{0})\eta_{0}(\sqrt{N}u_N)\right)d\theta_2,
\end{align}
where $D^N(\epsilon)\subseteq (\mathbb{Z}^3)^4$ is defined as
\begin{align*}
D^N(\epsilon)=\Bigg\{(z_1, z_2, z_3, z_4)\in (\mathbb{Z}^3)^4:&
~\|z_1\|_1\leq \sqrt{N}\epsilon, \|z_2-\sqrt{N}u\|_1\leq \sqrt{N}\epsilon, \\
&\|z_3\|_1\leq \sqrt{N}\epsilon, \|z_4-\sqrt{N}u\|_1\leq \sqrt{N}\epsilon\Bigg\}.
\end{align*}
For any $(z_1, z_2, z_3, z_4)\in D^N(\epsilon)$, according to the definition of $\nu_p$,
\begin{align*}
{\rm Cov}_{\nu_p}\left(\eta_{\theta_2}(z_1)\eta_{\theta_2}(z_2), \eta_0(z_3)\eta_0(z_4)\right)=\lim_{r\rightarrow+\infty}{\rm Cov}_{\mu_p}\left(\eta_{\theta_2+r}(z_1)\eta_{\theta_2+r}(z_2), \eta_r(z_3)\eta_r(z_4)\right).
\end{align*}
By Proposition \ref{Proposition A.1 dual},
\begin{align}\label{equ A.6 1 two}
&{\rm Cov}_{\mu_p}\left(\eta_{\theta_2+r}(z_1)\eta_{\theta_2+r}(z_2), \eta_r(z_3)\eta_r(z_4)\right)=\notag\\
&\tilde{\mathbb{E}}\left(\mathbb{E}_{\mu_p}\left(\prod_{j=1}^2\eta_0(\mathcal{X}_{r+\theta_2, j}^{z_j, 0})\prod_{j=3}^4\eta_0(\mathcal{X}_{r+\theta_2, j}^{z_j, \theta_2})\right)\right)\notag\\
&-\tilde{\mathbb{E}}\left(\mathbb{E}_{\mu_p}\left(\prod_{j=1}^2\eta_0(\mathcal{X}_{r+\theta_2, j}^{z_j, 0})\right)\right)
\tilde{\mathbb{E}}\left(\mathbb{E}_{\mu_p}\left(\prod_{j=3}^4\eta_0(\tilde{\mathcal{X}}_{r+\theta_2, j-2}^{z_j, \theta_2})\right)\right),
\end{align}
where $\left\{\left(\tilde{\mathcal{X}}_{t, 1}^{z_3, \theta_2}, \tilde{\mathcal{X}}_{t, 2}^{z_4, \theta_2}\right)\right\}_{t\geq 0}$ is a copy of
$\left\{\left(\mathcal{X}_{t, 1}^{z_3, \theta_2}, \mathcal{X}_{t, 2}^{z_4, \theta_2}\right)\right\}_{t\geq 0}$.

We introduce four independent random walks to make $\left\{\left(\tilde{\mathcal{X}}_{t, 1}^{z_3, \theta_2}, \tilde{\mathcal{X}}_{t, 2}^{z_4, \theta_2}\right)\right\}_{t\geq 0}$ and
\[
\left\{\left(\mathcal{X}_{t, 1}^{z_1, 0}, \mathcal{X}_{t, 2}^{z_2, 0}, \mathcal{X}_{t, 3}^{z_3, \theta_2}, \mathcal{X}_{t, 4}^{z_4, \theta_2}\right)\right\}_{t\geq 0}
\]
defined in the same probability space. For $j=1, 2$, we denote by $\{\mathcal{Z}_t^j\}_{t\geq 0}$ a copy of $\{X_t\}_{t\geq 0}$ starting from $z_j$. For $j=3, 4$, we denote by $\{\mathcal{Z}_t^j\}_{t\geq 0}$ the random process such that $\mathcal{Z}_t^j=z_j$ for $0\leq t\leq \theta_2$ and $\{\mathcal{Z}_{t+\theta_2}^j\}_{t\geq 0}$ is a copy of $\{X_t\}_{t\geq 0}$ starting from $z_j$. We assume that
\[
\{\mathcal{Z}_t^j:~t\geq 0\}_{j=1}^4
\]
are independent. Then, for any $t\geq 0$, we define $\mathcal{X}_{t, 1}^{z_1, 0}=\mathcal{Z}_t^1$. We further define
\[
\mathcal{X}_{t, 2}^{z_2, 0}=
\begin{cases}
\mathcal{Z}_t^2 &\text{~if~}t\leq \Gamma_{12},\\
\mathcal{X}_{t, 1}^{z_1, 0} &\text{~if~} t>\Gamma_{12},
\end{cases}
\]
where
\[
\Gamma_{ij}=\inf\left\{t:~t\geq c_i, t\geq c_j, \mathcal{Z}_t^i=\mathcal{Z}_t^j\right\}
\]
for any different $i, j\in \{1, 2, 3, 4\}$, where $(c_1, c_2, c_3, c_4)=(0, 0, \theta_2, \theta_2)$. For any $t\geq 0$, we define
$\tilde{\mathcal{X}}_{t, 1}^{z_3, \theta_2}=\mathcal{Z}_t^3$ and
\[
\tilde{\mathcal{X}}_{t, 2}^{z_4, \theta_2}=
\begin{cases}
\mathcal{Z}_t^4 &\text{~if~}t\leq \Gamma_{34},\\
\tilde{\mathcal{X}}_{t, 1}^{z_3, \theta_2} &\text{~if~} t>\Gamma_{34}.
\end{cases}
\]
For $i\in \{1, 2\}$, we define
\[
\hat{\Gamma}_{i3}=\inf\{t:t\geq \theta_2, \mathcal{X}_{t, i}^{z_i, 0}=\mathcal{Z}_t^3\}.
\]
For any $t\geq 0$, we define
\[
\mathcal{X}_{t, 3}^{z_3, \theta_2}=
\begin{cases}
\mathcal{Z}_t^3 &\text{~if~}t\leq \min\{\hat{\Gamma}_{13}, \hat{\Gamma}_{23}\},\\
\mathcal{X}_{t, 2}^{z_2, 0} &\text{~if~}t>\min\{\hat{\Gamma}_{13}, \hat{\Gamma}_{23}\}
\text{~and~}\hat{\Gamma}_{13}>\hat{\Gamma}_{23},\\
\mathcal{X}_{t, 1}^{z_1, 0} &\text{~if~}t>\min\{\hat{\Gamma}_{13}, \hat{\Gamma}_{23}\}
\text{~and~}\hat{\Gamma}_{23}\geq \hat{\Gamma}_{13}.
\end{cases}
\]
For $i\in \{1, 2, 3\}$, we define
\[
\hat{\Gamma}_{i4}=\inf\{t:t\geq \theta_2, \mathcal{X}_{t, i}^{z_i, c_i}=\mathcal{Z}_t^4\}.
\]
For any $t\geq 0$, we define
\[
\mathcal{X}_{t, 4}^{z_4, \theta_2}=
\begin{cases}
\mathcal{Z}_t^4 &\text{~if~}t\leq \min\{\hat{\Gamma}_{14}, \hat{\Gamma}_{24}, \hat{\Gamma}_{34}\},\\
\mathcal{X}_{t, 3}^{z_3, \theta_2} &\text{~if~}\min\{\hat{\Gamma}_{14}, \hat{\Gamma}_{24}, \hat{\Gamma}_{34}\}=\hat{\Gamma}_{34}\text{~and~}t\geq \hat{\Gamma}_{34},\\
\mathcal{X}_{t, 2}^{z_2, 0} &\text{~if~}\min\{\hat{\Gamma}_{14}, \hat{\Gamma}_{24}, \hat{\Gamma}_{34}\}=\hat{\Gamma}_{24}\text{~and~}t\geq \hat{\Gamma}_{24},\\
\mathcal{X}_{t, 1}^{z_1, 0} &\text{~if~}\min\{\hat{\Gamma}_{14}, \hat{\Gamma}_{24}, \hat{\Gamma}_{34}\}=\hat{\Gamma}_{14}
\text{~and~}t\geq \hat{\Gamma}_{14}.
\end{cases}
\]
According to the above definitions, $\left\{\left(\tilde{\mathcal{X}}_{t, 1}^{z_3, \theta_2}, \tilde{\mathcal{X}}_{t, 2}^{z_4, \theta_2}\right)\right\}_{t\geq 0}$ is independent of
\[
\left\{\left(\mathcal{X}_{t, 1}^{z_1, 0}, \mathcal{X}_{t, 2}^{z_2, 0}\right)\right\}_{t\geq 0}
\]
and we have
\[
\left(\tilde{\mathcal{X}}_{t, 1}^{z_3, \theta_2}, \tilde{\mathcal{X}}_{t, 2}^{z_4, \theta_2}\right)
=\left(\mathcal{X}_{t, 3}^{z_3, \theta_2}, \mathcal{X}_{t, 4}^{z_4, \theta_2}\right)
\]
for any $t\leq \tilde{\Gamma}$, where
\[
\tilde{\Gamma}=\inf\left\{t:~t\geq \theta_2, \left\{\mathcal{X}_{t, 1}^{z_1, 0}, \mathcal{X}_{t, 2}^{z_2, 0}\right\}\bigcap \left\{\tilde{\mathcal{X}}_{t, 1}^{z_3, \theta_2}, \tilde{\mathcal{X}}_{t, 2}^{z_4, \theta_2}\right\}\neq \emptyset\right\}.
\]
For any $i\in \{1, 2\}$, we denote by $\mathcal{I}_i$ the event that $\Gamma_{i3}<+\infty$ and $\Gamma_{i4}<+\infty$. For any $j\in \{3, 4\}$, we denote by $\mathcal{K}_j$ the event that $\Gamma_{1j}<+\infty$ and $\Gamma_{2j}<+\infty$. For any $i\in \{1, 2\}$ and $j\in \{3, 4\}$, we denote by $\mathcal{H}_{ij}$ the event that $\Gamma_{ij}<+\infty$ and $\Gamma_{\hat{i}\hat{j}}<+\infty$, where $\hat{i}=3-i, \hat{j}=7-j$. For any $i\in \{1, 2\}$ and $j\in \{3, 4\}$, we denote by $\mathcal{N}_{ij}$ the event that $\Gamma_{i\hat{i}}<+\infty$ and $\Gamma_{ij}<\infty$. We further denote by $\mathcal{T}_{ji}$ the event that
$\Gamma_{j\hat{j}}<+\infty$ and $\Gamma_{ij}<+\infty$. We write $\mathcal{I}_1$ as $\mathcal{I}_1(z_1, z_2, z_3, z_4, \theta_2)$ when we need to emphasize the dependence of $\mathcal{I}_1$ on $z_1, z_2, z_3, z_4, \theta_2$. So do other notations.

The following lemma is crucial for us to prove \eqref{equ 4.8}.
\begin{lemma}\label{lemma appendix A.6 1}
Let $\epsilon$ be sufficiently small such that $\|u\|_1\geq 3\epsilon$. For any given $i\in \{1, 2\}$ and $j\in \{3, 4\}$, we have
\[
\lim_{N\rightarrow+\infty}\frac{1}{N^{\frac{3}{2}}}
\sup_{(z_1, z_2, z_3, z_4)\in D^N(\epsilon)}\int_0^{tN}d\theta_1\int_0^{\theta_1}\tilde{\mathbb{P}}\left(\mathcal{I}_i(z_1, z_2, z_3, z_4, \theta_2)\right)d\theta_2=0,
\]
\[
\lim_{N\rightarrow+\infty}\frac{1}{N^{\frac{3}{2}}}
\sup_{(z_1, z_2, z_3, z_4)\in D^N(\epsilon)}\int_0^{tN}d\theta_1\int_0^{\theta_1}\tilde{\mathbb{P}}\left(\mathcal{K}_j(z_1, z_2, z_3, z_4, \theta_2)\right)d\theta_2=0,
\]
\[
\lim_{N\rightarrow+\infty}\frac{1}{N^{\frac{3}{2}}}
\sup_{(z_1, z_2, z_3, z_4)\in D^N(\epsilon)}\int_0^{tN}d\theta_1\int_0^{\theta_1}\tilde{\mathbb{P}}\left(\mathcal{H}_{ij}(z_1, z_2, z_3, z_4, \theta_2)\right)d\theta_2=0,
\]
\[
\lim_{N\rightarrow+\infty}\frac{1}{N^{\frac{3}{2}}}
\sup_{(z_1, z_2, z_3, z_4)\in D^N(\epsilon)}\int_0^{tN}d\theta_1\int_0^{\theta_1}\tilde{\mathbb{P}}\left(\mathcal{N}_{ij}(z_1, z_2, z_3, z_4, \theta_2)\right)d\theta_2=0
\]
and
\[
\lim_{N\rightarrow+\infty}\frac{1}{N^{\frac{3}{2}}}
\sup_{(z_1, z_2, z_3, z_4)\in D^N(\epsilon)}\int_0^{tN}d\theta_1\int_0^{\theta_1}\tilde{\mathbb{P}}\left(\mathcal{T}_{ji}(z_1, z_2, z_3, z_4, \theta_2)\right)d\theta_2=0.
\]
\end{lemma}
\proof[Proof of Lemma \ref{lemma appendix A.6 1}]
We mainly deal with the parts of $\mathcal{I}_1, \mathcal{H}_{13}$ and $\mathcal{N}_{13}$, since other parts follow from similar arguments. For the $\mathcal{I}_1$ part, let $\hat{m}(dr, dz)$ be the joint distribution of $\left(\Gamma_{13}, \mathcal{Z}^1_{\Gamma_{13}}\right)$, then according to the Markov property,
\begin{align*}
&\tilde{\mathbb{P}}\left(\mathcal{I}_1, \Gamma_{13}\leq \Gamma_{14}<+\infty\right)\\
&=\sum_{z\in \mathbb{Z}^3}\sum_{w\in \mathbb{Z}^3}\int_{\theta_2}^{+\infty}\left(p_{r-\theta_2}(z_4, w)\Phi(w-z)\right)\hat{m}(dr, dz)\\
&=K_1\sum_{z}\int_{\theta_2}^{+\infty}\left(\int_0^{+\infty}p_{r-\theta_2+s}(z_4,z)ds\right)\hat{m}(dr, dz)\\
&\leq K_1\sum_{z}\int_{\theta_2}^{+\infty}\left(\int_0^{+\infty}p_{s}(z_4,z)ds\right)\hat{m}(dr, dz)\\
&=\sum_{z}\int_{\theta_2}^{+\infty}\Phi(z_4-z)\hat{m}(dr, dz).
\end{align*}
According to the definition of $\hat{m}$, for any $M\geq 1$,
\begin{align*}
&\sum_{z:\|z-z_4\|_1\leq M}\int_{\theta_2}^{+\infty}\Phi(z_4-z)\hat{m}(dr, dz)\\
&\leq \sum_{z:\|z-z_4\|_1\leq M}\int_{\theta_2}^{+\infty}1\hat{m}(dr, dz)\\
&\leq \sum_{z:\|z-z_4\|_1\leq M}\Phi(z_3-z)\tilde{\mathbb{P}}(\mathcal{Z}_t^1=z\text{~for some~}t\geq \theta_2)\\
&=K_1\sum_{z:\|z-z_4\|_1\leq M}\Phi(z_3-z)\int_0^{+\infty}p_{\theta_2+s}(z_1, z)ds\\
&\leq K_1\int_0^{+\infty}p_{\theta_2+s}(\textbf{0}, \textbf{0})ds\sum_{z:\|z-z_4\|_1\leq M}\Phi(z_3-z).
\end{align*}
For $(z_1, z_2, z_3, z_4)\in D^N(\epsilon)$, we have $\|z_3-z\|_1\geq (\|u\|_1-2\epsilon)\sqrt{N}-M$ for $z$ satisfying $\|z-z_4\|_1\leq M$. Hence,
\begin{equation}\label{equ A.6 2}
\lim_{N\rightarrow+\infty}\sup_{(z_1, z_2, z_3, z_4)\in D^N(\epsilon)}\sum_{z:\|z-z_4\|_1\leq M}\Phi(z_3-z)=0
\end{equation}
for sufficiently small $\epsilon$.
According to the local central limit theorem of the simple random walk on $\mathbb{Z}^3$, $p_{t, 3}(\textbf{0}, \textbf{0})=O(1)t^{-\frac{3}{2}}$ and hence
\[
\int_0^{tN}d\theta_1\int_0^{\theta_1}d\theta_2\int_0^{+\infty}p_{\theta_2+s}(\textbf{0}, \textbf{0})ds=O(1)N^{\frac{3}{2}}.
\]
Consequently, by \eqref{equ A.6 2},
\begin{align}\label{equ A.6 4}
&\limsup_{N\rightarrow+\infty}\frac{1}{N^{\frac{3}{2}}}
\sup_{(z_1, z_2, z_3, z_4)\in D^N(\epsilon)}\int_0^{tN}d\theta_1\int_0^{\theta_1}\tilde{\mathbb{P}}\left(\mathcal{I}_1(z_1, z_2, z_3, z_4, \theta_2), \Gamma_{13}\leq \Gamma_{14}<+\infty\right)d\theta_2\notag\\
&\leq \limsup_{N\rightarrow+\infty}\frac{1}{N^{\frac{3}{2}}}\int_0^{tN}d\theta_1\int_0^{\theta_1}\left(\sum_{z:\|z-z_4\|_1\geq M}\int_{\theta_2}^{+\infty}\Phi(z_4-z)\hat{m}(dr, dz)\right)d\theta_2\notag\\
&\leq \limsup_{N\rightarrow+\infty}\frac{1}{N^{\frac{3}{2}}}\int_0^{tN}d\theta_1\int_0^{\theta_1}\left(\sum_{z:\|z-z_4\|_1\geq M}\int_{\theta_2}^{+\infty}1\hat{m}(dr, dz)\right)d\theta_2\left(\sup_{z:\|z\|_1\geq M}\Phi(z)\right)\notag\\
&\leq \limsup_{N\rightarrow+\infty}\frac{1}{N^{\frac{3}{2}}}\int_0^{tN}d\theta_1\int_0^{\theta_1}\left(\sum_{z}\int_{\theta_2}^{+\infty}1\hat{m}(dr, dz)\right)d\theta_2\left(\sup_{z:\|z\|_1\geq M}\Phi(z)\right)\notag\\
&=\limsup_{N\rightarrow+\infty}\frac{1}{N^{\frac{3}{2}}}\int_0^{tN}d\theta_1\int_0^{\theta_1}\tilde{\mathbb{P}}(\Gamma_{13}<+\infty)d\theta_2\left(\sup_{z:\|z\|_1\geq M}\Phi(z)\right)
\end{align}
for any $M\geq 0$. According to the Markov property,
\[
\tilde{\mathbb{P}}(\Gamma_{13}<+\infty)=K_1\int_0^{+\infty}p_{\theta_2+s}(z_1, z_3)ds
\leq K_1\int_0^{+\infty}p_{\theta_2+s}(\textbf{0}, \textbf{0})ds
\]
and hence
\[
\int_0^{tN}d\theta_1\int_0^{\theta_1}\tilde{\mathbb{P}}(\Gamma_{13}<+\infty)d\theta_2=O(1)N^{\frac{3}{2}}
\]
according to the local cental limit theorem of the simple random walk on $\mathbb{Z}^3$. According to the transience of the simple random walk on $\mathbb{Z}^3$, we have $\lim_{M\rightarrow+\infty}\left(\sup_{z:\|z\|_1\geq M}\Phi(z)\right)=0$. As a result,
\[
\lim_{N\rightarrow+\infty}\frac{1}{N^{\frac{3}{2}}}
\sup_{(z_1, z_2, z_3, z_4)\in D^N(\epsilon)}\int_0^{tN}d\theta_1\int_0^{\theta_1}\tilde{\mathbb{P}}\left(\mathcal{I}_1(z_1, z_2, z_3, z_4, \theta_2), \Gamma_{13}\leq \Gamma_{14}<+\infty\right)d\theta_2=0.
\]
According to a similar argument, the above limit still holds when $\Gamma_{13}\leq \Gamma_{14}<+\infty$ is replaced by $\Gamma_{14}\leq \Gamma_{13}<+\infty$. Consequently, the $\mathcal{I}_1$ part of Lemma \ref{lemma appendix A.6 1} holds.

The $\mathcal{I}_2$ and $\mathcal{K}_j$ parts follow from arguments similar to that leading to the $\mathcal{I}_1$ part. Now we deal with the $\mathcal{H}_{13}$ part, which is relatively easy. According to the Markov property,
\begin{align*}
\tilde{\mathbb{P}}\left(\mathcal{H}_{13}\right)
&=K_1^2\int_0^{+\infty}p_{\theta_2+r}(z_1, z_3)dr\int_0^{+\infty}p_{\theta_2+r}(z_2, z_4)dr\\
&\leq K_1^2\left(\int_0^{+\infty}p_{\theta_2+r}(\textbf{0}, \textbf{0})dr\right)^2.
\end{align*}
Therefore, according to the fact that $p_{t, 3}(\textbf{0}, \textbf{0})=O(1)t^{-3/2}$, we have
\[
\sup_{(z_1, z_2, z_3, z_4)\in D^N(\epsilon)}\int_0^{tN}d\theta_1\int_0^{\theta_1}\tilde{\mathbb{P}}\left(\mathcal{H}_{ij}(z_1, z_2, z_3, z_4, \theta_2)\right)d\theta_2=O(1)N\log N
\]
and hence the $\mathcal{H}_{13}$ part of Lemma \ref{lemma appendix A.6 1} holds. Other $\mathcal{H}_{ij}$ parts follow from a similar argument.

Now we deal with the $\mathcal{N}_{13}$ part. This part is further divided into two parts, i.e., $\{\Gamma_{12}<\Gamma_{13}<+\infty\}$ and $\{\Gamma_{13}<\Gamma_{12}<+\infty\}$. The $\{\Gamma_{13}<\Gamma_{12}<+\infty\}$ part follows from an argument similar to that in the $\mathcal{I}_1$ part. So we only deal with the $\{\Gamma_{12}<\Gamma_{13}<+\infty\}$ part. Let $\tilde{m}(dr, dz)$ be the joint distribution of $\Gamma_{12}$ and $\mathcal{Z}^1_{\Gamma_{12}}$, then
\begin{align*}
&\tilde{\mathbb{P}}\left(\Gamma_{12}<\theta_2<\Gamma_{13}<+\infty\right)\\
&=K_1\sum_{z\in \mathbb{Z}^3}\sum_{w\in \mathbb{Z}^3}\int_0^{\theta_2}p_{\theta_2-r}(z, w)\left(\int_0^{+\infty}p_s(w, z_3)ds\right)\tilde{m}(dr, dz)\\
&=K_1\sum_{z}\int_0^{\theta_2}\left(\int_0^{+\infty}p_{\theta_2-r+s}(z_3, z)ds\right)\tilde{m}(dr, dz).
\end{align*}
Therefore, according to the fact that $p_{t, 3}(\textbf{0}, \textbf{0})=O(1)t^{-3/2}$,
\begin{align*}
&\int_0^{\theta_1}\tilde{\mathbb{P}}\left(\Gamma_{12}<\theta_2<\Gamma_{13}<+\infty\right)d\theta_2\\
&\leq K_1\sum_{z}\int_0^{\theta_1}\left(\int_0^{\theta_2}\left(\int_0^{+\infty}p_{\theta_2-r+s}(z_3, z)ds\right)\tilde{m}(dr, dz)\right)d\theta_2\\
&\leq K_1\sum_{z}\int_0^{\theta_1}\left(\int_0^{\theta_2}\left(\int_0^{+\infty}p_{\theta_2-r+s}(\textbf{0}, \textbf{0})ds\right)\tilde{m}(dr, dz)\right)d\theta_2\\
&=O(1)\sum_{z}\int_0^{\theta_1}\left(\int_0^{\theta_2}\frac{1}{\sqrt{\theta_2-r}}\tilde{m}(dr, dz)\right)d\theta_2\\
&=O(1)\sum_{z}\int_0^{\theta_1}\left(\int_r^{\theta_1}\frac{1}{\sqrt{\theta_2-r}}d\theta_2\right)\tilde{m}(dr, dz)\\
&=O(1)\sum_{z}\int_0^{\theta_1}\sqrt{\theta_1-r}\tilde{m}(dr, dz)\\
&\leq O(1)\sqrt{\theta_1}\tilde{\mathbb{P}}(\Gamma_{12}<\theta_2)\leq O(1)\sqrt{\theta_1}\Phi(z_1-z_2).
\end{align*}
For $(z_1, z_2, z_3, z_4)\in D^N(\epsilon)$, $\Phi(z_1-z_2)\leq \sup_{z:\|z\|_1\geq (\|u\|_1-2\epsilon)\sqrt{N}}\Phi(z)$. Hence,
\begin{align*}
&\sup_{(z_1, z_2, z_3, z_4)\in D^N(\epsilon)}\int_0^{tN}d\theta_1\int_0^{\theta_1}\tilde{\mathbb{P}}\left(\Gamma_{12}<\theta_2<\Gamma_{13}<+\infty\right)d\theta_2\\
&=O(1)N^{\frac{3}{2}}\sup_{z:\|z\|_1\geq (\|u\|_1-2\epsilon)\sqrt{N}}\Phi(z).
\end{align*}
According to the transience of the simple random walk on $\mathbb{Z}^3$,
\[
\lim_{N\rightarrow+\infty}\sup_{z:\|z\|_1\geq (\|u\|_1-2\epsilon)\sqrt{N}}\Phi(z)=0
\]
and hence
\begin{align}\label{equ A.6 5}
\lim_{N\rightarrow+\infty}\frac{1}{N^{\frac{3}{2}}}
\sup_{(z_1, z_2, z_3, z_4)\in D^N(\epsilon)}\int_0^{tN}d\theta_1\int_0^{\theta_1}\tilde{\mathbb{P}}\left(\Gamma_{12}<\theta_2<\Gamma_{13}<+\infty\right)d\theta_2=0.
\end{align}
According to a similar argument,
\begin{align*}
&\int_0^{\theta_1}\tilde{\mathbb{P}}\left(\theta_2\leq \Gamma_{12}<\Gamma_{13}<+\infty\right)d\theta_2\\
&=O(1)\sum_{z}\int_0^{+\infty}\left(\int_0^{\min\{r, \theta_1\}}\frac{1}{\sqrt{r-\theta_2}}d\theta_2\right)\tilde{m}(dr, dz)\\
&\leq O(1)\sqrt{\theta}\Phi(z_1-z_2)
\end{align*}
and hence \eqref{equ A.6 5} still holds when $\Gamma_{12}<\theta_2<\Gamma_{13}<+\infty$ is replaced by $\theta_2\leq \Gamma_{12}<\Gamma_{13}<+\infty$.
In conclusion, the $\mathcal{N}_{13}$ part of Lemma \ref{lemma appendix A.6 1} holds. Other $\mathcal{N}_{ij}$ parts and $\mathcal{T}_{ji}$ parts follow from arguments similar to that leading to the $\mathcal{N}_{13}$ part and the proof is complete.
\qed

Now we prove \eqref{equ 4.8}.
\proof[Proof of \eqref{equ 4.8}]
For given $r,\theta_2\geq 0$ and any $i\in \{1, 2\}$ and $j\in \{3, 4\}$, we denote by $\Xi_{i, j}$ the event that $\Gamma_{ij}<r+\theta_2$ and denote by $\tilde{\Xi}_{i, j}$ the event that $\Gamma_{ij}<r+\theta_2$ and $\Gamma_{lk}=+\infty$ for any different $l, k\in \{1, 2, 3, 4\}$ such that $(l, k)\neq (i, j)$. We write $\Xi_{i, j}$ as $\Xi_{i, j}(z_1, z_2, z_3, z_4, \theta_2, r)$ when we need to emphasize the dependence of $\Xi_{i, j}$ on $z_1, z_2, z_3, z_4, \theta_2, r$. According to the above definitions, for $i\in \{1, 2\}$ and $j\in \{3, 4\}$
\[
\left(\Xi_{i, j}\setminus \tilde{\Xi}_{i, j}\right)
\subseteq \mathcal{N}_{ij}\bigcup \mathcal{I}_i\bigcup \mathcal{K}_j\bigcup \mathcal{H}_{ij}\bigcup\mathcal{T}_{ji}.
\]
Then, by Lemma \ref{lemma appendix A.6 1},
\begin{equation}\label{equ A.6 6}
\lim_{N\rightarrow+\infty}\frac{1}{N^{\frac{3}{2}}}
\sup_{(z_1, z_2, z_3, z_4)\in D^N(\epsilon), r}\int_0^{tN}d\theta_1\int_0^{\theta_1}\tilde{\mathbb{P}}\left(\left(\Xi_{i, j}\setminus \tilde{\Xi}_{i, j}\right)(z_1, z_2, z_3, z_4, \theta_2, r)\right)d\theta_2=0
\end{equation}
for any $i\in \{1, 2\}$ and $j\in \{3, 4\}$.

Let
\[
\tilde{\Gamma}=\inf\left\{t:~t\geq \theta_2, \left\{\mathcal{X}_{t, 1}^{z_1, 0}, \mathcal{X}_{t, 2}^{z_2, 0}\right\}\bigcap \left\{\tilde{\mathcal{X}}_{t, 1}^{z_3, \theta_2}, \tilde{\mathcal{X}}_{t, 2}^{z_4, \theta_2}\right\}\neq \emptyset\right\}
\]
be defined as above, then
\[
\{\tilde{\Gamma}\leq r+\theta_2\}=\bigcup_{i=1}^2\bigcup_{j=3}^4 \Xi_{i, j}
\]
according to our equivalent definitions of the coalescing random walk. For simplicity, we denote by $\tilde{F}$ the event $\{\tilde{\Gamma}\leq r+\theta_2\}$.
We further define
\[
\hat{F}=\bigcup_{i=1}^2\bigcup_{j=3}^4 \tilde{\Xi}_{i, j}.
\]
Then, by \eqref{equ A.6 6},
\begin{equation}\label{equ A.6 7}
\lim_{N\rightarrow+\infty}\frac{1}{N^{\frac{3}{2}}}
\sup_{(z_1, z_2, z_3, z_4)\in D^N(\epsilon), r}\int_0^{tN}d\theta_1\int_0^{\theta_1}\tilde{\mathbb{P}}\left(\left(\tilde{F}\setminus \hat{F}\right)(z_1, z_2, z_3, z_4, \theta_2, r)\right)d\theta_2=0.
\end{equation}
By \eqref{equ A.6 1 two} and our equivalent definitions,
\begin{equation}\label{equ A.6 9}
{\rm Cov}_{\mu_p}\left(\eta_{\theta_2+r}(z_1)\eta_{\theta_2+r}(z_2), \eta_r(z_3)\eta_r(z_4)\right)
=\tilde{\mathbb{E}}\left(\hat{Y}1_{\tilde{F}}\right),
\end{equation}
where
\begin{align*}
\hat{Y}=&\mathbb{E}_{\mu_p}\left(\prod_{j=1}^2\eta_0(\mathcal{X}_{r+\theta_2, j}^{z_j, 0})\prod_{j=3}^4\eta_0(\mathcal{X}_{r+\theta_2, j}^{z_j, \theta_2})\right)\\
&-\mathbb{E}_{\mu_p}\left(\prod_{j=1}^2\eta_0(\mathcal{X}_{r+\theta_2, j}^{z_j, 0})\right)\mathbb{E}_{\mu_p}\left(\prod_{j=3}^4\eta_0(\tilde{\mathcal{X}}_{r+\theta_2, j-2}^{z_j, \theta_2})\right).
\end{align*}
Note that, to obtain \eqref{equ A.6 9}, we utilize the facts that $\mu_p$ is a product measure and
\[
\left(\tilde{\mathcal{X}}_{r+\theta_2, 1}^{z_3, \theta_2}, \tilde{\mathcal{X}}_{r+\theta_2, 2}^{z_4, \theta_2}\right)
=\left(\mathcal{X}_{r+\theta_2, 3}^{z_3, \theta_2}, \mathcal{X}_{r+\theta_2, 4}^{z_4, \theta_2}\right)
\]
when $\tilde{\Gamma}\geq r+\theta_2$. On the event $\hat{F}$, according to the definition of $\tilde{\Xi}_{i, j}$,
\[
\hat{Y}=p^3-p^4.
\]
Then, according to \eqref{equ A.6 6}-\eqref{equ A.6 9} and the fact that $\{\tilde{\Xi}_{i, j}\}_{i\in \{1, 2\}, j\in \{3, 4\}}$ are disjoint with each other, we have
\begin{align*}
{\rm Cov}_{\mu_p}\left(\eta_{\theta_2+r}(z_1)\eta_{\theta_2+r}(z_2), \eta_r(z_3)\eta_r(z_4)\right)
=(p^3-p^4)\sum_{i=1}^2\sum_{j=3}^4\tilde{\mathbb{P}}(\Gamma_{ij}\leq r+\theta_2)+\varepsilon_2
\end{align*}
for any $(z_1, z_2, z_3, z_4)\in D^N(\epsilon)$, where $\varepsilon_2=\varepsilon_2(z_1, z_2, z_3, z_4, \theta_2, r)$ such that
\[
\lim_{N\rightarrow+\infty}\frac{1}{N^{\frac{3}{2}}}
\sup_{(z_1, z_2, z_3, z_4)\in D^N(\epsilon), r}\int_0^{tN}d\theta_1\int_0^{\theta_1}|\varepsilon_2(z_1, z_2, z_3, z_4, \theta_2, r)|d\theta_2=0.
\]
Consequently,
\begin{align}\label{equ A.6 10}
&{\rm Cov}_{\nu_p}\left(\eta_{\theta_2}(z_1)\eta_{\theta_2}(z_2), \eta_0(z_3)\eta_0(z_4)\right)
=(p^3-p^4)\sum_{i=1}^2\sum_{j=3}^4\tilde{\mathbb{P}}(\Gamma_{ij}<+\infty)+\varepsilon_3\notag\\
&=K_1(p^3-p^4)\sum_{i=1}^2\sum_{j=3}^4\int_0^{+\infty}p_{s+\theta_2}(z_i, z_j)ds+\varepsilon_3
\end{align}
for any $(z_1, z_2, z_3, z_4)\in D^N(\epsilon)$, where $\varepsilon_3=\varepsilon_3(z_1, z_2, z_3, z_4, \theta_2)$ such that
\begin{align*}
\lim_{N\rightarrow+\infty}\frac{1}{N^{\frac{3}{2}}}
\sup_{(z_1, z_2, z_3, z_4)\in D^N(\epsilon)}\int_0^{tN}d\theta_1\int_0^{\theta_1}|\varepsilon_3(z_1, z_2, z_3, z_4, \theta_2)|d\theta_2=0.
\end{align*}
According to the local central limit theorem of the simple random walk on $\mathbb{Z}^3$, for any $M>0$ and $t_1>t_2>0$,
\begin{equation}\label{equ A.6 11}
\lim_{N\rightarrow+\infty}\sup_{z_1, z_2, r}\left|N^{\frac{3}{2}}p_{rN}(z_1, z_2)-f_{2r}^3\left(\frac{z_2-z_1}{\sqrt{N}}\right)\right|=0,
\end{equation}
where the supremum in \eqref{equ A.6 11} is over all $(z_1, z_2, r)$ such that $\|z_1\|_1\leq M\sqrt{N}, \|z_2\|_1\leq M\sqrt{N}$ and $t_2\leq r\leq t_1$. Equation \eqref{equ 4.8} follows from \eqref{equ A.6 1}, \eqref{equ A.6 10} and \eqref{equ A.6 11}.
\qed

\quad

\textbf{Acknowledgments.}
The author is grateful to financial
supports from the National Natural Science Foundation of China with grant number 12371142.

{}
\end{document}